\documentclass[10pt,b5paper]{article}
\usepackage[T1]{fontenc}
\usepackage[utf8]{inputenc}
\usepackage[english]{babel}
\usepackage{amsmath,amssymb,amsthm}
\usepackage{lmodern}
\usepackage{geometry}
\usepackage{graphicx}
\usepackage{url}
\usepackage{titlesec}
\titleformat{\paragraph}[runin]{\normalfont\itshape}{}{0pt}{}[.]
\titlespacing*{\paragraph}{0pt}{0.7em}{0.5em}
\newtheorem{theorem}{Theorem}[section]
\newtheorem{lemma}[theorem]{Lemma}
\newtheorem{proposition}[theorem]{Proposition}
\newtheorem{corollary}[theorem]{Corollary}
\newtheorem{conjecture}[theorem]{Conjecture}
\theoremstyle{definition}
\newtheorem{definition}[theorem]{Definition}
\theoremstyle{remark}
\newtheorem{remark}[theorem]{Remark}
\numberwithin{equation}{section}

\DeclareMathOperator{\E}{E}
\DeclareMathOperator{\Var}{Var}
\DeclareMathOperator{\Cov}{Cov}
\DeclareMathOperator{\Corr}{Corr}
\DeclareMathOperator{\Bin}{Bin}
\DeclareMathOperator{\Mult}{Mult}
\DeclareMathOperator{\sgn}{sgn}
\DeclareMathOperator{\fp}{f.p.}
\DeclareMathOperator{\pv}{p.v.}
\newcommand{\Z}{\mathbb{Z}}
\newcommand{\R}{\mathbb{R}}
\newcommand{\C}{\mathbb{C}}
\newcommand{\hadj}{h_{\mathrm{adj}}}
\newcommand{\alatt}{a_{\mathrm{latt}}}
\newcommand{\asm}{a_{\mathrm{sm}}}

\title{Discrete total variation under the multinomial, Part~2:\\
variance under the uniform null}
\author{
Nikola Bani\'c\\
Gideon Brothers, 10000 Zagreb, Croatia\\
\texttt{nbanic@gmail.com}
\and
Neven Elezovi\'c\\
Department of Applied Mathematics,\\
Faculty of Electrical Engineering and Computing,\\
University of Zagreb, 10000 Zagreb, Croatia\\
\texttt{neven.elezovic@fer.hr}
}
\date{\today}

\begin{document}
\maketitle

\begin{abstract}
For the counts $X=(X_1,\dots,X_n)\sim\Mult(N;\tfrac1n,\dots,\tfrac1n)$ of $N$ observations in $n$
ordered bins we study the variance of the discrete total variation
$d=\sum_{i=1}^{n-1}|X_{i+1}-X_i|$. An exact decomposition splits $\Var(d)$ into per-edge,
adjacent-pair and disjoint-pair terms and yields, for fixed $n$ as $N\to\infty$,
\[
	\frac{\Var(d)}{N}=\alpha(n)-\frac{\beta(n)}{N}+\frac{\gamma(n)}{N^2}+o(N^{-2}).
\]
For every $n\ge2$, $N\alpha(n)$ is the classical normal-theory variance (Hoel, Kamat) of the sum of
absolute successive differences of $n$ independent $N(0,N/n)$ variables; the correction $\beta(n)$
is obtained in closed form for $n\ge3$ (the case $n=2$ is treated separately). Two disjoint edge
differences $D_1=X_2-X_1$ and $D_3=X_4-X_3$ are exactly uncorrelated and asymptotically independent,
yet $\Cov(|D_1|,|D_3|)=-1/(\pi n)+O(1/N)$: a correlation of relative order $1/N$, invisible in the
Gaussian limit and absent under Poissonized sampling, whose constant is proportional to the exact
joint fourth cumulant $\kappa_{2,2}(D_1,D_3)=-4N/n^2$ and is proved here on the lattice. The adjacent
pair contributes the lattice-Edgeworth constant $a_*=-2\sqrt3/(9\pi)$, of which one quarter is a
smooth Edgeworth term and three quarters a lattice-corner term; the disjoint skeleton $h(m)$ is
closed through a Bessel generating function; and $\gamma(n)$ ($n\ge5$) is expressed through a single
triangular-lattice constant whose closed form we conjecture. Together with the mean of Part~1, the
variance supplies the two moments used by moment-matched approximations to the null law of $d$.
This is Part~2 of a series; the cyclic statistic $d^{c}$ will be treated in Part~5.
\end{abstract}



\small

\tableofcontents

\normalsize

\newpage
\section{Introduction}

	The variance of the discrete total variation $d=\sum_{i=1}^{n-1}|X_{i+1}-X_i|$ of an ordered
	uniform multinomial sample is governed by three kinds of covariances between absolute edge
	differences. The least expected concerns two edges that share no bin: $D_1=X_2-X_1$ and
	$D_3=X_4-X_3$ are exactly uncorrelated and asymptotically independent, yet
	$\Cov(|D_1|,|D_3|)=-1/(\pi n)+O(1/N)$. This is a correlation of relative order $1/N$, invisible
	in the Gaussian limit, whose constant is proportional to the exact joint fourth cumulant
	$\kappa_{2,2}(D_1,D_3)=-4N/n^2$. Computing $\Var(d)$ to two asymptotic orders, and identifying
	this constant together with its adjacent-pair analogue $a_*$, is the object of the present paper.

	Introduced as a histogram statistic in~\cite{banic_elezovic_tvor}, $d$ is the
	$\ell_1$ norm of the discrete derivative of a histogram. Because it compares neighboring bins
	it depends on the ordering of the categories, to which Pearson's chi-square is invariant; in
	the empirical comparisons of~\cite{banic_elezovic_ct} this gives power gains concentrated on
	period-2 (alternating) departures $p_i=\tfrac1n+(-1)^i\varepsilon$ ($n$ even), chi-square
	being preferable against smooth and longer-period departures. The variance of $d$ was left
	open in~\cite{banic_elezovic_tvor}, which gave $\Var(X_{i+1}-X_i)=2N/n$ under the uniform law
	and an exact double-sum formula for the mean (closed at $n=2$), but handled the dispersion of
	$d$ only empirically, through a $\sqrt N$ law fitted by simulation; the present paper supplies it.
	Part~1~\cite{banic_elezovic_p1} computes the expectation of $d$ under the uniform null;
	\emph{the present paper computes its variance}. The exact null law of $d$ is computable by
	dynamic programming for $N$ up to a few hundred, and beyond that the comb
	test~\cite{banic_elezovic_ct} uses a gamma law matched to simulated moments; the closed-form
	mean (Part~1) and variance (this paper) supply those two moments analytically and are the
	inputs of any moment-based (gamma, normal or Edgeworth) approximation to the null law.

	The statistic sits in a classical lineage. It is the $\ell_1$ analogue of the
	mean-square successive difference $\sum_i(X_{i+1}-X_i)^2$~\cite{vonneumann_kent1941} --- von
	Neumann's ratio~\cite{vonneumann1941}, and the Durbin--Watson statistic in the regression
	setting~\cite{durbin_watson1950} --- with the absolute value in place of the square; its
	direct ancestor is the mean \emph{absolute} successive difference of a normal sample, whose
	moments were obtained by Hoel~\cite{hoel1946} and Kamat~\cite{kamat1953} and which reappears
	here as the Gaussian limit of $d$ (Theorem~\ref{thm:z2}). It also belongs among the
	order-sensitive goodness-of-fit tools for ordered categories (Neyman's smooth test and its
	components~\cite{neyman1937}, the cumulative chi-square~\cite{hirotsu1986}, runs
	tests~\cite{wald_wolfowitz1940}), whereas Pearson's $\chi^2$, being invariant under
	permutations of the categories, ignores the order.

	The variance is harder than the mean for a structural reason. Under the uniform law the
	edge variables $D_j=X_{j+1}-X_j$ are pairwise correlated only between neighbors, with
	the $n$-independent correlation $\rho=-\tfrac12$, and uncorrelated at any larger separation;
	yet $\Var(d)$ is built not from the $D_j$ but from the \emph{absolute} values
	$|D_j|$, and the absolute value does not respect this clean structure. Three
	structurally distinct pieces enter: the per-edge variance $\Var(|D_1|)$, the adjacent
	covariance $\Cov(|D_1|,|D_2|)$, and the disjoint covariance $\Cov(|D_1|,|D_3|)$. The
	last is the non-Gaussian point. Distant edges are asymptotically independent, so their
	absolute values decorrelate at the Gaussian order; at constant order, however,
	$\Cov(|D_1|,|D_3|)$ tends to the non-zero value $-1/(\pi n)$ --- a relative-order-$1/N$
	effect of the fixed sample size, proportional to the joint fourth cumulant $\kappa_{2,2}$,
	which the lattice analysis of \S\ref{ssec:Cov-disj-asymp-proof} pins down, together with its
	$1/N$ correction (Lemma~\ref{lem:r1-disjoint}).

	Throughout the paper we write $\sigma^2(n,N):=\Var(d)/N$ for the normalized variance,
	and every asymptotic statement is in the regime of \emph{fixed} $n$ and $N\to\infty$.
	The resulting expansion is
	\begin{equation}\label{eq:intro-expansion}
		\sigma^2(n,N)=\frac{\Var(d)}{N}
		=\alpha(n)-\frac{\beta(n)}{N}+\frac{\gamma(n)}{N^2}+o(N^{-2}),
	\end{equation}
	where $\alpha(n)$ is the leading normalized variance, $\beta(n)$ is the
	constant-order correction in $\Var(d)$ with the sign convention shown in
	\eqref{eq:intro-expansion}, and $\gamma(n)$ is the next coefficient.  Assembling
	the three pieces yields $\alpha(n)$ for every $n\ge2$ --- $N\alpha(n)$ being the classical
	Hoel--Kamat variance of the sum of absolute successive differences of $n$ i.i.d.\ $N(0,N/n)$
	variables --- and $\beta(n)$ in closed form for all $n\ge3$ (only $n=2$ is exceptional);
	the second-order coefficient $\gamma(n)$ is given for $n\ge5$, expressed through a
	single adjacent triangular-lattice constant $a_\triangle$, and only the closed
	algebraic evaluation of that constant is conjectural. Since $\beta(n)\sim\beta_1n$ and
	$\gamma(n)\sim\gamma_2n^2$, the expansion is effectively one in $n/N$: it is accurate for
	$N\gg n$, and $\gamma(n)$ measures the first departure from that regime (the table at the
	end of \S\ref{ssec:beta-correction} quantifies this). The boundary cases $n\le4$,
	with their residue and parity effects, are treated separately in
	\S\ref{sec:small-n}. Numerically the constant is
	$a_\triangle=0.002738875595574\ldots$, with conjectured closed form
	$a_\triangle=\tfrac1{32}-\tfrac{4\sqrt3}{243}$ (confirmed to more than $40$ significant
	digits, Definition~\ref{def:atriangle}). The adjacent
	covariance contributes a constant $a_*=-2\sqrt3/(9\pi)$, proved here with a uniform
	remainder, and the disjoint skeleton is identified in closed form. Thus the present
	paper gives the fixed-$n$ variance expansion of the linear statistic under the
	uniform null; the large-bin questions will be taken up in Part~5~\cite{banic_elezovic_p5},
	for which the exact linear--cyclic identity of Lemma~\ref{lem:linear-cyclic} is the
	natural tool.

	\medskip\noindent\textbf{Main results.} Three statements carry the paper. For $n\ge3$ the
	constant-order correction is closed (Theorem~\ref{thm:beta-closed}):
	\[
		\beta(n)=-\frac{(n-2)(n+1)}{2\pi n}+\frac{4\sqrt3\,(n-2)}{9\pi}.
	\]
	Its technically hardest ingredient is the adjacent-pair constant: conditionally on
	$X_1+X_2+X_3=\ell$, the pair $(D_1,D_2)$ of adjacent edges is a sum of $\ell$ i.i.d.\ increments
	uniform on $\{(-1,0),(1,-1),(0,1)\}$, a lattice walk supported on one coset of an index-$3$
	sublattice of $\Z^2$, and (Theorem~\ref{thm:astar})
	\[
	\begin{gathered}
		\hadj(\ell):=\E\bigl[\,|D_1|\,|D_2|\,\big|\,X_1+X_2+X_3=\ell\,\bigr]=c_L\,\ell+a_*+O(\ell^{-1}),\\
		c_L=\frac{2\sqrt3}{3\pi}+\frac19,\qquad
		a_*=\asm+\alatt=-\frac{\sqrt3}{18\pi}-\frac{\sqrt3}{6\pi},
	\end{gathered}
	\]
	where the smooth part $\asm$ is the second-order Edgeworth correction and the lattice part
	$\alatt$ is produced by the constant term $-\tfrac16$ of the periodic kernel
	$-\tfrac12\csc^2(\theta/2)$ (the Fourier multiplier of $|\cdot|$) on the two lines where
	$|D_1|\,|D_2|$ has a kink. And for two edges sharing no bin (Theorem~\ref{thm:Cov-disj}
	and \S\ref{ssec:Cov-disj-asymp-proof}),
	\[
	\begin{aligned}
		\Cov(|D_1|,|D_3|)&=-\frac1{\pi n}+O(1/N)&&(n\ge4),\\
		\Cov(|D_1|,|D_3|)&=\kappa_{2,2}\,\Pr(W_{N-1}=0)+O(1/N)&&(n\ge5),
	\end{aligned}
	\]
	where $\kappa_{2,2}=-4N/n^2$ is the exact joint fourth cumulant of the pair and $W_{N-1}$ the
	sum of $N-1$ one-trial increments of $(D_1,D_3)$, a lattice walk on $\Z^2$ that is aperiodic
	exactly when $n\ge5$.

	This paper belongs to a series on the exact and asymptotic moment theory of discrete total
	variation under the multinomial model, organized by \emph{which} moment is computed and
	\emph{under which} null. Two papers treat the uniform null --- the expectation
	(Part~1~\cite{banic_elezovic_p1}), with its closed form through central trinomial
	coefficients and the bin-independent limit $\E(d^{c}/\sqrt{Nn})\to2/\sqrt\pi$, and
	the variance, the present paper. Part~1 is a genuine prerequisite here, not merely a
	companion: its per-edge mean expansion (imported as Lemma~\ref{lem:imported-mean})
	enters the per-edge variance and, materially, the second-order coefficients of
	Theorem~\ref{thm:gamma-linear}. Two more parts will carry the same questions to an
	arbitrary reference law $\mathbf r$, for the reference-centered statistic
	$d_{\mathbf r}=\sum_{i=1}^{n-1}|(X_{i+1}-Nr_{i+1})-(X_i-Nr_i)|$ in place of
	$d=d_{\mathbf u}$, $\mathbf u:=(\tfrac1n,\dots,\tfrac1n)$
	(Parts~3--4~\cite{banic_elezovic_p3,banic_elezovic_p4}), where the
	absolute value softens into a folded-normal theory governed by the Nabeya identity. A
	final part will close the histogram into a ring and treat the cyclic statistic $d^{c}$
	(Part~5~\cite{banic_elezovic_p5}); the exact identity
	$\Var(d)=(1-\tfrac2n)\Var(d^{c})+\Var(|D_1|)$ proved below
	(Lemma~\ref{lem:linear-cyclic}) ties the two statistics order by order and already
	shows that the cyclic correction is exactly linear, $\beta_c(n)=\beta_1n$ (see the paragraph
	\emph{Cyclic consequence} after Theorem~\ref{thm:beta-closed}). The present paper determines the first two
	coefficients of the fixed-$n$ expansion of $\Var(d)$ under the uniform null in closed
	form, and the third up to one numerically evaluated constant.

	A word on how these moments are used, and in what regime. Two limits must be kept
	apart. In the regime of the present paper --- fixed $n$, $N\to\infty$ --- the number
	of summands is fixed, and $d/\sqrt N$ converges in distribution to
	$\sum_{j=1}^{n-1}|G_j|$, where $(G_1,\dots,G_{n-1})$ is the centered Gaussian vector
	with the edge covariances of Lemma~\ref{lem:cov-Dj} divided by $N$; this limit law is a functional
	of finitely many folded Gaussians and is \emph{not} normal in general. Approximate
	normality of $d$ is a \emph{large-$n$} phenomenon: the folded summands $|G_j|$ are
	$1$-dependent, so a central limit theorem applies to their number, $n$, not to $N$.
	The null distribution has two regimes --- the exact fixed-$n$ law and its expansions, where
	the $\beta(n)/N$ term earns its keep at moderate $N$, and the large-$n$ normal approximation;
	in both, the mean and the variance computed here are the inputs of moment-based calibrations
	of $d$, and the exact finite-$(n,N)$ variance of \eqref{eq:Var-d-3comp} (with
	Corollary~\ref{cor:exact-disjoint-moment}) is available in either.

	For reference, the constants that recur through the proof are collected here.
	\[
	\begin{array}{c|c|l}
	\text{symbol} & \text{value} & \text{role}\\ \hline
	A & 1-2/\pi & \text{single-edge Gaussian variance}\\
	B & (\sqrt3-2)/\pi+1/6 & \text{adjacent Gaussian covariance}\\
	C_1 & 2A+8B & \text{linear--cyclic leading gap}\\
	c_L & 2\sqrt3/(3\pi)+1/9 & \text{leading adjacent conditional slope}\\
	\asm & -\sqrt3/(18\pi) & \text{smooth part of }a_*\\
	\alatt & -\sqrt3/(6\pi) & \text{lattice-corner part of }a_*\\
	a_* & -2\sqrt3/(9\pi) & \text{adjacent constant}\\
	a_\triangle & 0.002738875595574\ldots & \text{second-order adjacent constant}\\
	\beta_1 & (8\sqrt3-9)/(18\pi) & \text{linear coefficient of }\beta(n)\\
	\gamma_2 & -0.0280979294715\ldots & \text{quadratic coefficient of }\gamma(n)
	\end{array}
	\]

\section{Preliminaries: the multinomial and the edge variables}
	\label{sec:multinomial-basics}

	We collect the few facts about the multinomial law on which the variance
	computations rest. Two carry the whole analysis: an edge $D_j=X_{j+1}-X_j$
	depends only on a two-bin marginal, and that marginal factors into independent
	binomial steps.

\subsection{Definition and marginals}
	\label{ssec:mult-def}

	Let $X=(X_1,\dots,X_n)\sim\Mult(N;p_1,\dots,p_n)$, $\sum_i p_i=1$, count the
	observations falling into $n$ ordered bins in $N$ independent trials, with mass
	function
	\begin{equation}\label{eq:mult-pmf}
		\Pr(X_1=k_1,\dots,X_n=k_n)=\frac{N!}{k_1!\cdots k_n!}\,p_1^{k_1}\cdots p_n^{k_n},
		\qquad \sum_i k_i=N.
	\end{equation}
	A single bin is binomial, $X_i\sim\Bin(N,p_i)$; more generally the marginal of any
	subset of bins is again multinomial, the unused bins pooled into one remainder.

\begin{proposition}[$k$-bin marginal]\label{prop:k-marginal}
	For $S\subseteq\{1,\dots,n\}$,
	\[
		(X_i)_{i\in S}\sim\Mult\bigl(N;\,(p_i)_{i\in S},\;1-\textstyle\sum_{i\in S}p_i\bigr),
	\]
	a multinomial whose last bin, of probability $1-\sum_{i\in S}p_i$, pools the
	observations outside $S$. (If $\sum_{i\in S}p_i=1$, the remainder bin has probability $0$
	and count $0$ a.s.)
\end{proposition}

\begin{proof}
	Sum the mass function \eqref{eq:mult-pmf} over the counts outside $S$; by the
	multinomial theorem the inner sum equals $(1-\sum_{i\in S}p_i)^{m}/m!$ with
	$m=N-\sum_{i\in S}k_i$, which is the stated law.
\end{proof}

\subsection{Factorization of a two-bin marginal}
	\label{ssec:2bin-factorization}

	The summands of the total variation are entangled with the constraint
	$\sum_iX_i=N$. The following factorization removes that entanglement one edge at a
	time, reducing a single edge to a two-stage binomial experiment: first decide how
	many observations fall into the pair $\{j,j+1\}$, then split them between the two.
	Write $K_j:=X_j+X_{j+1}$ for the pair total.

\begin{proposition}[Factorization of the two-bin marginal]\label{prop:2bin-factor}
	If $p_j+p_{j+1}>0$ then $K_j\sim\Bin(N,p_j+p_{j+1})$, and conditionally
	\[
		X_{j+1}\mid K_j=k\;\sim\;\Bin(k,q_j),\qquad q_j:=\frac{p_{j+1}}{p_j+p_{j+1}}.
	\]
\end{proposition}

\begin{proof}
	$K_j$ counts the trials landing in $\{j,j+1\}$, so $K_j\sim\Bin(N,p_j+p_{j+1})$.
	Given $K_j=k$, Proposition~\ref{prop:k-marginal} and \eqref{eq:mult-pmf} reduce
	$\Pr(X_{j+1}=m\mid K_j=k)$ to $\binom{k}{m}q_j^m(1-q_j)^{k-m}$, the law of
	$\Bin(k,q_j)$.
\end{proof}

	Hence, conditionally on $K_j=k$, the edge $D_j=2X_{j+1}-k$ is a shifted binomial;
	under the \emph{uniform} law ($p_j=p_{j+1}$) one has $q_j=\tfrac12$, so $D_j\mid
	K_j=k$ is the position of a symmetric $\pm1$ walk of $k$ steps --- the
	identification on which every computation below rests.

\subsection{Edge variables and their covariance structure}
	\label{ssec:edge-vars}

	The total variation is assembled from the differences between neighboring bins, so
	these \emph{edge variables} --- not the counts $X_i$ --- are the true objects of study.
	Before any absolute value is taken their joint second-order behavior is completely
	elementary; recording it here, once, fixes notation and previews the structural
	difficulty that the variance --- unlike the mean --- will have to confront.

\begin{definition}[Edge variables]\label{def:Dj}
	The \emph{edge variables} are $D_j:=X_{j+1}-X_j$, $j=1,\dots,n-1$, together with
	the wrap edge $D_n:=X_1-X_n$ (the conventions $X_{n+1}:=X_1$, $p_{n+1}:=p_1$); the linear and
	cyclic statistics are $d:=\sum_{j=1}^{n-1}|D_j|$ and $d^{c}:=\sum_{j=1}^{n}|D_j|$.
	The centered
	versions are $\tilde D_j:=D_j-\E D_j=D_j-N(p_{j+1}-p_j)$. Under the uniform law
	$\E D_j=0$, so $\tilde D_j=D_j$.
\end{definition}

\begin{lemma}[Covariances of the raw differences]\label{lem:cov-Dj}
	For the uniform law and the linear edges $D_j=X_{j+1}-X_j$, $1\le j,k\le n-1$,
	\[
		\Var(D_j)=\frac{2N}{n},\qquad \Cov(D_j,D_{j+1})=-\frac{N}{n},\qquad
		\Cov(D_j,D_k)=0\ \ (|k-j|\ge2).
	\]
	In particular the adjacent correlation is
	$\rho=\Cov(D_j,D_{j+1})/\Var(D_j)=-\tfrac12$ ($n\ge3$), while disjoint edges are
	uncorrelated. For the wrap edge $D_n$ of Definition~\ref{def:Dj} the same values hold
	with the cyclic distance: $\Cov(D_1,D_n)=\Cov(D_{n-1},D_n)=-N/n$ for $n\ge3$, while
	$D_2=-D_1$ at $n=2$.
\end{lemma}

\begin{proof}
	From the multinomial moments $\Var(X_a)=N\tfrac1n(1-\tfrac1n)$ and
	$\Cov(X_a,X_b)=-N/n^2$ for $a\neq b$. Then
	$\Var(D_j)=\Var X_{j+1}+\Var X_j-2\Cov(X_{j+1},X_j)
	=2(N/n-N/n^2)+2N/n^2=2N/n$. For the adjacent pair, bilinearity gives
	\[
	\begin{aligned}
		\Cov(D_j,D_{j+1})
		&=\Cov(X_{j+1},X_{j+2})-\Var(X_{j+1})\\
		&\quad-\Cov(X_j,X_{j+2})+\Cov(X_j,X_{j+1})\\
		&=-\frac{N}{n^2}-\left(\frac Nn-\frac{N}{n^2}\right)
		+\frac{N}{n^2}-\frac{N}{n^2}
		=-\frac Nn.
	\end{aligned}
	\]
	For a gap $\ge2$ all four indices are distinct, the four off-diagonal covariances
	are each $-N/n^2$ with signs $+,-,-,+$, and they cancel. (The first identity is
	Equation~(13) of~\cite{banic_elezovic_tvor}.)
\end{proof}

	These are exactly the covariances of the successive differences $Y_{j+1}-Y_j$ of $n$
	i.i.d.\ $N(0,N/n)$ variables; the Gaussian limit of $d/\sqrt N$ is therefore the sum of
	absolute successive differences of an i.i.d.\ normal sample, whose variance is classical
	(Hoel~\cite{hoel1946}, Kamat~\cite{kamat1953}; see Theorem~\ref{thm:z2}).
	Two features already foreshadow the analysis. \emph{Adjacent} edges share the bin
	$X_{j+1}$, which forces the fixed negative correlation $\rho=-\tfrac12$;
	\emph{disjoint} edges share no bin and, under the uniform law, are exactly
	uncorrelated, tied only through the global constraint $\sum_iX_i=N$. The warning is
	that this is the covariance of the \emph{raw} differences: once the absolute values
	enter, disjoint edges no longer decorrelate, and that residual dependence ---
	invisible to the mean, a plain sum of single-edge expectations --- is exactly what
	makes the variance the harder object and the probabilistic point of this paper
	(\S\ref{ssec:Cov-disj-asymp-proof}); the technical heart is the adjacent constant $a_*$ of
	\S\S\ref{sec:adj-reduction}--\ref{sec:adj-smooth-lattice}.

\subsection{The local Edgeworth expansion}
	\label{ssec:edgeworth}

	Alongside the Nabeya identity of Lemma~\ref{lem:sheppard-mardia}, the
	second tool used throughout is the local Edgeworth expansion for sums of
	lattice vectors: it supplies the first non-Gaussian corrections that the central
	limit theorem cannot see. We record here the form on which the constant-order and
	second-order calculations rest. A word on its status: the expansion of the
	characteristic function with derivative control on a growing central window is
	classical (Bhattacharya--Rao~\cite[Ch.~2, \S9]{bhattacharya_rao}); for bounded
	lattice increments it is elementary, and we include a self-contained proof via
	analyticity of $\log\Phi$ near the origin and Cauchy estimates, in the bivariate
	case used in this paper (Lemma~\ref{lem:c2-remainder}: $\nu=2$, $s=3$; the same
	argument is run at $s=5$ in Lemma~\ref{lem:atriangle-existence}).

\begin{lemma}[Local expansion of the characteristic function]\label{thm:edgeworth}
	Let \(Y,Y_1,Y_2,\dots\) be i.i.d.\ random vectors in \(\Z^\nu\) with mean \(0\),
	nonsingular covariance \(C\), and \emph{bounded} support, and let
	\(\Phi(\theta)=\E\,e^{i\theta^\top Y}\) be the per-step
	characteristic function. For each fixed \(s\ge1\) and \(0<\epsilon<\tfrac12\), on the
	central window \(|u|\le\ell^{1/2-\epsilon}\),
	\begin{equation}\label{eq:edgeworth}
		\Phi(u/\sqrt\ell)^{\ell}
		=e^{-\frac12u^\top Cu}\Bigl[1+\sum_{k=1}^{s}\ell^{-k/2}P_k(u)\Bigr]+R_{\ell,s}(u),
	\end{equation}
	where \(P_k\) is a polynomial of degree at most \(3k\) and of the same parity as
	\(k\), with coefficients determined by the cumulants of \(Y\) of order at most
	\(k+2\); in particular \(P_1\) is odd and cubic, carried by the third cumulants.
	The remainder is uniformly small together with its low-order derivatives: there are
	\(\ell_0\), \(C'\) and \(M\), depending on \(s\), \(\epsilon\) and the law of \(Y\), such
	that for \(\ell\ge\ell_0\) and every multi-index \(a=(a_1,\dots,a_\nu)\) with
	\(0\le a_j\le2\),
	\[
		\sup_{|u|\le\ell^{1/2-\epsilon}}\bigl|\partial_u^{a}R_{\ell,s}(u)\bigr|
		\le C'\,\ell^{-(s+1)/2}\,(1+|u|^{M})\,e^{-\frac18u^\top Cu}.
	\]
\end{lemma}

	The origin expansion \eqref{eq:edgeworth} is the analytic expansion of
	\(\ell\log\Phi(u/\sqrt\ell)\) about \(0\), truncated and exponentiated; only its
	shape and the parity of the \(P_k\) are needed for the leading term; the constant-order
	term uses the explicit \(P_2\) and, for its \(O(\ell^{-1})\) error, the quantitative
	\(C^{2,2}\) remainder at \(s=3\), which also drives the second-order analysis.  As
	\eqref{eq:edgeworth} is a purely \emph{local} statement on
	the central window \(|u|\le\ell^{1/2-\epsilon}\) about the origin, it is valid
	regardless of any global periodicity of \(Y\): only the analyticity of
	\(\log\Phi\) near \(0\) and the nonsingularity of \(C\) are used, so it applies verbatim
	to the index-\(3\) adjacent increment, whose two non-origin maxima of \(|\Phi|\) are
	handled by separate saddle contributions.

\begin{remark}\label{rem:edgeworth-use}
	Two features of \eqref{eq:edgeworth} recur below. First, the parity of \(P_k\):
	paired against an even, reflection-symmetric kernel the odd-order terms
	\(P_1,P_3,\dots\) integrate to zero, so the first surviving correction is the
	\(\ell^{-1}\) term \(P_2\); the smooth part \(\asm\) of the adjacent constant and the
	second-order coefficient \(\gamma(n)\) are accordingly \emph{second-order} Edgeworth
	effects, while the remaining three quarters of \(a_*\), \(\alatt\), is a lattice-kernel
	effect of the same order (\S\ref{ssec:adj-assembly}).
	Second, the uniform control of the remainder in a \(C^{2,2}\) seminorm is what lets
	\eqref{eq:edgeworth} be paired with the singular finite-part kernels
	\(\fp(u_1^{-2})\otimes\fp(u_2^{-2})\) that carry the absolute values; the
	quantitative form at \(s=3\) needed for that pairing is established, for the
	increment at hand, in Lemma~\ref{lem:c2-remainder}.
\end{remark}

\section{Variance decomposition, leading term, and the per-edge correction}
	\label{sec:Var-uniform}

	The variance is harder than the mean for a structural reason: it must account for how
	the $n-1$ absolute differences fluctuate \emph{together}, and neighboring edges share
	a bin while distant ones are tied only through the global constraint $\sum_i X_i=N$.
	The spine of the analysis is the exact three-term decomposition \eqref{eq:Var-d-3comp},
	from which the leading term $N\alpha(n)$ and the closed-form constant correction
	$\beta(n)$ are extracted, the latter assembled from three structurally different
	pieces: a per-edge term, an adjacent-pair term whose constant $a_*$ is the technical
	heart of the paper, and a disjoint-pair term whose constant-order residual $-1/(\pi n)$
	the Gaussian limit cannot see.

	We build this layer by layer: the exact variance structure, the constants $A$, $B$ and
	$C_1$, the per-edge and disjoint contributions, the adjacent contribution, and finally
	their synthesis into $\beta(n)$ and the second-order coefficient $\gamma(n)$. One running
	theme is worth keeping in view: the constant $1/\pi$ recurs in every piece and
	$\sqrt3/\pi$ in the adjacent ones.

	For orientation, this section contains only the parts of the calculation that are
	either exact or Gaussian.  Nothing delicate about lattice corrections is needed yet.
	The exact identity \eqref{eq:Var-d-3comp} tells us \emph{which} three building
	blocks must be computed.  The central limit theorem tells us their leading
	order, hence the coefficient $\alpha(n)$.  The only constant-order information
	included here is the per-edge correction, because it follows immediately from the
	known expansion of the mean.  The two harder constant-order corrections, disjoint
	and adjacent, are deferred to \S\ref{sec:disjoint-pair}--\S\ref{sec:adj-smooth-lattice}.

\subsection*{Methodological map and proof dependencies}

	The paper uses several methods that enter at different levels of precision.  It is
	helpful to state their roles before the calculations begin.
	\begin{enumerate}
		\item The exact variance decomposition \eqref{eq:Var-d-3comp} separates the problem
		into a per-edge term, an adjacent-pair term, and a disjoint-pair term.  This is an
		exact combinatorial reduction; no asymptotics enter.  The decomposition is proved
		in \S\ref{ssec:Var-exact}, using only the exchangeability and pair counts of the
		edge variables.
		\item The multivariate central limit theorem together with the Nabeya
		formula determines the coefficient of order $N$ in all three parts.  It does not
		determine the constant-order corrections.  The normal absolute-moment identity is
		Lemma~\ref{lem:sheppard-mardia}, its Price-theorem derivation is recorded in
		\S\ref{ssec:price-sheppard-details}, and the leading coefficient
		$\alpha(n)$ is assembled in Theorem~\ref{thm:z2}.
		\item Conditioning on the number of observations in the relevant bins reduces exact
		finite-$N$ moments to symmetric random walks.  For disjoint pairs this produces the
		skeleton $h(m)$ and its Bessel generating function in
		\S\ref{ssec:h-closed}, Theorem~\ref{thm:h-closed} and
		Proposition~\ref{prop:h-Bessel}.  For adjacent pairs the analogous conditioning on
		$L=X_1+X_2+X_3$ is set up in \S\ref{ssec:adj-bridge}.
		\item The disjoint constant $-1/(\pi n)$ is proved from the exact skeleton $h(m)$
		(\S\ref{ssec:Cov-disj-asymp-proof}, every $n\ge4$).  An exact characteristic-function
		representation (Lemma~\ref{lem:periodic-abs-kernel} and Steps~1--4 of
		\S\ref{ssec:Cov-disj-asymp-proof}) explains it as a fourth-cumulant effect ---
		$\kappa_{2,2}$ times a lattice origin probability --- and gives a second proof for
		$n\ge5$: the periodic kernel carries the non-smooth weight, an algebraic factor cancels
		its singularity, and a two-dimensional lattice expansion supplies the origin
		probability.
		\item The adjacent constant $a_*$ requires both a smooth bivariate Edgeworth
		correction and a genuinely lattice correction.  The latter is caused by the
		index-$3$ support lattice and by the two kink lines of $|d_1||d_2|$.  The
		conditional reduction and the statement of $a_*$ are in
		\S\ref{ssec:adj-bridge}; the detailed proof is organized in
		\S\ref{sec:adj-reduction}--\S\ref{sec:adj-smooth-lattice}, with the final assembly in
		\S\ref{ssec:adj-assembly}.
		\item Only after the three constant-order pieces have been proved separately are they
		assembled into the closed form for $\beta(n)$ in
		\S\ref{sec:synthesis}, Theorem~\ref{thm:beta-closed}.  One order further, the
		same decomposition gives $\gamma(n)$ for $n\ge5$ in
		Theorem~\ref{thm:gamma-linear}, with the remaining explicit evaluation isolated in
		the triangular-lattice constant $a_\triangle$.
	\end{enumerate}
	Thus the paper deliberately moves from exact algebra, through Gaussian leading
	terms, to local lattice analysis.  Whenever a later step invokes an earlier method,
	the relevant dependency is stated explicitly.

\subsection{Exact structure of the variance $\Var(d)$}
	\label{ssec:Var-exact}

	By definition $d=\sum_{j=1}^{n-1}|D_j|$, where $D_j=X_{j+1}-X_j$, so
	\begin{equation}\label{eq:Var-d-decomp}
		\Var(d)=\sum_{j=1}^{n-1}\Var(|D_j|)+2\sum_{1\le j<k\le n-1}\Cov(|D_j|,|D_k|).
	\end{equation}
	The factor $2$ in the second sum is the usual one in the variance of a sum:
	each unordered pair $(j,k)$, $j<k$, contributes both
	$\Cov(|D_j|,|D_k|)$ and $\Cov(|D_k|,|D_j|)$.
	Under the uniform distribution all edges are structurally equivalent: the
	multinomial law is invariant under \emph{arbitrary} permutations of the bins,
	not merely order-preserving ones, so $\Var(|D_j|)=\Var(|D_1|)$ for
	all $j$, and the covariances depend only on the \emph{gap} $k-j$:
	\begin{itemize}
		\item $k-j=1$ (\emph{adjacent} pairs): share the variable $X_{j+1}$;
		\item $k-j\ge2$ (\emph{disjoint} pairs): have no common variable.
	\end{itemize}
	For a disjoint pair the equivalence uses this freedom in full: the joint law of
	$(X_j,X_{j+1},X_k,X_{k+1})$ involves four \emph{distinct} bins, and any
	relabeling of the bins carrying them to $(1,2,3,4)$ --- regardless of how the
	remaining bins are rearranged --- identifies
	$\Cov(|D_j|,|D_k|)$ with $\Cov(|D_1|,|D_3|)$, whatever the gap $k-j\ge2$.

	The covariance structure of the raw differences $D_j$ (without $|\cdot|$) was
	recorded as a warm-up in Lemma~\ref{lem:cov-Dj} of \S\ref{ssec:edge-vars}:
	$\Var(D_j)=2N/n$, $\Cov(D_j,D_{j+1})=-N/n$ (adjacent correlation $-\tfrac12$), and
	$\Cov(D_j,D_k)=0$ for disjoint pairs.

	Collecting equal terms in \eqref{eq:Var-d-decomp} via the equivalences
	$\Var(|D_j|)\equiv\Var(|D_1|)$, $\Cov(|D_j|,|D_{j+1}|)\equiv\Cov(|D_1|,|D_2|)$ and
	$\Cov(|D_j|,|D_k|)\equiv\Cov(|D_1|,|D_3|)$ for $k\ge j+2$,
	\begin{equation}\label{eq:Var-d-3comp}
		\Var(d)=(n-1)\,\Var(|D_1|)+2(n-2)\,\Cov(|D_1|,|D_2|)+(n-2)(n-3)\,\Cov(|D_1|,|D_3|).
	\end{equation}
	Here the coefficients are only counts.  There are $n-1$ edge variances.  There
	are $n-2$ adjacent pairs $(D_j,D_{j+1})$, and each is multiplied by $2$ by
	\eqref{eq:Var-d-decomp}.  Finally, among the $n-1$ edges there are
	$\binom{n-1}{2}$ unordered pairs in total; after removing the $n-2$ adjacent
	pairs, the number of disjoint pairs is
	\[
		\binom{n-1}{2}-(n-2)=\frac{(n-2)(n-3)}{2}.
	\]
	The factor $2$ in \eqref{eq:Var-d-decomp} turns this into
	$(n-2)(n-3)$, the coefficient in \eqref{eq:Var-d-3comp}.
	Note that at finite $N$, $\Cov(D_j,D_k)=0$ does \emph{not} imply $\Cov(|D_j|,|D_k|)=0$ for a
	disjoint pair: $(D_j,D_k)$ is only asymptotically bivariate normal, so the $O(N)$ term of the
	covariance vanishes but an $O(1)$ residual remains, carried by the joint fourth cumulant
	$\kappa_{2,2}$ (see \S\ref{ssec:Cov-disj-asymp-proof}). The three structurally distinct
	quantities of \eqref{eq:Var-d-3comp} are computed, in increasing order of difficulty, in
	\S\ref{ssec:Var-Dj}--\S\ref{sec:adj-reduction}.

	In skeleton form: with $K\sim\Bin(N,\tfrac2n)$, $L\sim\Bin(N,\tfrac3n)$,
	$M\sim\Bin(N,\tfrac4n)$ and the conditional skeletons $g$, $\hadj$, $h$ of
	\S\ref{ssec:h-closed} and \S\ref{ssec:adj-bridge}, so that $\E|D_1|=\E g(K)$,
	$\E|D_1||D_2|=\E\hadj(L)$ and $\E|D_1||D_3|=\E h(M)$, the decomposition
	\eqref{eq:Var-d-3comp} reads
	\[
	\begin{aligned}
		\Var(d)&=(n-1)\Bigl[\frac{2N}{n}-(\E g(K))^2\Bigr]
		+2(n-2)\bigl[\E\hadj(L)-(\E g(K))^2\bigr]\\
		&\quad+(n-2)(n-3)\bigl[\E h(M)-(\E g(K))^2\bigr],
	\end{aligned}
	\]
	an exact finite-$N$ formula that evaluates in $O(N)$ binomial sums once the skeletons are
	tabulated (a bracket whose multiplicity vanishes is omitted; at $n=2$ only the first
	survives, with $K=N$); every exact value quoted in this paper was computed from it.

\subsection{The constants $A$, $B$, $C_1$ and the leading variance}
	\label{ssec:ABCones}

	Here we name the three constants used repeatedly in what follows. They look
	innocuous --- $A=1-2/\pi$, $B=(\sqrt3-2)/\pi+1/6$, and their combination $C_1$ ---
	but each is a fingerprint. $A$ is what the absolute value does to the variance of a
	single normal; $B$ is what it does to a \emph{pair} correlated at the universal
	$\rho=-\tfrac12$ of adjacent edges; $C_1$ is the price of being linear rather than
	cyclic. Once they are isolated, the leading variance $\alpha(n)$ is merely a weighted
	count of edges and adjacent pairs.

	Both $A$ and $B$ are values of one classical identity, which we state once and use
	repeatedly.
	The identity says what happens when the two edge variables have already entered
	their Gaussian limit.  The exact multinomial variables are not Gaussian, but their
	leading $O(N)$ contribution to the variance is obtained by replacing them by this
	Gaussian limit.  The later sections measure precisely what this replacement misses
	at constant order.

\begin{lemma}[Nabeya 1951]\label{lem:sheppard-mardia}
	Let $(W_1,W_2)$ be centered bivariate normal with variances $\sigma_1^2,\sigma_2^2$
	and correlation $\rho\in(-1,1)$. Then
	\begin{equation}\label{eq:SM}
		\E\,|W_1|\,|W_2|=\frac{2\sigma_1\sigma_2}{\pi}\Big(\sqrt{1-\rho^2}+\rho\arcsin\rho\Big),
	\end{equation}
	and consequently
	$\Cov(|W_1|,|W_2|)=\tfrac{2\sigma_1\sigma_2}{\pi}\big(\sqrt{1-\rho^2}+\rho\arcsin\rho-1\big)$.
	In the equal-variance case $\sigma_1=\sigma_2=\sigma$ used throughout this paper
	the leading factor is $2\sigma^2/\pi$.
\end{lemma}

\begin{proof}
	By Price's theorem~\cite{price1958},
	\[
		\partial_\rho\,\E|W_1||W_2|=\sigma_1\sigma_2\,\E[\sgn W_1\,\sgn W_2]
		=\tfrac{2\sigma_1\sigma_2}{\pi}\arcsin\rho
	\]
	(the last equality is the standard quadrant probability). Integrating from $\rho=0$, where the factors are independent
	and $\E|W_1||W_2|=\sigma_1\sigma_2(\E|Z|)^2=\tfrac{2\sigma_1\sigma_2}{\pi}$, and using
	$\int\arcsin\rho\,d\rho=\rho\arcsin\rho+\sqrt{1-\rho^2}-1$, gives \eqref{eq:SM}; the
	covariance form follows on subtracting
	$\E|W_1|\,\E|W_2|=\tfrac{2\sigma_1\sigma_2}{\pi}$.  The use of Price's theorem for
	the non-smooth weight $|W_1||W_2|$ is justified by the mollification recorded in
	Remark~\ref{rem:regularization} below.  The resulting
	identity is the classical bivariate-normal absolute-moment formula, which goes
	back to Sheppard's arcsine formula~\cite{sheppard1899} and was given in this form
	by Nabeya~\cite{nabeya1951}; its value at $\rho=-\tfrac12$ is
	$\E|W_1||W_2|=\sigma^2\bigl(\tfrac{\sqrt3}{\pi}+\tfrac16\bigr)$, and for the successive
	differences $W_j=Y_{j+1}-Y_j$ of i.i.d.\ $N(0,s^2)$ variables ($\sigma^2=2s^2$) this is
	Hoel's adjacent product moment $s^2\bigl(\tfrac13+\tfrac{2\sqrt3}{\pi}\bigr)$~\cite[eq.~(8)]{hoel1946}.
\end{proof}

\subsection{Price's theorem behind the Nabeya identity}
	\label{ssec:price-sheppard-details}

	The proof of Lemma~\ref{lem:sheppard-mardia} is short because it uses a standard
	form of Price's theorem.  We record the exact form needed here, since it is often
	the fastest way to verify where the factor $\arcsin\rho$ enters.

	Let $(X,Y)$ be a centered bivariate normal vector with fixed variances
	\[
		\Var X=\sigma_1^2,\qquad \Var Y=\sigma_2^2,
	\]
	and covariance
	\[
		c=\Cov(X,Y).
	\]
	If $f$ and $g$ are sufficiently smooth and integrable, Price's theorem
	\cite{price1958} states that
	\[
		\frac{\partial}{\partial c}\,\E[f(X)g(Y)]
		=
		\E[f'(X)g'(Y)].
	\]
	Equivalently, since $c=\rho\sigma_1\sigma_2$,
	\[
		\frac{\partial}{\partial \rho}\,\E[f(X)g(Y)]
		=
		\sigma_1\sigma_2\,\E[f'(X)g'(Y)].
	\]
	The higher-order version, used later in the Edgeworth calculations in
	distributional form, is
	\[
		\frac{\partial^m}{\partial c^m}\,\E[f(X)g(Y)]
		=
		\E[f^{(m)}(X)g^{(m)}(Y)].
	\]
	Only $m=1,2,3$ occur below (in the kink integrals $J_{ab}$ and the smooth constant
	$\asm$), and always in the distributional form legitimized once and for all by the
	mollification of Remark~\ref{rem:regularization}.

	Taking $f(x)=|x|$, $g(y)=|y|$ (so $f'=g'=\sgn$) and using the quadrant formula
	$\E[\sgn X\,\sgn Y]=\tfrac2\pi\arcsin\rho$ gives $\partial_\rho\E|X|\,|Y|=\tfrac{2\sigma_1\sigma_2}{\pi}\arcsin\rho$;
	integrating from $\rho=0$ (where $\E|X|\,|Y|=\sigma_1\sigma_2(\E|Z|)^2=2\sigma_1\sigma_2/\pi$)
	recovers~\eqref{eq:SM}, exactly as in the proof of Lemma~\ref{lem:sheppard-mardia}.

	The non-smooth use of Price's theorem in this derivation is justified by the
	regularization convention recorded once in the following remark.

\begin{remark}[Regularization of the kink weights]\label{rem:regularization}
	The only technical point in the preceding proof is that $|x|$ is not smooth at
	the origin, while the classical form of Price's theorem is stated for smooth
	nonlinearities.  The same issue recurs later when distributional derivatives of
	kink weights are paired with Gaussian or Edgeworth densities.  Every such use of
	Price's theorem, and every pairing of a corner weight
	($|w|$, $|w_1||w_2|$, or their distributional derivatives $\sgn$, $\delta$, $\delta'$)
	against a Gaussian or Edgeworth density, is made rigorous by one mollification.
	Replace $|t|$ by the smooth even convex function
	$m_\varepsilon:=\eta_\varepsilon*|\cdot|$,
	with $\eta_\varepsilon$ a Gaussian mollifier: then $m_\varepsilon\to|\cdot|$
	uniformly, $0\le m_\varepsilon(t)\le|t|+\varepsilon$, $m_\varepsilon'\to\sgn$
	boundedly, and $m_\varepsilon''\to2\delta$, $m_\varepsilon'''\to2\delta'$ in the
	sense of distributions. For each fixed $\varepsilon>0$ Price's theorem and the
	integrations by parts are the classical smooth statements. The Gaussian and
	Edgeworth-polynomial densities and all their derivatives that occur here are Schwartz
	functions (the lattice remainder is treated separately in
	Lemma~\ref{lem:c2-remainder}), so $|t|+\varepsilon$
	(resp.\ $(|t_1|+\varepsilon)(|t_2|+\varepsilon)$) is an integrable majorant, and every
	moment, every $\rho$-derivative, and every kink integral $J_{ab}$ converges as
	$\varepsilon\downarrow0$ --- by dominated convergence for the pairings with
	$m_\varepsilon,m_\varepsilon'$, and by $\eta_\varepsilon*\varphi\to\varphi$ uniformly on
	compacts for those with $m_\varepsilon'',m_\varepsilon'''$ --- uniformly in the parameters on
	which it is subsequently used ($\rho$ on compact subsets of $(-1,1)$, the lattice
	index, etc.). We invoke this argument silently at each occurrence below.
\end{remark}

\begin{definition}[The constants $A$, $B$, $C_1$]\label{def:A-B}
	\begin{equation}\label{eq:A-B-def}
		A:=1-\frac2\pi,\qquad
		B:=\frac{\sqrt3-2}{\pi}+\frac16,\qquad
		C_1:=2A+8B=\frac{10}{3}+\frac{8\sqrt3-20}{\pi}.
	\end{equation}
	Numerically $A\approx0.363380$, $B\approx0.081376$, $C_1\approx1.377767$.
\end{definition}

	$A$ is the variance contribution of a single edge: for $Z\sim\mathcal N(0,1)$,
	$\E|Z|=\sqrt{2/\pi}$, so $\Var|Z|=\E|Z|^2-(\E|Z|)^2=1-2/\pi=A$, and $A=\Var|Z|/\Var
	Z$. $B$ is the (normalized) adjacent covariance: by Lemma~\ref{lem:sheppard-mardia}
	at $\rho=-\tfrac12$, $\sigma^2=1$, and subtracting the product of marginal means
	$\E|W_1|\E|W_2|=\tfrac2\pi$,
	\begin{align}
		\Cov(|W_1|,|W_2|)\Big|_{\rho=-\frac12}
		&=\frac2\pi\Big[\sqrt{1-\tfrac14}+\big(-\tfrac12\big)\arcsin\big(-\tfrac12\big)-1\Big]\notag\\
		&=\frac2\pi\Big[\frac{\sqrt3}{2}+\frac{\pi}{12}-1\Big]
		=\frac{\sqrt3-2}{\pi}+\frac16=B,\label{eq:B-derivation}
	\end{align}
	the term $\rho\arcsin\rho=(-\tfrac12)(-\tfrac\pi6)=+\tfrac\pi{12}$ being positive.
	$C_1=2A+8B$ is the combination that measures the gap between the linear and cyclic
	problems (see Theorem~\ref{thm:z2}).

\begin{theorem}[Leading variance $\Var(d)/N$]\label{thm:z2}
	For the uniform law and fixed $n\ge2$, as $N\to\infty$,
	\begin{equation}\label{eq:z2}
		\alpha(n):=\lim_{N\to\infty}\frac{\Var(d)}{N}=\frac{2(n-1)}{n}A+\frac{4(n-2)}{n}B.
	\end{equation}
	For the cyclic variant with $n\ge3$, $\alpha_c:=\lim_{N\to\infty}\Var(d^{c})/N=2A+4B
	=\tfrac83+\tfrac{4\sqrt3-12}{\pi}$, and $\alpha_c-\alpha(n)=C_1/n$.
\end{theorem}

\begin{proof}
	By Lemma~\ref{lem:cov-Dj} the normalized pair
	$(D_j,D_{j+1})/\sqrt{2N/n}$ converges to a bivariate normal with $\rho=-\tfrac12$,
	while a disjoint normalized pair converges to a bivariate normal with zero
	correlation, hence with independent components.  Thus the three terms in
	\eqref{eq:Var-d-3comp} have the following Gaussian leading forms:
	\[
		\Var(|D_1|)\sim \frac{2N}{n}\,A,\qquad
		\Cov(|D_1|,|D_2|)\sim \frac{2N}{n}\,B,\qquad
		\Cov(|D_1|,|D_3|)=o(N).
	\]
	The first constant is $A=\Var|Z|$ for a standard normal $Z$, scaled by the
	edge variance $2N/n$.  The second is the Nabeya covariance for two
	standard normals with correlation $-1/2$, again scaled by $2N/n$.  The third
	has no $O(N)$ term because the limiting disjoint Gaussian variables are
	independent.

	To pass from convergence in
	distribution to second moments of absolute values, note that each $D_j$ is a sum
	of independent centered bounded increments.  Expanding the fourth moment gives
	$\E D_j^4=O(N^2)$; by Cauchy--Schwarz the corresponding mixed fourth moments are
	also $O(N^2)$.  Equivalently, the normalized variables $\hat D_j:=D_j/\sqrt{2N/n}$
	have $\sup_N\E\hat D_j^4<\infty$, so $\hat D_j^2$ and, by Cauchy--Schwarz,
	$|\hat D_j|\,|\hat D_k|$ are bounded in $L^2$ and hence uniformly integrable.  Convergence
	in distribution together with uniform integrability upgrades to convergence of the
	corresponding first and second (cross-)moments of the absolute values, so the
	distributional limits give
	\[
		\Var(|D_j|)=(2N/n)A+o(N),\qquad
		\Cov(|D_1|,|D_2|)=(2N/n)B+o(N),
	\]
	and $\Cov(|D_1|,|D_3|)=o(N)$. Substituting into
	\eqref{eq:Var-d-3comp} and dividing by $N$ gives \eqref{eq:z2}; for $n\ge3$ the cyclic
	case adds the wrap-around edge, which creates two additional adjacent pairs (one at
	each endpoint) with the same correlation $-\tfrac12$, giving $2A+4B$; at $n=2$ the wrap
	edge is $D_2=-D_1$, so $d^{c}=2|D_1|$ and $\Var(d^{c})/N\to4A$. The later
	Theorems~\ref{thm:Var-D1}, \ref{thm:Cov-disj}, and \ref{thm:astar} sharpen these
	three remainders to the constant order needed for $\beta(n)$.
	Multiplied by $N$, formula~\eqref{eq:z2} is the normal-theory variance of the sum of absolute
	successive differences of $n$ i.i.d.\ $N(0,N/n)$ variables, i.e., $(n-1)^2$ times Hoel's
	variance of the mean moving range~\cite[eq.~(10)]{hoel1946} (see also Kamat~\cite{kamat1953});
	what is specific to the multinomial is that
	Lemma~\ref{lem:cov-Dj} holds \emph{exactly} and not only in the limit, so that the constant-order
	corrections below are where the lattice enters.
\end{proof}

\begin{corollary}[Asymptotic dispersion]\label{cor:sigma-infty}
	$\displaystyle\sigma_\infty^2:=\lim_{n\to\infty}\alpha(n)=2A+4B=\frac83+\frac{4(\sqrt3-3)}{\pi}$,
	so $\sigma_\infty=1.025799\ldots$.
\end{corollary}

\begin{proof}
	In \eqref{eq:z2}, $(n-1)/n\to1$ and $(n-2)/n\to1$, so $\alpha(n)\to2A+4B$; substitute
	$A,B$.
\end{proof}

\subsection{The exact linear--cyclic identity}

\begin{lemma}[Exact linear--cyclic identity]\label{lem:linear-cyclic}
	Let $d^{c}=\sum_{j=1}^{n}|D_j|$ be the cyclic statistic, the wrap-around edge
	$D_n=X_1-X_n$ included, so that $d=d^{c}-|D_n|$. Under the uniform
	$\Mult(N;1/n,\dots,1/n)$ the $n$ cyclic increments $|D_1|,\dots,|D_n|$ are
	cyclically exchangeable, and
	\begin{equation}\label{eq:linear-cyclic}
		\Var(d)=\Bigl(1-\frac2n\Bigr)\Var(d^{c})+\Var(|D_1|).
	\end{equation}
\end{lemma}

	\begin{proof}
	Cyclic exchangeability allows us to compute
	$\Cov(d^{c},|D_n|)$ without knowing any individual covariance.  Cyclic invariance
	of the uniform multinomial makes $(|D_1|,\dots,|D_n|)$ cyclically exchangeable, so
	$\Cov(d^{c},|D_j|)$ is the same for every $j$; averaging over $j$,
	\[
		\Cov(d^{c},|D_j|)=\frac1n\sum_{i=1}^{n}\Cov(d^{c},|D_i|)
		=\frac1n\Cov(d^{c},d^{c})=\frac{\Var(d^{c})}{n}.
	\]
	Since $d=d^{c}-|D_n|$ and $\Var(|D_n|)=\Var(|D_1|)$ by exchangeability,
	\[
		\Var(d)=\Var(d^{c})-2\Cov(d^{c},|D_n|)+\Var(|D_n|)
		=\Bigl(1-\frac2n\Bigr)\Var(d^{c})+\Var(|D_1|).
	\]
	Dividing by $N$ and writing $\sigma_c^2(n,N):=\Var(d^{c})/N$ gives
	\[
		\sigma^2(n,N)=\Bigl(1-\frac2n\Bigr)\sigma_c^2(n,N)+\frac{\Var(|D_1|)}{N};
	\]
	as $N\to\infty$ this yields $\alpha(n)=(1-\tfrac2n)\alpha_c+2A/n$, the leading gap
	$\alpha_c-\alpha(n)=C_1/n$ of Theorem~\ref{thm:z2} being the infinite-$N$ shadow of the
	exact identity.
\end{proof}

\subsection{Per-edge variance with $1/N$ correction}
	\label{ssec:Var-Dj}

	The CLT value $\Var(|D_1|)=(2N/n)A$ is only the leading order; the next order in
	$1/N$ is a concrete analytic correction.
	This is the only constant-order correction in this section because it does not
	require a new two-edge calculation.  Since
	\[
		\Var(|D_1|)=\E D_1^2-(\E|D_1|)^2,
	\]
	the correction follows as soon as Part~1 supplies the expansion of $\E|D_1|$; recall that
	the law of $|D_1|$ is that of $|2\Bin(K,\tfrac12)-K|$ with $K=X_1+X_2\sim\Bin(N,2/n)$
	(Proposition~\ref{prop:2bin-factor}).

\begin{lemma}[Imported per-edge mean expansion]\label{lem:imported-mean}
	For fixed $n\ge3$, Part~1~\cite[Lemma~5.1, Corollary~5.6 and Remark~5.7]{banic_elezovic_p1}
	gives
	\[
		W_N:=\E|D_1|
		=
		2\sqrt{\frac{N}{n\pi}}
		\left[
			1-\frac{n-2}{16N}
			+\frac{-3n^2+12n+4}{512N^2}
			+O(N^{-3})
		\right],
	\]
	the implied constant in $O(N^{-3})$ depending on the fixed $n$ (the expansion
	is not claimed uniform in $n$).
	In particular, the coefficient of $N^{-1}$ in the bracket is
	$e_1(n)=-(n-2)/16$, and the coefficient of $N^{-2}$ is
	$e_2(n)=(-3n^2+12n+4)/512$. This lemma is a genuine prerequisite: it enters the
	per-edge variance below, the constant-order covariances (Theorem~\ref{thm:Cov-disj},
	Proposition~\ref{prop:adj-transfer}) and all three second-order coefficients of
	Theorem~\ref{thm:gamma-linear}, so Part~1 is assumed throughout.
\end{lemma}

\begin{theorem}[Per-edge variance]\label{thm:Var-D1}
	For fixed $n\ge3$, as $N\to\infty$,
	\begin{equation}\label{eq:Var-D1}
		\Var(|D_1|)=\frac{2N}{n}A+\frac{n-2}{2\pi n}+O\!\Big(\frac1N\Big).
	\end{equation}
	(For $n=2$ the constant term is $(-1)^N/\pi$ instead; Theorem~\ref{thm:Var-n2}.)
\end{theorem}

\begin{proof}
	$\E[D_1^2]=\Var(D_1)=2N/n$ exactly (Lemma~\ref{lem:cov-Dj}). From the mean
	asymptotics in Lemma~\ref{lem:imported-mean},
	\[
		\E|D_1|
		=2\sqrt{\frac{N}{n\pi}}
		\left[1-\frac{n-2}{16N}+O(N^{-2})\right].
	\]
	Squaring gives
	\[
		(\E|D_1|)^2
		=\frac{4N}{\pi n}
		\left[1-\frac{n-2}{8N}+O(N^{-2})\right],
	\]
	because $(1-u)^2=1-2u+O(u^2)$ with $u=(n-2)/(16N)$.  Therefore
	\[
		\Var(|D_1|)=\frac{2N}{n}-\frac{4N}{\pi n}+\frac{n-2}{2\pi n}+O\!\Big(\frac1N\Big)
		=\frac{2N}{n}A+\frac{n-2}{2\pi n}+O\!\Big(\frac1N\Big).
	\]
\end{proof}

	For $n\ge3$ define $\beta(n)$ by $\Var(d)=N\,\alpha(n)-\beta(n)+O(1/N)$ (at $n=2$ the
	constant term oscillates with the parity of $N$, Theorem~\ref{thm:Var-n2}). The per-edge sub-leading
	term contributes, through \eqref{eq:Var-d-3comp},
	\begin{equation}\label{eq:beta-per-edge}
		\beta_{\text{per-edge}}(n)=-(n-1)\cdot\frac{n-2}{2\pi n}=-\frac{(n-1)(n-2)}{2\pi n},
	\end{equation}
	with $\beta_1^{\text{per-edge}}:=\lim_n\beta_{\text{per-edge}}(n)/n=-1/(2\pi)$.
	The minus sign is only a convention: $\beta(n)$ is defined by subtracting the
	constant-order term from $N\alpha(n)$.  Thus a positive constant correction in
	$\Var(|D_1|)$ becomes a negative contribution to $\beta(n)$.

\section{The disjoint-pair contribution}
	\label{sec:disjoint-pair}

	The disjoint term is the first place where the variance sees something that the
	covariance matrix of the raw differences does not show.  The raw variables
	$D_1=X_2-X_1$ and $D_3=X_4-X_3$ are uncorrelated under the uniform law, but the
	absolute values $|D_1|$ and $|D_3|$ retain a small dependence.  This dependence
	is invisible at Gaussian order and appears only as a constant-order residue.

	The section has two layers.  First we compute the exact finite-sample object
	$h(m)$, the expected product of two absolute random-walk positions after the
	four relevant bins have received $m$ observations.  This is an elementary
	conditioning calculation and produces the Bessel generating function.  Then we
	return to the full multinomial model, average over $M\sim\Bin(N,4/n)$, and extract
	the asymptotic covariance.  The second step is the more delicate one: it explains why
	the residual is $-1/(\pi n)$, rather than zero.

\subsection{The disjoint skeleton $h(m)$}
	\label{ssec:h-closed}

	Before the clean asymptotic $-1/(\pi n)$ we need its exact finite-sample skeleton,
	a sequence $h(m)$ of independent interest. Where the mean involved a
	\emph{single} central binomial coefficient, the disjoint covariance involves their
	\emph{square}, and the difference between a coefficient and its square is the whole
	signature of a product of two random walks rather than one.

	We reduce $\Cov(|D_1|,|D_3|)$ to the four-bin marginal $(X_1,X_2,X_3,X_4)$. Let
	$M:=X_1+X_2+X_3+X_4\sim\Bin(N,4/n)$; conditionally on $M=m$, the vector is
	$\Mult(m;\tfrac14,\tfrac14,\tfrac14,\tfrac14)$. Hence
	\begin{equation}\label{eq:Edisj-reduction}
		\E[|D_1|\,|D_3|]=\E_M[h(M)],\qquad
		h(m):=\E\big[|X_2-X_1|\,|X_4-X_3|\big]
	\end{equation}
	under $\Mult(m;\tfrac14,\tfrac14,\tfrac14,\tfrac14)$.
	In words, $h(m)$ freezes the number of observations falling into the four bins
	that matter for $D_1$ and $D_3$.  Once $m$ is fixed, no other bin can influence
	$|D_1||D_3|$, so all information about the disjoint product is compressed into
	this one sequence.

\begin{theorem}[Closed form for $h(m)$]\label{thm:h-closed}
	With $h(0)=h(1)=0$, for $j\ge1$
	\begin{equation}\label{eq:h-closed}
		h(2j)=\frac{j^2\binom{2j}{j}^2}{2^{4j-1}},\qquad
		h(2j+1)=\frac{j(2j+1)\binom{2j}{j}^2}{2^{4j}}.
	\end{equation}
\end{theorem}

\begin{proof}
	\emph{Step 1 (reduction to a convolution).} Under
	$\Mult(m;\tfrac14,\tfrac14,\tfrac14,\tfrac14)$ put $K:=X_1+X_2\sim\Bin(m,\tfrac12)$.
	Conditionally on $K=k$ the groups $\{X_1,X_2\}$ and $\{X_3,X_4\}$ are independent,
	with $X_1\mid K\sim\Bin(k,\tfrac12)$ and $X_3\mid(m-K)\sim\Bin(m-k,\tfrac12)$. Thus
	$\E[|X_2-X_1|\,|X_4-X_3|\mid K=k]=g(k)\,g(m-k)$, where $g(k):=\E|S_k|$ is the mean
	absolute position of the symmetric $\pm1$ walk after $k$ steps---equivalently twice the
	mean absolute deviation of $\Bin(k,\tfrac12)$~\cite{elezovic_mad}---
	$g(k)=k\binom{k-1}{\lfloor(k-1)/2\rfloor}/2^{k-1}$, $g(0)=0$. Summing over $K$,
	\begin{equation}\label{eq:h-sum}
		h(m)=\sum_{k=0}^{m}\binom{m}{k}2^{-m}\,g(k)\,g(m-k),
	\end{equation}
	and $h(0)=h(1)=0$ since each summand then contains a factor $g(0)=0$.
	The formula is simply the law of total expectation.  The binomial weight
	$\binom mk2^{-m}$ chooses how many of the $m$ observations fall in the first
	pair of bins.  Given that split, the two pairs are independent, and each pair is
	just a symmetric random walk.

	\emph{Step 2 (closed form).} The closed form \eqref{eq:h-closed} is the coefficient of
	$t^m/m!$ in the generating function
	$\widehat h(t)=[\widehat g(t/2)]^2=\tfrac{t^2}4[I_0(t/2)+I_1(t/2)]^2$ of
	Proposition~\ref{prop:h-Bessel}, which is proved below directly from the convolution
	\eqref{eq:h-sum} and does not use the present theorem; its part~(C) extracts the
	coefficients, giving \eqref{eq:h-closed} for every $j\ge1$ (the base values
	$h(2)=\tfrac12$, $h(3)=\tfrac34$ are the cases $j=1$, so no separate base case arises).
	As a consistency check, \eqref{eq:h-closed} satisfies the first-order recurrence
	\begin{equation}\label{eq:h-recurrence}
		\frac{h(2j)}{h(2j-1)}=\frac{2j-1}{2j-2},\qquad
		\frac{h(2j+1)}{h(2j)}=\frac{2j+1}{2j}\qquad(j\ge2),
	\end{equation}
	(consecutive ratios), which, started from $h(2)=\tfrac12$, gives
	$h(2j)=\tfrac12\bigl[\prod_{i=1}^{j-1}\tfrac{2i+1}{2i}\bigr]^2$ and, via the classical
	product $\prod_{i=1}^{j-1}\tfrac{2i+1}{2i}=j\binom{2j}{j}/2^{2j-1}$, the even case; the
	odd case is $h(2j+1)=\tfrac{2j+1}{2j}h(2j)$.
\end{proof}

\subsection{Generating function for the skeleton}

\begin{proposition}[Generating function for $h(m)$]\label{prop:h-Bessel}
	The exponential generating function $H(z):=\sum_{m\ge0}h(m)z^m/m!$ has the closed
	form
	\begin{equation}\label{eq:h-bessel}
		H(z)=\frac{z^2}{4}\big[I_0(z/2)+I_1(z/2)\big]^2
	\end{equation}
	in modified Bessel functions of the first kind~\cite{abramowitz_stegun}. Separating
	parities (using that $I_0$ is even and $I_1$ odd) gives the even part
	\[
		\sum_{j\ge0}h(2j)z^{2j}/(2j)!=\tfrac{z^2}4[I_0(z/2)^2+I_1(z/2)^2].
	\]
\end{proposition}

\begin{proof}[Proof of Proposition~\ref{prop:h-Bessel}]
	We compute the generating function directly from the convolution \eqref{eq:h-sum}
	established in Step~1 of Theorem~\ref{thm:h-closed}; this proof uses only that
	reduction, not the theorem's conclusion.  Parts (A)--(C) below establish
	\eqref{eq:h-bessel} and, by extracting coefficients in part~(C), the closed form
	\eqref{eq:h-closed} of Theorem~\ref{thm:h-closed} as well.
	The reason for using an exponential generating function is that the convolution in
	\eqref{eq:h-sum} contains the binomial coefficient $\binom mk$.  Exponential
	generating functions turn such binomial convolutions into ordinary products; the
	extra factor $2^{-m}$ is handled by replacing $t$ with $t/2$.

	\emph{(A) Series of $I_0+I_1$.} With $b_s:=\binom{s}{\lfloor s/2\rfloor}$, the
	standard series of the modified Bessel functions give, at argument $z/2$,
	\[
		I_0(z/2)+I_1(z/2)=\sum_{s\ge0}\frac{b_s}{s!}\Big(\frac z4\Big)^{s},
	\]
	since the coefficient of $(z/4)^{2m}$ is $1/(m!)^2=b_{2m}/(2m)!$ and that of
	$(z/4)^{2m+1}$ is $1/(m!\,(m+1)!)=b_{2m+1}/(2m+1)!$.

	\emph{(B) Squaring.} The sum \eqref{eq:h-sum} is the binomial convolution of the
	sequence $g$ with weight $2^{-m}$; in terms of the exponential generating function
	$\widehat g(t)=\sum_k g(k)t^k/k!$ the weight is absorbed by $t\mapsto t/2$, so
	$\widehat h(t):=\sum_m h(m)t^m/m!=[\widehat g(t/2)]^2$. To identify $\widehat g$,
	insert $g(k)=k\,b_{k-1}/2^{k-1}$ (with $b_{k-1}=\binom{k-1}{\lfloor(k-1)/2\rfloor}$,
	from $g(k)=k\binom{k-1}{\lfloor(k-1)/2\rfloor}/2^{k-1}$) and use $k/k!=1/(k-1)!$;
	the single index shift $s=k-1$ then turns the $b_{k-1}$ into the coefficients of
	part~(A):
	\[
		\widehat g(t)=\sum_{k\ge1}\frac{k\,b_{k-1}}{2^{k-1}}\frac{t^k}{k!}
		=\sum_{k\ge1}\frac{b_{k-1}}{2^{k-1}}\frac{t^k}{(k-1)!}
		\stackrel{s=k-1}{=}t\sum_{s\ge0}\frac{b_s}{s!}\Big(\frac t2\Big)^s
		=t\bigl[I_0(t)+I_1(t)\bigr],
	\]
	the last equality being part~(A) evaluated at $z/4=t/2$. Hence
	$\widehat g(t/2)=\tfrac t2[I_0(t/2)+I_1(t/2)]$ and
	$\widehat h(t)=\tfrac{t^2}4[I_0(t/2)+I_1(t/2)]^2$ --- this is \eqref{eq:h-bessel}.
	This is the entire source of the Bessel functions in the disjoint skeleton: they
	are not introduced from outside, but are exactly the generating functions of the
	central binomial coefficients that already occur in the one-walk expectation
	$g(k)$.

	\emph{(C) Closed form.} Decompose $w(t)^2:=[I_0(t/2)+I_1(t/2)]^2$ by parity in $t$
	(recall $I_0(t/2)$ is even, $I_1(t/2)$ odd),
	\[
		w(t)^2=\bigl[I_0(t/2)^2+I_1(t/2)^2\bigr]+2\,I_0(t/2)I_1(t/2),
	\]
	and expand each Bessel product. From
	\[
		I_0(z)=\sum_{a\ge0}\frac{(z/2)^{2a}}{(a!)^2},\qquad
		I_1(z)=\sum_{b\ge0}\frac{(z/2)^{2b+1}}{b!\,(b+1)!},
	\]
	the Cauchy products are summed by Vandermonde's convolution,
	\[
		\sum_{a=0}^K\binom Ka^2=\binom{2K}{K},\qquad
		\sum_{a=0}^K\binom Ka\binom{K+1}{a}=\binom{2K+1}{K},
	\]
	together with $\sum_{a=0}^K\binom{K+1}{a}\binom{K+1}{a+1}=\binom{2K+2}{K}$, giving
	\[
		I_0(z)^2=\sum_{K\ge0}\frac{\binom{2K}{K}}{(K!)^2}\Bigl(\tfrac z2\Bigr)^{2K},\qquad
		I_0(z)I_1(z)=\sum_{K\ge0}\frac{\binom{2K+1}{K}}{K!\,(K+1)!}\Bigl(\tfrac z2\Bigr)^{2K+1},
	\]
	and $I_1(z)^2=\sum_{K\ge0}\binom{2K+2}{K}\big/((K+1)!)^2\,(z/2)^{2K+2}$. Substituting into
	$\widehat h(t)=\tfrac{t^2}4 w(t)^2$ now gives the coefficients without any
	unwritten convolution.

	\emph{Odd coefficient.}
	Only the term $\tfrac{t^2}{2}I_0(t/2)I_1(t/2)$ contributes to $t^{2j+1}$.
	In the displayed series for $I_0(z)I_1(z)$ put $z=t/2$ and $K=j-1$.  Therefore
	\[
		\frac{h(2j+1)}{(2j+1)!}
		=\frac{1}{2\cdot4^{2j-1}}\,
		\frac{\binom{2j-1}{j-1}}{(j-1)!\,j!},
	\]
	and hence
	\[
		h(2j+1)=\frac{(2j+1)!}{2^{4j-1}}\,\frac{1}{(j-1)!\,j!}\binom{2j-1}{j-1}
		=\frac{j(2j+1)}{2^{4j}}\binom{2j}{j}^2,
	\]
	where the last equality uses
	$\binom{2j-1}{j-1}=\tfrac12\binom{2j}{j}$ and
	$(2j+1)!/((j-1)!j!)=j(2j+1)\binom{2j}{j}$.

	\emph{Even coefficient.}
	The coefficient of $t^{2j}$ comes from
	$\tfrac{t^2}{4}[I_0(t/2)^2+I_1(t/2)^2]$.  The $I_0^2$ term uses $K=j-1$,
	while the $I_1^2$ term uses $K=j-2$ (the latter is absent when $j=1$, and the
	formula below still holds).  Thus
	\begin{align*}
		\frac{h(2j)}{(2j)!}
		&=\frac{1}{4^{2j-1}((j-1)!)^2}
		\left[\binom{2j-2}{j-1}+\binom{2j-2}{j-2}\right]\\
		&=\frac{1}{4^{2j-1}((j-1)!)^2}\binom{2j-1}{j-1},
	\end{align*}
	by Pascal's identity.  Since
	\[
		\binom{2j-1}{j-1}=\frac12\binom{2j}{j},
		\qquad
		\frac{(2j)!}{((j-1)!)^2}=j^2\binom{2j}{j},
	\]
	we obtain
	\[
		h(2j)=\frac{j^2}{2^{4j-1}}\binom{2j}{j}^2.
	\]
	These are precisely \eqref{eq:h-closed}.  Their consecutive ratios give
	\eqref{eq:h-recurrence}.  
\end{proof}

\begin{corollary}[Exact disjoint moment]
\label{cor:exact-disjoint-moment}
	For fixed $n\ge4$,
	\[
		\E[\,|D_1|\,|D_3|\,]=\sum_{m=2}^{N}\binom Nm\Big(\frac4n\Big)^m \Big(1-\frac4n\Big)^{N-m}h(m),
	\]
	where $h(m)$ is given explicitly in Theorem~\ref{thm:h-closed} (for $n=4$ read $0^0=1$,
	so that the sum is $h(N)$).
	This formula is useful for numerical computation of exact variance values: it reduces
	the disjoint contribution from the naive $O(N^4)$ quadruple loop over the
	four-bin marginal to a single $O(N)$ binomial sum.
\end{corollary}

\subsection{Fourier dictionary for absolute values}
	\label{ssec:Cov-disj}

	The closed form for $h(m)$ gives an efficient exact way to compute the disjoint
	moment.  For the asymptotic constant, however, it is more revealing to switch to
	characteristic functions.  The reason is that $|D_1||D_3|$ is non-smooth as a
	function of the lattice variables.  In Fourier variables, the absolute-value weight
	becomes an explicit multiplier $S(\theta)$.  The singularity of this multiplier
	at $\theta=0$ is exactly where the non-Gaussian residue lives, and the algebraic
	cancellation below makes that singularity manageable.

	The Fourier dictionary is as follows.  Multiplying the probability mass at an
	integer point $u$ by $|u|$ is equivalent, on the characteristic-function side, to
	applying the Fourier multiplier $|q|$.  Since this multiplier is not represented by
	an ordinary integrable periodic function, we introduce it through Abel
	regularization:
	\[
		S_r(\theta)=\sum_{m\in\Z}|m|r^{|m|}e^{im\theta},\qquad 0<r<1.
	\]
	The factor $r^{|m|}$ makes the series absolutely convergent, so all termwise
	integrations are classical before the limit $r\uparrow1$ is taken.  The sign in the
	Fourier coefficient convention is harmless but worth recording: when $S_r$ is
	paired with $e^{ik\theta}$, the selected index is $m=-k$, hence the coefficient is
	$|-k|r^{|-k|}=|k|r^{|k|}$.

	Pointwise away from $2\pi\Z$, the same kernel has the Abel sum
	\[
		S_r(\theta)=2\sum_{m\ge1}mr^m\cos(m\theta)
		=2\mathop{\rm Re}\frac{re^{i\theta}}{(1-re^{i\theta})^2}
		\longrightarrow -\frac12\csc^2(\theta/2).
	\]
	This is the standard Abel summation of a trigonometric series; see, for example,
	Zygmund~\cite{zygmund1959}, Katznelson~\cite{katznelson2004}, or
	Stein--Shakarchi~\cite{stein_shakarchi2003}.  Because
	$S(\theta)\sim -2/\theta^2$ at the origin, $S$ is used below only as a periodic
	finite-part distribution, defined by this Abel limiting procedure.  The general
	finite-part viewpoint is standard in distributional asymptotics; see
	Estrada--Kanwal~\cite{estrada_kanwal} or Gel'fand--Shilov~\cite{gelfand_shilov}.

	In two dimensions the construction is applied independently in the two variables.
	Thus, for a finite-support lattice vector $(U,V)$, the tensor-product Abel pairing
	sends each monomial $e^{i(u\theta_1+v\theta_2)}$ to $|u||v|$.  The following lemma
	records this dictionary in the form used in the disjoint-pair proof.

\begin{lemma}[Periodic absolute-value kernel]\label{lem:periodic-abs-kernel}
	For $0<r<1$ define
	\[
		S_r(\theta):=\sum_{m\in\Z}|m|r^{|m|}e^{im\theta}.
	\]
	Then, for every $k\in\Z$,
	\[
		\frac1{2\pi}\int_{-\pi}^{\pi}S_r(\theta)e^{ik\theta}\,d\theta
		=|k|r^{|k|}\longrightarrow |k|,
	\]
	and, away from $2\pi\Z$,
	\[
		S_r(\theta)\longrightarrow S(\theta):=-\frac12\csc^2(\theta/2).
	\]
	The periodic distribution $S$ is therefore unambiguously defined by the Abel
	pairing
	\[
		\frac1{2\pi}\fp\!\int_{-\pi}^{\pi}S(\theta)T(\theta)\,d\theta
		:=\lim_{r\uparrow1}\frac1{2\pi}\int_{-\pi}^{\pi}S_r(\theta)T(\theta)\,d\theta
	\]
	for every trigonometric polynomial $T$.  In particular, if $(U,V)$ has finite
	support in $\Z^2$, then
	\begin{equation}\label{eq:periodic-abs-moment}
		\E|U||V|=\frac1{(2\pi)^2}\fp\!\iint_{[-\pi,\pi]^2}
		\E e^{i(\theta_1U+\theta_2V)}S(\theta_1)S(\theta_2)\,d\theta ,
	\end{equation}
	where the two-dimensional finite part is the tensor-product Abel limit.
\end{lemma}

\begin{proof}
	For $0<r<1$ the defining series is absolutely and uniformly convergent, since
	\[
		\sum_{m\in\Z}|m|r^{|m|}<\infty.
	\]
	We may therefore integrate term by term:
	\begin{align*}
		\frac1{2\pi}\int_{-\pi}^{\pi}S_r(\theta)e^{ik\theta}\,d\theta
		&=\sum_{m\in\Z}|m|r^{|m|}
		\frac1{2\pi}\int_{-\pi}^{\pi}e^{i(m+k)\theta}\,d\theta\\
		&=|-k|r^{|-k|}=|k|r^{|k|},
	\end{align*}
	because the last integral is $1$ only for $m=-k$ and is $0$ otherwise.  This
	proves the first identity and its limit as $r\uparrow1$.

	To identify the pointwise Abel limit, pair the positive and negative indices:
	\[
		S_r(\theta)=2\sum_{m\ge1}mr^m\cos(m\theta)
		=2\mathop{\rm Re}\frac{re^{i\theta}}{(1-re^{i\theta})^2},
	\]
	because $\sum_{m\ge1}mz^m=z/(1-z)^2$.  Letting $r\uparrow1$ gives
	\[
		\frac{e^{i\theta}}{(1-e^{i\theta})^2}
		=\frac{e^{i\theta}}{(-2i e^{i\theta/2}\sin(\theta/2))^2}
		=-\frac1{4\sin^2(\theta/2)},
	\]
	so $S(\theta)=-[2\sin^2(\theta/2)]^{-1}$ for
	$\theta\notin2\pi\Z$.  This function is not locally integrable at the origin,
	which is why the Abel limit, rather than an ordinary improper integral, is part
	of the statement.

	Let $T(\theta)=\sum_{|q|\le Q}t_qe^{iq\theta}$ be a trigonometric polynomial.
	Since this sum is finite,
	\[
		\frac1{2\pi}\int_{-\pi}^{\pi}S_r(\theta)T(\theta)\,d\theta
		=\sum_{|q|\le Q}t_q\,|q|r^{|q|}
		\longrightarrow\sum_{|q|\le Q}t_q\,|q|.
	\]
	Thus the Abel pairing is well defined and has Fourier multiplier $|q|$.

	Finally, if $(U,V)$ has finite support, its characteristic function is the
	finite trigonometric polynomial
	\[
		\E e^{i(\theta_1U+\theta_2V)}
		=\sum_{u,v}\Pr(U=u,V=v)e^{i(u\theta_1+v\theta_2)}.
	\]
	Apply the one-dimensional Abel pairing first in $\theta_1$ and then in
	$\theta_2$.  Each monomial is sent to $|u||v|$; because the probability sum is
	finite, no limiting interchange is hidden:
	\[
	\begin{aligned}
		\frac1{(2\pi)^2}\fp\!\iint
		\E e^{i(\theta_1U+\theta_2V)}S(\theta_1)S(\theta_2)\,d\theta
		&=\sum_{u,v}|u|\,|v|\Pr(U=u,V=v)\\
		&=\E|U|\,|V|.
	\end{aligned}
	\]
	This is \eqref{eq:periodic-abs-moment}.
\end{proof}

\subsection{Asymptotic disjoint covariance}
	\label{ssec:Cov-disj-asymp-proof}

	Two edge differences that share no bin --- $D_1$ from $(X_1,X_2)$, $D_3$ from $(X_3,X_4)$ ---
	are exactly uncorrelated, $\Cov(D_1,D_3)=0$ (Lemma~\ref{lem:cov-Dj}), and asymptotically
	independent. Their absolute values nevertheless keep a covariance of constant order. The
	mechanism is classical: the multinomial is a vector of independent Poisson counts conditioned on
	its total, and conditioning on the total couples otherwise independent blocks at relative order
	$1/N$ (Remark~\ref{rem:poisson-heuristic} below). The exact object that carries the effect is
	the joint fourth cumulant of the disjoint pair.

\begin{proposition}[Exact joint fourth cumulant of a disjoint pair]\label{prop:kappa22}
	For the uniform multinomial and any two edge differences $D_i,D_j$ with $|i-j|\ge2$
	(written $D_1,D_3$), the joint fourth cumulant is exact for every $N$:
	\begin{equation}\label{eq:kappa22}
		\kappa_{2,2}(D_1,D_3)=-\frac{4N}{n^2}\qquad(n\ge4).
	\end{equation}
\end{proposition}

\begin{proof}
	The cumulant is additive over the $N$ independent trials, and a single trial cannot
	move both edges ($Y_1Y_3\equiv0$), so its per-trial contribution is
	$\E Y_1^2Y_3^2-\E Y_1^2\,\E Y_3^2-2(\E Y_1Y_3)^2=0-\tfrac2n\!\cdot\!\tfrac2n-0
	=-\tfrac4{n^2}$, the first and last terms vanishing by $Y_1Y_3\equiv0$; summing over
	the $N$ trials gives~\eqref{eq:kappa22}.
\end{proof}

	The exact characteristic-function identity of Steps~1--4 below writes the covariance as
	$\kappa_{2,2}$ times a lattice origin probability, so the limit $-1/(\pi n)$ in
	Theorem~\ref{thm:Cov-disj} below is non-zero precisely because $\kappa_{2,2}\ne0$;
	Remark~\ref{rem:poisson-heuristic} gives the one-line heuristics behind the value.

\begin{theorem}[Asymptotic disjoint covariance]\label{thm:Cov-disj}
	For the uniform law, as $N\to\infty$ with fixed $n\ge4$,
	\begin{equation}\label{eq:Cov-disj-asymp}
		\Cov(|D_1|,|D_3|)=-\frac{1}{\pi n}+O\!\Big(\frac1N\Big),
	\end{equation}
	and hence, by exchangeability of the bins, the same holds for every pair $D_i,D_j$ with
	$|i-j|\ge2$. For $n\ge5$ the $O(1/N)$ term is $-\tfrac{n+1}{8\pi n}\tfrac1N$
	(Lemma~\ref{lem:r1-disjoint}); for $n=4$ it is $r_1(4;N)/N$ with
	$r_1(4;N)=-\tfrac{5}{32\pi}+\tfrac{(-1)^N}{4\pi}$ (\S\ref{ssec:three-four-bins}).
\end{theorem}

	We give first a short proof from the exact skeleton $h(m)$, valid for every $n\ge4$. The
	characteristic-function derivation that follows it needs $n\ge5$ (for its local-limit step) and
	is not needed for the value; it exhibits the residue as $\kappa_{2,2}$ times a lattice origin
	probability.

\begin{proof}[Proof of Theorem~\ref{thm:Cov-disj} via the exact skeleton]
	Let $M:=X_1+X_2+X_3+X_4\sim\Bin(N,4/n)$, as in \eqref{eq:Edisj-reduction}. Conditionally on
	$M=m$ the four cells $\{1,2,3,4\}$ carry the uniform four-bin multinomial, so
	$\E(|D_1||D_3|\mid M=m)=h(m)$ with the disjoint skeleton of
	Theorem~\ref{thm:h-closed}, and hence $\E(|D_1||D_3|)=\E h(M)$. Stirling's expansion
	of the closed form of $h$ gives $h(m)=\tfrac m\pi-\tfrac1{2\pi}+O\bigl((m+1)^{-1}\bigr)$
	uniformly in $m\ge0$ (the parity term of $h$ is already of this order), so averaging over
	$M$ gives, for every $n\ge4$,
	\[
		\E h(M)=\frac{\E M}{\pi}-\frac1{2\pi}+O\bigl(\E(M+1)^{-1}\bigr)
		=\frac{4N}{\pi n}-\frac1{2\pi}+O(N^{-1}).
	\]
	By Lemma~\ref{lem:imported-mean},
	$(\E|D_1|)^2=\tfrac{4N}{\pi n}-\tfrac{n-2}{2\pi n}+O(N^{-1})$. Subtracting,
	\[
		\Cov(|D_1|,|D_3|)=\E h(M)-(\E|D_1|)^2
		=-\frac1{2\pi}+\frac{n-2}{2\pi n}+O(N^{-1})
		=-\frac1{\pi n}+O(N^{-1}).
	\]
	Carried one order further this is Lemma~\ref{lem:r1-disjoint} ($n\ge5$); at the endpoint
	$n=4$, where $M=N$ is deterministic, the $O(1/N)$ term inherits the parity oscillation of
	$h$ (\S\ref{ssec:three-four-bins}).
\end{proof}

\begin{remark}[Why the constant is $-1/(\pi n)$: two heuristics]\label{rem:poisson-heuristic}
	The value can be read off in two ways, both classical in spirit.
	(i)~\emph{Conditioning on the total (Poissonization).} If the $n$ counts were
	independent $\mathrm{Poisson}(\lambda)$, $\lambda=N/n$, then $|D_1|$ and $|D_3|$ would be
	independent. Conditioning on the total $T=\sum_iX_i$ adds, to first order,
	$-\Cov(|D_1|,T)\Cov(|D_3|,T)/\Var T$ --- the conditional-variance correction of the
	Poissonization central limit theorems of Holst~\cite{holst1972} and
	Morris~\cite{morris1975}; since $\E|D_1|\approx\sqrt{4\lambda/\pi}$,
	$\Cov(|D_1|,T)=\lambda\,\tfrac{d}{d\lambda}\sqrt{4\lambda/\pi}=\sqrt{\lambda/\pi}$ and
	$\Var T=n\lambda$, this gives $-(\lambda/\pi)/(n\lambda)=-1/(\pi n)$. The covariance is thus a
	relative-order-$1/N$ effect of the fixed sample size: $\Var(|D_j|)\sim2NA/n$, so
	$\Corr(|D_1|,|D_3|)\approx-1/(2\pi AN)\to0$, in agreement with the Gaussian limit.
	(ii)~\emph{First Edgeworth term.} For an uncorrelated pair with joint
	cumulant $\kappa_{2,2}$ and variances $\sigma^2=2N/n$, the first Gram--Charlier correction to
	$\Cov(f(D_1),f(D_3))$ (the $\kappa_{2,2}$ cross term) is $\tfrac14\kappa_{2,2}(\E f'')^2$; with $f=|\cdot|$,
	$\E f''=2\phi_\sigma(0)=\sqrt{2/\pi}/\sigma$, so the correction is
	$\kappa_{2,2}/(2\pi\sigma^2)=-1/(\pi n)$. The characteristic-function proof below ($n\ge5$) is
	heuristic~(ii) made rigorous on the lattice, the kink of $|\cdot|$ included; the $O(1/N)$
	coefficient itself is read off the skeleton $h(M)$ in Lemma~\ref{lem:r1-disjoint}. Heuristic~(i)
	is confirmed by the theorem but not used in the proofs. For squares the analogue is exact:
	$\Cov(D_1^2,D_3^2)=\kappa_{2,2}=-4N/n^2$.
\end{remark}

	\medskip
	\noindent\emph{The characteristic-function picture: why disjoint edges stay
	correlated.} The elementary proof settles the value but hides the mechanism. We now
	record the exact joint characteristic function, where the non-smoothness of
	$|D_1||D_3|$ is carried by an explicit Fourier multiplier and an algebraic identity
	removes it; this exhibits the residue as a genuine non-Gaussian (fourth-cumulant)
	effect and is the representation that will be used for the cyclic transfer in
	Part~5~\cite{banic_elezovic_p5}.

	\emph{Step 1 (per-trial characteristic function).} Write
	\[
		D_1=\sum_{k=1}^N Y_{1,k},\qquad D_3=\sum_{k=1}^N Y_{3,k},
	\]
	with $Y_{j,k}=\mathbf1[Z_k=j+1]-\mathbf1[Z_k=j]$ and $Z_k$ i.i.d.\ uniform on
	$\{1,\dots,n\}$. A single trial yields $(Y_1,Y_3)=(-1,0),(+1,0),(0,-1),(0,+1)$ for
	$Z\in\{1,2,3,4\}$ (each with probability $1/n$) and $(0,0)$ otherwise, so the
	one-trial joint characteristic function is
	\[
		\varphi(\theta_1,\theta_3)=1-a_1-a_3,\qquad a_j:=\tfrac2n\bigl(1-\cos\theta_j\bigr),
		\qquad \E\,e^{i(\theta_1D_1+\theta_3D_3)}=\varphi(\theta)^N,
	\]
	with marginal characteristic functions $\varphi_1(\theta_1)=\varphi(\theta_1,0)=1-a_1$ and
	$\varphi_3(\theta_3)=1-a_3$. Note $\Cov(D_1,D_3)=0$ exactly (disjoint supports,
	$Y_1Y_3\equiv0$), so a Gaussian model would predict independence; the covariance
	below is the exact non-Gaussian residue.
	The notation is chosen so that all dependence is visible in one elementary
	algebraic defect.  The joint one-trial characteristic function is
	$\varphi=1-a_1-a_3$, while the product of the two marginal one-trial
	characteristic functions is
	\[
		\varphi_1\varphi_3=(1-a_1)(1-a_3)=1-a_1-a_3+a_1a_3.
	\]
	Thus $a_1a_3$ is exactly the amount by which the joint one-trial law fails to be
	the product of its two marginals.  Probabilistically this is forced by the fact that
	a single observation cannot contribute simultaneously to $D_1$ and to $D_3$:
	the four nonzero moves lie on the coordinate axes, never in the four diagonal
	directions.

	\emph{Step 2 (exact finite-part representation).}
	Apply Lemma~\ref{lem:periodic-abs-kernel} to the joint law and to the product of
	the two marginal laws.  Subtraction gives the exact identity
	\begin{equation}\label{eq:cov-disj-rep}
		\Cov(|D_1|,|D_3|)=\frac1{(2\pi)^2}\,\fp\!\iint_{[-\pi,\pi]^2}
		\bigl[\varphi(\theta)^N-\varphi_1(\theta_1)^N\varphi_3(\theta_3)^N\bigr]
		S(\theta_1)S(\theta_3)\,d\theta
	\end{equation}
	without any interchange issue, since both laws have finite support.  The integrand
	difference vanishes precisely to the extent that $\varphi$ fails to factor.
	If $D_1$ and $D_3$ were independent, then the joint characteristic function
	would factor as $\varphi_1^N\varphi_3^N$, and the bracket would vanish.  Thus
	\eqref{eq:cov-disj-rep} is an exact Fourier expression for the failure of
	independence after taking absolute values.

	\emph{Step 3 (kernel-smoothing identity and an exact reduction).} Because
	$1-\cos\theta=2\sin^2(\theta/2)$,
	\begin{equation}\label{eq:aS}
		a_j(\theta_j)\,S(\theta_j)
		=\tfrac2n\cdot2\sin^2\tfrac{\theta_j}2\cdot\Bigl(-\tfrac1{2\sin^2(\theta_j/2)}\Bigr)
		=-\frac2n\qquad(\theta_j\ne0):
	\end{equation}
	the product extends continuously to the same value at $\theta_j=0$. Thus one
	factor of $a_j$ annihilates the singular kernel. Now
	$\varphi_1\varphi_3=(1-a_1)(1-a_3)=\varphi+a_1a_3$ \emph{exactly}.
	Here the sign is important: since $\varphi_1\varphi_3=\varphi+a_1a_3$, the
	factorized model is obtained by adding a positive correction $a_1a_3$ to
	$\varphi$.  Therefore the bracket
	$\varphi^N-(\varphi_1\varphi_3)^N$ is the negative of the binomial tail generated
	by this correction.
	Indeed,
	\[
		\varphi^N-(\varphi_1\varphi_3)^N
		=\varphi^N-(\varphi+a_1a_3)^N
		=-\sum_{k=1}^N\binom Nk\varphi^{N-k}(a_1a_3)^k.
	\]
	Multiplying by $S_1S_3$ and using
	\[
		(a_1a_3)^kS_1S_3
		=(a_1S_1)(a_3S_3)(a_1a_3)^{k-1}
		=\frac4{n^2}(a_1a_3)^{k-1}
	\]
	from
	\eqref{eq:aS} renders every term \emph{smooth} (the finite part is no longer
	needed); the cancellation $a_jS_j=-2/n$ is an identity of periodic distributions: the
	Fourier coefficients of $(1-\cos\theta)S$ are $|q|-\tfrac12(|q-1|+|q+1|)=-\delta_{q0}$, so for
	every trigonometric polynomial $T$ the Abel pairing of $S$ with $a_jT$ equals
	$-\tfrac2n\cdot\tfrac1{2\pi}\int T$; since every integrand here is a trigonometric
	polynomial in each variable, this applies factor by factor to the tensor-product pairing.
	Thus \eqref{eq:cov-disj-rep} becomes the exact identity
	\begin{equation}\label{eq:cov-disj-exact}
		\Cov(|D_1|,|D_3|)=-\frac4{n^2}\sum_{k=1}^N\binom Nk\,\mathcal I_{N,k},
	\end{equation}
	where
	\[
		\mathcal I_{N,k}:=\frac1{(2\pi)^2}\iint_{[-\pi,\pi]^2}\varphi(\theta)^{N-k}(a_1a_3)^{k-1}\,d\theta.
	\]
	This is the decisive simplification.  The singular kernels $S_1S_3$ have
	disappeared because one factor $a_1a_3$ cancels them.  What remains is an ordinary
	integral of smooth trigonometric polynomials, and the leading term is now accessible
	by a local limit estimate.
	Notice also what has not been approximated: \eqref{eq:cov-disj-exact} is still an
	exact finite-$N$ identity.  All asymptotics enter only after the singular
	absolute-value kernels have been removed.

\emph{Step 4 (leading term).} 
	For $k=1$ the weight $(a_1a_3)^{k-1}$ is equal to $1$.  Therefore
	\[
		\mathcal I_{N,1}
		=\frac1{(2\pi)^2}\iint_{[-\pi,\pi]^2}\varphi(\theta)^{N-1}\,d\theta .
	\]
	This is no longer a singular finite-part integral and no polynomial weight remains.
	It is exactly the Fourier inversion formula for the probability that a certain
	two-dimensional lattice random walk is at the origin after $N-1$ trials. To make this precise, let $Y=(Y_1,Y_3)$ be the one-observation contribution to
	$(D_1,D_3)$.  
	Denote
	\[
		W_{N-1}:=Y^{(1)}+\cdots+Y^{(N-1)},
	\]
	the sum of $N-1$ independent copies of $Y$, then the characteristic function of
	$W_{N-1}$ is $\varphi^{N-1}$.
	
	For any lattice variable $W$ on $\Z^2$ with characteristic function
	$\psi(\theta)=\E e^{i\theta\cdot W}$, Fourier inversion gives
	\[
		\Pr(W=0)=\frac1{(2\pi)^2}\int_{[-\pi,\pi]^2}\psi(\theta)\,d\theta .
	\]
	Applying this to $W=W_{N-1}$ gives
	\[
		\mathcal I_{N,1}=\Pr(W_{N-1}=0)
		=\Pr(D_1=0,D_3=0)
		\quad\text{for $N-1$ trials.}
	\]
	
	\paragraph{Aperiodicity and covariance}
	The one-step law is aperiodic on $\Z^2$ for $n\ge5$ because
	$\Pr(Y=(0,0))=(n-4)/n>0$: the support $\{(0,0),(\pm1,0),(0,\pm1)\}$ contains the
	origin together with the generators $(1,0),(0,1)$, so it cannot lie in any coset
	$c+\Lambda'$ of a proper sublattice $\Lambda'\subsetneq\Z^2$ (the origin forces
	$c\in\Lambda'$, and then $(1,0),(0,1)\in\Lambda'$ force $\Lambda'=\Z^2$).  Merely
	generating $\Z^2$ as a group would not suffice --- simple random walk on $\Z^2$ has
	generating steps yet period two --- so it is the stay-put atom, present exactly when
	$n\ge5$, that gives aperiodicity with span one.  (For $n=4$ the origin atom vanishes,
	the four remaining steps lie in the odd coset of $\{a+b\ \text{even}\}$, the walk is
	$2$-periodic, and the local limit theorem acquires a covolume factor $2$; this is why the
	characteristic-function derivation is restricted to $n\ge5$ --- the skeleton proof above
	covers $n=4$, whose parity term is made explicit in \S\ref{ssec:three-four-bins} --- and why
	the torus bound degenerates at $(\pi,\pi)$.)  Aperiodicity also makes $|\varphi|<1$ off the origin, so the
	off-origin part of the torus integral is exponentially small.
	This is the lattice condition needed for the usual local central limit theorem with
	span one.
	The one-step mean is zero, since the steps in the positive and negative directions
	occur with equal probabilities.  The covariance matrix of one step is diagonal:
	\[
		\Var(Y_1)=1^2\frac1n+(-1)^2\frac1n=\frac2n,\qquad
		\Var(Y_3)=\frac2n,
	\]
	and
	\[
		\Cov(Y_1,Y_3)=\E(Y_1Y_3)-\E Y_1\,\E Y_3=0.
	\]
	The product $Y_1Y_3$ is always zero, because a single observation can affect at most
	one of the two disjoint differences.  Hence the covariance matrix of one step is
	\[
		\Sigma=\frac2n I_2.
	\]
	After $N-1$ independent trials the covariance matrix is
	\[
		(N-1)\Sigma=(N-1)\frac2n I_2.
	\]
	
	\paragraph{Local central limit estimate at the origin}
	The local central limit theorem for aperiodic lattice
	walks~\cite[Ch.~5, \S22]{bhattacharya_rao} says that, uniformly at lattice points in the central
	range,
	\[
	\begin{aligned}
		\Pr(W_{N-1}=x)
		&=\frac1{2\pi\sqrt{\det((N-1)\Sigma)}}\\
		&\quad\times\exp\!\left(-\frac12 x^T((N-1)\Sigma)^{-1}x\right)
		\bigl(1+O(N^{-1/2})\bigr).
	\end{aligned}
	\]
	At the origin $x=0$, the exponential factor is $1$.  With
        $\Sigma=(2/n)I_2$,
	\[
	\det((N-1)\Sigma)
	=\det\!\left(\frac{2(N-1)}n I_2\right)
	=\left(\frac{2(N-1)}n\right)^2.
	\]
	Therefore
	\[
		\sqrt{\det((N-1)\Sigma)}
		=\frac{2(N-1)}n,
	\]
	and the leading normal density at the origin is
	\[
		\frac1{2\pi\sqrt{\det((N-1)\Sigma)}}
		=\frac1{2\pi\,2(N-1)/n}
		=\frac{n}{4\pi(N-1)}.
	\]
	Since
	\[
		\frac{1}{N-1}=\frac1N\left(1+O\!\left(\frac1N\right)\right),
	\]
	this becomes
	\[
		\frac{n}{4\pi(N-1)}
		=\frac{n}{4\pi N}\left(1+O\!\left(\frac1N\right)\right).
	\]
	Hence
	\[
		\mathcal I_{N,1}
		=\Pr(W_{N-1}=0)
		=\frac{n}{4\pi N}\left(1+O\!\left(\frac1N\right)\right).
	\]
	
	\paragraph{The relative error is $O(1/N)$}
	For a general lattice local expansion, one might expect the first correction after
	the Gaussian term to be of relative order $N^{-1/2}$.  Here that term vanishes by
        symmetry.
	
	The one-step law is centrally symmetric:
	\[
		Y\stackrel{d}{=}-Y.
	\]
	Consequently all odd moments and all odd cumulants are zero.  In an Edgeworth
	expansion, the term of order $N^{-1/2}$ is built from third cumulants.  Since all
	third cumulants vanish, the whole first Edgeworth polynomial is identically zero.
	The first possible nonzero correction is therefore of order $N^{-1}$, coming from
	fourth cumulants and products of third cumulants.  The products of third cumulants
	also vanish here, but fourth cumulants need not vanish.  Hence the natural statement is
	\[
		\Pr(W_{N-1}=0)
		=\frac{n}{4\pi N}\left(1+O\!\left(\frac1N\right)\right),
	\]
	not merely a formula with an $O(N^{-1/2})$ relative error.
	
	\paragraph{Contribution of the first binomial term}
	The $k=1$ term in the exact identity is
	\[
		-\frac{4N}{n^2}\,\mathcal I_{N,1}
		=-\frac4{n^2}N\cdot\frac{n}{4\pi N}\left(1+O\!\left(\frac1N\right)\right)
		=-\frac1{\pi n}+O\!\left(\frac1N\right).
	\]
	The coefficient $-4N/n^2$ of $\mathcal I_{N,1}$ is exactly $\kappa_{2,2}$ of
	Proposition~\ref{prop:kappa22}, so that
	\[
		\Cov(|D_1|,|D_3|)=\kappa_{2,2}\,\Pr(W_{N-1}=0)+O(1/N):
	\]
	the constant is the joint fourth cumulant times the lattice origin probability.

	\emph{The remainder $k\ge2$.} Since the constant is already established, it suffices to
	see that the terms $k\ge2$ contribute only $O(1/N)$. They are summed \emph{exactly}, at
	each fixed $\theta$, by Taylor's formula with integral remainder applied to
	$x\mapsto(\varphi+x)^N$ at $x=a_1a_3\ge0$:
	\begin{equation}\label{eq:cov-disj-taylor}
		\sum_{k=2}^N\binom Nk\,\varphi^{N-k}(a_1a_3)^{k-1}
		=N(N-1)\,a_1a_3\int_0^1(1-t)\,(\varphi+t\,a_1a_3)^{N-2}\,dt .
	\end{equation}
	Bounding the convex combination $\varphi+t\,a_1a_3=(1-t)\varphi+t\,\varphi_1\varphi_3$
	($0\le t\le1$) by
	\[
		|\varphi+t\,a_1a_3|\le\max\bigl(|\varphi|,\varphi_1\varphi_3\bigr)
		\le e^{-\frac3{10}(a_1+a_3)}
	\]
	--- an elementary estimate, from $1-a\le e^{-a}$ and $0\le a_j\le4/n$, valid for every
	$n\ge5$ --- makes the $k\ge2$ block of \eqref{eq:cov-disj-exact} at most
	\[
		\frac{N^2}{2\pi^2n^2}
		\Bigl[\int_{-\pi}^{\pi}a(\theta)\,e^{-\frac3{10}(N-2)a(\theta)}\,d\theta\Bigr]^2,
		\qquad a(\theta)=\tfrac2n(1-\cos\theta) .
	\]
	Since $a(\theta)\asymp\theta^2/n$ near the origin, the bracket is $O(N^{-3/2})$, so the
	block is $O(N^2\cdot N^{-3})=O(1/N)$. This confirms the constant $-1/(\pi n)$ of
	\eqref{eq:Cov-disj-asymp}, now exhibited as the residue carried by the joint fourth
	cumulant.

\section{Adjacent pairs: reduction and Edgeworth setup}
	\label{sec:adj-reduction}

	With the per-edge and disjoint pieces in hand, only the adjacent pairs remain, and
	assembling the three gives $\beta(n)$. The
	adjacent covariance hides a constant $a_*=-2\sqrt3/(9\pi)$ which enters $\beta_1$,
	the leading coefficient of the variance correction; it is the technical heart of the
	paper, and this section and the next are devoted to deriving it from first principles.

\subsection{Reduction to the conditional three-bin problem}
	\label{ssec:adj-bridge}

	The adjacent pair $(D_1,D_2)=(X_2-X_1,X_3-X_2)$ involves only the three cells
	$(X_1,X_2,X_3)$.  We first isolate the exact conditional model and its Gaussian
	leading coefficient.

\begin{lemma}[Conditional three-bin reduction]\label{lem:adj-conditional-reduction}
	Let
	\[
		L:=X_1+X_2+X_3\sim\Bin(N,3/n).
	\]
	Conditionally on $L=\ell$, the triple $(X_1,X_2,X_3)$ has the uniform three-bin law
	$\Mult(\ell;\tfrac13,\tfrac13,\tfrac13)$.  With
	\begin{equation}\label{eq:hadj-def}
		\hadj(\ell):=\E\big[\,|D_1|\,|D_2|\ \big|\ L=\ell\,\big]
		=\E\,|D_1|\,|D_2|\ \text{under }\Mult(\ell;\tfrac13,\tfrac13,\tfrac13),
	\end{equation}
	one has
	\[
		\E|D_1||D_2|=\E\,\hadj(L).
	\]
	Moreover the one-trial conditional covariance matrix is
	\[
		C=\begin{pmatrix}2/3&-1/3\\-1/3&2/3\end{pmatrix},
	\]
	so the limiting Gaussian variables have common variance $\sigma^2=2/3$ and
	correlation $\rho=-1/2$.  The corresponding Nabeya leading coefficient is
	\begin{equation}\label{eq:cL-def}
		c_L=\E|Z_1||Z_2|
		=\frac{2\sqrt3}{3\pi}+\frac19,
	\end{equation}
	where $(Z_1,Z_2)$ is centered normal with covariance matrix $C$.
\end{lemma}

\begin{proof}
	The multinomial conditioning is immediate: after fixing the total number $\ell$ of
	observations in the three relevant cells, those observations are distributed uniformly
	among the three cells.  The identity $\E|D_1||D_2|=\E\,\hadj(L)$ is then the law of
	total expectation.

	For one conditional trial, the contribution to $(D_1,D_2)$ is $(-1,0)$, $(1,-1)$,
	or $(0,1)$, each with probability $1/3$.  Therefore
	\[
		\Var(Y_1)=\Var(Y_2)=\frac23,\qquad
		\Cov(Y_1,Y_2)=-\frac13.
	\]
	After $\ell$ trials the covariance matrix is $\ell C$, and the normalized vector
	$(D_1,D_2)/\sqrt\ell$ has limiting covariance matrix $C$.  Finally,
	Lemma~\ref{lem:sheppard-mardia}, applied with $\sigma^2=2/3$ and $\rho=-1/2$, gives
	\[
		c_L=\frac2\pi\cdot\frac23
		\left(\frac{\sqrt3}{2}+\frac{\pi}{12}\right)
		=\frac{2\sqrt3}{3\pi}+\frac19.
	\]
\end{proof}

	Thus the adjacent problem has been reduced to the one-parameter sequence
	$\hadj(\ell)$.  The theorem proved in the next subsections gives the sharper
	conditional expansion
	\begin{equation}\label{eq:hadj-form}
		\hadj(\ell)=c_L\,\ell+a_*+O(\ell^{-1}),
	\end{equation}
	where $c_L$ is given by~\eqref{eq:cL-def}.

	More explicitly, we shall use the following consequence of this expansion: there are
	constants $C>0$ and an integer threshold $\ell_0\ge1$ such that
	\begin{equation}
	\label{eq:ell0}
	        \left|\hadj(\ell)-c_L\ell-a_*\right|
	        \le \frac{C}{\ell},
	        \qquad \ell\ge \ell_0.
	\end{equation}
	The constants $C$ and $\ell_0$ are absolute: $\hadj$ involves neither $n$ nor $N$.

\begin{proposition}[Transfer to the adjacent covariance]\label{prop:adj-transfer}
	Assume the conditional expansion \eqref{eq:hadj-form}.  Then, for fixed $n\ge3$,
	\begin{equation}\label{eq:adj-moment-transfer}
		\E|D_1||D_2|=\frac{3c_LN}{n}+a_*+O(N^{-1}),
	\end{equation}
	and
	\begin{equation}\label{eq:cov-adj-struct}
		\Cov(|D_1|,|D_2|)
		=\frac{2N}{n}B+
		\left(a_*+\frac{n-2}{2\pi n}\right)
		+O(1/N),
	\end{equation}
	where
	\[
	        B=\frac{\sqrt3-2}{\pi}+\frac16.
	\]
\end{proposition}

\begin{proof}

	Let $\mu:=\E L=3N/n$. We first justify averaging the conditional expansion over the random value $L$.
	Fix the typical event
	\[
	        A_N:=\left\{L\ge\frac{\mu}{2}\right\}.
	\]
	Here $\ell_0$ is the threshold from~\eqref{eq:ell0}.  Since $n$ is fixed, $\mu=3N/n$ is a positive constant multiple of $N$.  Choose $N_0$ so large that
	$\mu/2\ge\ell_0$ whenever $N\ge N_0$.  Then, for $N\ge N_0$, the event $A_N$ implies
	$L\ge\ell_0$, and the conditional expansion with the bound
	$|\hadj(L)-c_LL-a_*|\le C/L$ is valid throughout $A_N$.

	The complementary event is exponentially unlikely. If $n>3$, this is the standard lower-tail
	Chernoff bound for a binomial variable:
	\[
	        \Pr(A_N^c)=
	        \Pr\left(L<\frac{\mu}{2}\right)\le e^{-cN}
	\]
	for some $c=c(n)>0$. If $n=3$, then $L=N$ deterministically, so $A_N^c$ is empty for large
	$N$. Thus the same conclusion holds for every fixed $n\ge3$.

	On the exceptional event $A_N^c$ we use only the elementary deterministic bound
	\[
	        |D_1|\le L,\qquad |D_2|\le L,\qquad \hadj(L)\le L^2\le N^2.
	\]
	Hence
	\[
	        \E\bigl[\hadj(L);A_N^c\bigr]
	        \le N^2\Pr(A_N^c)=O(e^{-c'N})
	\]
	after decreasing $c'$ if necessary. This is smaller than any power of $N^{-1}$.

	On the typical event $A_N$, write
	\[
	        \hadj(L)=c_LL+a_*+R(L),
	        \qquad |R(L)|\le \frac{C}{L}.
	\]
	Then
	\[
	        \E[\hadj(L);A_N]=
	        c_L\E[L;A_N]+a_*\Pr(A_N)+\E[R(L);A_N].
	\]
	The first two terms may be replaced by their unconditional versions, with only an
	exponentially small error:
	\[
	        \E[L;A_N]=\E L-\E[L;A_N^c]=
	        \frac{3N}{n}+O(Ne^{-cN}),
	\]
	and
	\[
	        \Pr(A_N)=1+O(e^{-cN}).
	\]
	For the remainder, on $A_N$ one has $L^{-1}\le2/\mu=2n/(3N)$, so
	\[
	        \left|\E[R(L);A_N]\right|
	        \le C\E[L^{-1};A_N]
	        \le
	        \frac{2Cn}{3N}=O(N^{-1}).
	\]
	Combining the typical and exceptional parts gives
	\[
	        \E\,\hadj(L)=
	        c_L\,\frac{3N}{n}+a_*+O(N^{-1}).
	\]
	Since $\E|D_1|\,|D_2|=\E\,\hadj(L)$, this proves the moment transfer formula.

	It remains to pass from the mixed moment to the covariance. By the imported one-edge mean
	expansion,
	\[
	        \E|D_1|=
	        2\sqrt{\frac{N}{\pi n}}
	        \left(1-\frac{n-2}{16N}+O(N^{-2})\right).
	\]
	Squaring this expansion gives
	\[
	(\E|D_1|)^2=\frac{4N}{\pi n}-\frac{n-2}{2\pi n}+
        O(N^{-1}).
	\]
	Therefore
	\begin{align*}
	        \Cov(|D_1|,|D_2|)
	        &=
	        \E|D_1|\,|D_2|-(\E|D_1|)^2\\
	        &=
	        \left(\frac{3c_LN}{n}+a_*+O(N^{-1})\right)
	        -
	        \left(\frac{4N}{\pi n}
	        -
	        \frac{n-2}{2\pi n}
	        +
	        O(N^{-1})\right)\\
	        &=
	        \frac{N}{n}\left(3c_L-\frac4\pi\right)
	        +
	        a_*+\frac{n-2}{2\pi n}
	        +
	        O(N^{-1}).
	\end{align*}
	Finally,
	\[
	        3c_L-\frac4\pi
	        =
	        3\left(\frac{2\sqrt3}{3\pi}+\frac19\right)-\frac4\pi
	        =
	        \frac{2\sqrt3-4}{\pi}+\frac13
	        =
	        2\left(\frac{\sqrt3-2}{\pi}+\frac16\right)
	        =
	        2B.
	\]
	Substituting this identity yields
	\[
	        \Cov(|D_1|,|D_2|)
	        =
	        \frac{2N}{n}B+
	        \left(a_*+\frac{n-2}{2\pi n}\right)
	        +
	        O(N^{-1}),
	\]
	as claimed.
\end{proof}

	Everything in \eqref{eq:cov-adj-struct} is now known except $a_*$.  Equivalently,
	\eqref{eq:hadj-form} identifies it as
	\begin{equation}\label{eq:astar-limit}
		a_*=\lim_{\ell\to\infty}\big[\hadj(\ell)-c_L\,\ell\big].
	\end{equation}
	The rest of the adjacent-pair analysis evaluates this limit.

\subsection{The adjacent-pair constant}

	The limit can be checked numerically: Richardson extrapolation of $\hadj(\ell)-c_L\ell$
	from exact values of $\hadj(\ell)$, $\ell\le2100$, within a fixed residue class $\ell\bmod3$,
	reproduces $-2\sqrt3/(9\pi)$ to more than $40$ significant digits. The residue-class
	qualification is not cosmetic: the conditional law is supported on a proper lattice whose
	geometry depends on $\ell\bmod3$, so the $O(1/\ell)$ term oscillates with the class even
	though the constant is common (Lemma~\ref{lem:atriangle-existence}). This section and the
	next \emph{prove} the value.

\begin{theorem}[The adjacent-pair constant]\label{thm:astar}
	For the conditional three-bin model \eqref{eq:hadj-def}, as $\ell\to\infty$,
	\begin{equation}
	\label{eq:three-bin-l}
		\hadj(\ell)=c_L\,\ell+a_*+O(\ell^{-1}),
	\end{equation}
	where
	\begin{equation}
	\label{eq:three-bin-constants}
		c_L=\frac{2\sqrt3}{3\pi}+\frac19,\qquad
		a_*=-\frac{2\sqrt3}{9\pi}=-0.12251753231595\ldots .
	\end{equation}
\end{theorem}

\subsection{Detailed proof of Theorem~\ref{thm:astar}: scope and roadmap}
	\label{ssec:astar-proof-roadmap}

	The leading coefficient $c_L$ is already visible at the Gaussian level
	(Nabeya). The difficulty is the constant correction $a_*$. Since
	$\hadj(\ell)=c_L\ell+O(1)$, the constant lies one order below the central-limit
	scale and cannot be recovered from the Gaussian approximation alone; to reach it we
	need the first non-Gaussian terms of the lattice Edgeworth expansion. The proof
	splits into two independent parts:
	\[
		a_*=\asm+\alatt,
	\]
	a \emph{smooth} part $\asm$ coming from the central-saddle Edgeworth coefficient,
	and a
	genuinely \emph{lattice} part $\alatt$ caused by the fact that $(D_1,D_2)$ lives on
	an index-$3$ sublattice rather than on the whole plane. This is the $\rho=-\tfrac12$
	analogue of the disjoint theorem ($\rho=0$); the new features are that the third
	cumulants no longer vanish, and that the increment lattice is now proper.
	The present section sets up the conditional random walk and the central Edgeworth
	polynomials.  The next section evaluates the two constants $\asm$ and
	$\alatt$ and assembles them (\S\ref{ssec:adj-assembly}).

\medskip
	The proof proceeds through seven logically separate stages.
	\begin{enumerate}
		\item The three-bin increment law is written explicitly and all cumulants through
		order four are computed in \S\ref{ssec:adj-cumulants}, especially
		\eqref{eq:cumulants}.
		\item The central characteristic function is expanded through the order that can
		contribute a constant after multiplication by $\ell$ in
		\S\ref{ssec:adj-edgeworth}, with the cumulant polynomials
		\eqref{eq:A2}--\eqref{eq:P2}.
		\item The Gaussian term gives $c_L\ell$, while parity eliminates the potentially
		larger $\sqrt\ell$ correction; these are Lemmas~\ref{lem:lead} and
		\ref{lem:nosqrt}.
		\item The order-$\ell^{-1}$ smooth Edgeworth coefficient is paired with the kink
		weight.  The fourth-cumulant block cancels and the squared-third-cumulant block gives
		$\asm=-\sqrt3/(18\pi)$; the evaluation is in \S\ref{ssec:adj-smooth}, with the
		formal steps in Lemmas~\ref{lem:Jab} and \ref{lem:price}.
		\item The support is identified as a coset of an index-$3$ sublattice.  A real-space
		calculation explains the expected corner correction but is explicitly labeled
		heuristic in \S\ref{ssec:adj-lattice}, especially \eqref{eq:corner-id} and
		\eqref{eq:alatt}.
		\item The exact finite-part Fourier representation is localized at its three dual
		saddles.  Uniform derivative estimates justify pairing the central expansion with
		the singular tensor kernel, and the two one-axis singular pieces give
		$\alatt=-\sqrt3/(6\pi)$; this is carried out in \S\ref{ssec:adj-dual}, using
		Lemmas~\ref{lem:continuous-abs-kernel}, \ref{lem:tensor-fp-continuity},
		\ref{lem:c2-remainder}, \ref{lem:central-singular}, \ref{lem:fpgauss}, and
		\ref{lem:corner-rig}.
		\item The origin and twisted saddles are assembled, giving
		$a_*=\asm+\alatt=-2\sqrt3/(9\pi)$ with an $O(\ell^{-1})$ remainder in
		\S\ref{ssec:adj-assembly}, the proof of Theorem~\ref{thm:astar}.
	\end{enumerate}

	Bhattacharya--Rao~\cite{bhattacharya_rao} give the lattice-Edgeworth background; the
	distributional pairing of asymptotic expansions with finite-part kernels follows
	Estrada--Kanwal~\cite{estrada_kanwal}.

\subsection{The increment and its cumulants}
	\label{ssec:adj-cumulants}

	Work inside the conditional three-bin model
	$(X_1,X_2,X_3)\sim\Mult(\ell;\tfrac13,\tfrac13,\tfrac13)$. The trial outcome
	$Z\in\{1,2,3\}$ contributes to $(Y_1,Y_2)=(\mathbf1[Z{=}2]-\mathbf1[Z{=}1],\,
	\mathbf1[Z{=}3]-\mathbf1[Z{=}2])$ one of the three equiprobable values
	\[
		Z=1:\,(-1,0),\qquad Z=2:\,(1,-1),\qquad Z=3:\,(0,1),
	\]
	and $(D_1,D_2)=\sum_{t=1}^\ell(Y_1^{(t)},Y_2^{(t)})$ is a sum of i.i.d.\ lattice
	vectors. Since the three values are equiprobable, every mixed moment is the mean of
	its three values; in particular both means vanish, and the cumulants
	$\kappa_{ab}$ of $(Y_1,Y_2)$ (order $a$ in $Y_1$, $b$ in $Y_2$) are computed from
	the moment-to-cumulant formulas for centered variables. Although elementary, we
	record the steps so the table can be read at a glance.
	This single-step description is the reason ordinary Edgeworth theory applies:
	the conditional vector $(D_1,D_2)$ is not an abstract multinomial object any more,
	but the sum of $\ell$ independent copies of one explicitly listed lattice vector.
	All constants below are therefore determined by the cumulants of this one-step
	law.

	\emph{Order two.} $\E Y_1^2=\tfrac13(1+1+0)=\tfrac23$, likewise $\E Y_2^2=\tfrac23$,
	and $\E Y_1Y_2=\tfrac13(0-1+0)=-\tfrac13$; so $\kappa_{20}=\kappa_{02}=\tfrac23$,
	$\kappa_{11}=-\tfrac13$.

	\emph{Order three} (equal to third raw moments here):
	$\E Y_1^3=0$, $\E Y_1^2Y_2=\tfrac13(0-1+0)=-\tfrac13$,
	$\E Y_1Y_2^2=\tfrac13(0+1+0)=\tfrac13$; so $\kappa_{30}=\kappa_{03}=0$,
	$\kappa_{21}=-\tfrac13$, $\kappa_{12}=\tfrac13$.

	\emph{Order four}, via $\kappa_{40}=\E Y_1^4-3(\E Y_1^2)^2$,
	$\kappa_{31}=\E Y_1^3Y_2-3\E Y_1^2\E Y_1Y_2$,
	$\kappa_{22}=\E Y_1^2Y_2^2-\E Y_1^2\E Y_2^2-2(\E Y_1Y_2)^2$:
	\[
		\kappa_{40}=\tfrac23-3\big(\tfrac23\big)^2=-\tfrac23,\quad
		\kappa_{31}=-\tfrac13-3\cdot\tfrac23\big(-\tfrac13\big)=\tfrac13,\quad
		\kappa_{22}=\tfrac13-\tfrac49-\tfrac29=-\tfrac13.
	\]
	Collecting, with $\kappa_{04}=\kappa_{40}$, $\kappa_{13}=\kappa_{31}$ by symmetry,
	\begin{equation}\label{eq:cumulants}
	\begin{aligned}
		&\kappa_{20}=\kappa_{02}=\tfrac23,\quad \kappa_{11}=-\tfrac13;\qquad
		\kappa_{30}=\kappa_{03}=0,\quad \kappa_{21}=-\tfrac13,\quad \kappa_{12}=\tfrac13;\\
		&\kappa_{40}=\kappa_{04}=-\tfrac23,\quad \kappa_{31}=\kappa_{13}=\tfrac13,\quad
		\kappa_{22}=-\tfrac13 .
	\end{aligned}
	\end{equation}
	Thus $\Cov(D_1,D_2)=\ell\,C$ with
	$C=\left(\begin{smallmatrix}2/3&-1/3\\-1/3&2/3\end{smallmatrix}\right)$,
	$\Var(D_i)=\tfrac{2\ell}3$, $\rho=-\tfrac12$, and
	\begin{equation}\label{eq:Cinv}
		\det C=\tfrac13,\qquad C^{-1}=\begin{pmatrix}2&1\\1&2\end{pmatrix}.
	\end{equation}
	The density of the centered Gaussian limit with covariance matrix $C$ will be denoted
	by
	\[
		\phi_C(w)=\frac{\sqrt3}{2\pi}e^{-(w_1^2+w_1w_2+w_2^2)} .
	\]

\subsection{Reduction to a bivariate lattice Edgeworth correction}
	\label{ssec:adj-edgeworth}

	Put $W=(D_1,D_2)/\sqrt\ell$, so $\hadj(\ell)=\ell\,\E\,|W_1||W_2|$.
	Because $W$ is lattice-valued, it has no ordinary density on $\R^2$.  We use the
	polynomials $q_1,q_2$ below only for the \emph{smooth central-saddle part} of the
	lattice Edgeworth expansion.  The exact Fourier argument in
	\S\ref{ssec:adj-dual} separately recovers the lattice correction and controls the
	remainder.  This distinction matters here because the test function
	$|w_1||w_2|$ is not smooth on the coordinate axes.
	The scaling explains which terms can affect $a_*$.  If the law of $W$ is expanded
	as a Gaussian term plus corrections of sizes $\ell^{-1/2}$ and $\ell^{-1}$, then
	multiplication by the outside factor $\ell$ turns these into contributions of
	sizes $\ell$, $\sqrt\ell$, and $1$.  Hence the constant $a_*$ can only come from
	the order-$\ell^{-1}$ central term and from lattice effects of the same size,
	after the order-$\ell^{-1/2}$ contribution has been shown to vanish.

	Let $K(u)=\log\E\,e^{\,iu\cdot Y}=\sum_{r\ge2}\tfrac{i^r}{r!}A_r(u)$ be the
	per-trial cumulant generating function, with
	$A_r(u)=\sum_{|a|=r}\binom ra\kappa_a u^a$. Substituting \eqref{eq:cumulants},
	\begin{align}
		A_2(u)&=\tfrac23u_1^2-\tfrac23u_1u_2+\tfrac23u_2^2=u^\top C\,u,\label{eq:A2}\\
		A_3(u)&=3\kappa_{21}u_1^2u_2+3\kappa_{12}u_1u_2^2=-u_1^2u_2+u_1u_2^2,\label{eq:A3}\\
		A_4(u)&=-\tfrac23u_1^4+\tfrac43u_1^3u_2-2u_1^2u_2^2+\tfrac43u_1u_2^3-\tfrac23u_2^4.
		\label{eq:A4}
	\end{align}
	With $u=t/\sqrt\ell$ and $A_r$ homogeneous of degree $r$, the characteristic
	function of $W$ has the central expansion
	\[
		\E e^{it\cdot W}=\exp\{\ell K(t/\sqrt\ell)\}
		=e^{-\frac12 t^\top Ct}\big(1+\ell^{-1/2}P_1+\ell^{-1}P_2+\cdots\big),
	\]
	where
	$P_1=\tfrac{i^3}{6}A_3$ and
	\begin{equation}\label{eq:P2}
		P_2=\frac{i^4}{24}A_4+\frac12P_1^2=\frac1{24}A_4-\frac1{72}A_3^2,
	\end{equation}
	a fourth-cumulant (Hermite-degree-$4$) part and a squared-third-cumulant
	(Hermite-degree-$6$) part.  Indeed, the quadratic term $i^2A_2/2=-A_2/2$ gives the
	Gaussian factor $e^{-t^\top Ct/2}$; the cubic term is of size $\ell^{-1/2}$; and the
	quartic term, together with the square of the cubic term after exponentiation, is of
	size $\ell^{-1}$.  The second summand of $P_2$ is present because the exponential of the
	cubic term must itself be expanded to second order: although $A_3$ enters at size
	$\ell^{-1/2}$, its square enters at size $\ell^{-1}$ and therefore contributes to the
	constant term.
	
	Under characteristic-function
	inversion, $(2\pi)^{-2}\int e^{-it\cdot w}(\cdot)\,dt$, multiplication by
	$(it_1)^a(it_2)^b$ corresponds to
	$(-1)^{a+b}\partial_{w_1}^a\partial_{w_2}^b$.  To apply this to $P_1=\tfrac{i^3}6A_3$,
	rewrite $A_3$ in these monomials: since $t_1^2t_2=i\,(it_1)^2(it_2)$ and
	$t_1t_2^2=i\,(it_1)(it_2)^2$,
	\[
		P_1=\frac{i^3}6A_3=\frac{i^4}6\big[-(it_1)^2(it_2)+(it_1)(it_2)^2\big]
		=\frac16\big[-(it_1)^2(it_2)+(it_1)(it_2)^2\big],
	\]
	and inverting the two monomials term by term,
	\[
		(it_1)^2(it_2)\mapsto-\partial_{w_1}^2\partial_{w_2},\qquad
		(it_1)(it_2)^2\mapsto-\partial_{w_1}\partial_{w_2}^2,
	\]
	gives
	\begin{equation}\label{eq:q1}
		q_1\phi_C=\tfrac16\big(\partial_{w_1}^2\partial_{w_2}-\partial_{w_1}\partial_{w_2}^2\big)\phi_C .
	\end{equation}
	At the level of the smooth central-saddle coefficients, multiplication by
	$\ell|w_1||w_2|$ and inversion therefore gives
	\begin{equation}\label{eq:expand}
		\mathcal E_{\rm sm}^{[2]}(\ell):=
		\underbrace{\ell\!\int|w_1||w_2|\phi_C}_{=:c_L\ell}
		+\sqrt\ell\!\int|w_1||w_2|q_1\phi_C
		+\int|w_1||w_2|q_2\phi_C .
	\end{equation}
	Here $\mathcal E_{\rm sm}^{[2]}$ denotes the smooth expansion truncated after its
	constant coefficient; it is not by itself an expansion of the lattice expectation
	$\hadj(\ell)$.  The missing lattice term and
	the $O(\ell^{-1})$ remainder are supplied by Lemmas~\ref{lem:central-singular} and
	\ref{lem:corner-rig}.

\section{Adjacent pairs: smooth and lattice constants}
	\label{sec:adj-smooth-lattice}

	This section turns the roadmap of the preceding section into the actual value of
	$a_*$.  There are two different mechanisms at the same asymptotic order.  The first
	is the usual central Edgeworth correction, computed as if the normalized law had a
	smooth density near the origin.  The second is a lattice correction: even after
	normalization, the exact mass lies on one coset of an index-$3$ lattice, and the
	test function $|w_1||w_2|$ has kinks on the coordinate axes.  The smooth calculation
	gives $\asm$; the lattice calculation gives $\alatt$; their sum is the desired
	constant.

	Since several parallel families of symbols are in play, we collect them once.
	\begin{center}\small
	\begin{tabular}{ll}
		$A_r(u)$ & homogeneous cumulant form of degree $r$ (from $K=\log\Phi$), \eqref{eq:A2}--\eqref{eq:A4}\\
		$B_r(u)$ & homogeneous degree-$r$ term of $\ell K(u/\sqrt\ell)$ (so $B_r=\tfrac{i^r}{r!}A_r$)\\
		$P_k(u)$ & Edgeworth polynomial on the characteristic-function side, \eqref{eq:P2}\\
		$q_k(w)$ & its real-space counterpart: $q_k\phi_C=$ inverse transform of $e^{-\frac12u^\top Cu}P_k$, \eqref{eq:q1}\\
		$g_k(w)$ & the same object in Lemma~\ref{lem:dual-real} ($g_1=q_1$, $g_2=q_2$)
	\end{tabular}
	\end{center}

	For orientation, it is useful to separate three layers that are interlaced in the
	proof.  First, the Gaussian leading term gives only $c_L\ell$, and the possible
	$\sqrt\ell$ correction vanishes by parity.  Second, the ordinary central
	Edgeworth calculation produces a genuine constant, denoted by $\asm$.  This is
	the part that would remain if the normalized law could be treated as a smooth
	two-dimensional density.  Third, the lattice support is not the whole integer
	plane but a single coset of an index-$3$ sublattice.  This discreteness is invisible
	to smooth test functions at constant order, but it is visible to
	$|w_1||w_2|$, because the two absolute values have kinks on the coordinate axes.
	The resulting corner correction is $\alatt$.  Thus the section proves
	\[
		\hadj(\ell)=c_L\ell+\asm+\alatt+O(\ell^{-1}),
		\qquad
		a_*=\asm+\alatt .
	\]
	The reason for the relatively long Fourier-domain closure is not the value of the
	Gaussian integrals themselves, but the need to justify that the finite-part
	calculus, the residue class modulo $3$, and the truncated Edgeworth expansion can
	be combined uniformly with only an $O(\ell^{-1})$ remainder.

\subsection{The leading term}
	\label{ssec:adj-lead}

\begin{lemma}[Gaussian leading term]\label{lem:lead}
	$\ell\int_{\R^2}|w_1||w_2|\phi_C\,dw=c_L\ell$ exactly, with
	$c_L=\tfrac{2\sqrt3}{3\pi}+\tfrac19$.
\end{lemma}

\begin{proof}
	By Lemma~\ref{lem:sheppard-mardia} with $\sigma^2=\tfrac23$, $\rho=-\tfrac12$
	(so $\sqrt{1-\rho^2}=\tfrac{\sqrt3}2$, $\rho\arcsin\rho=+\tfrac\pi{12}$),
	\[
		\E|W_1||W_2|=\frac2\pi\cdot\frac23\Big(\frac{\sqrt3}2+\frac\pi{12}\Big)
		=\frac{2\sqrt3}{3\pi}+\frac19=c_L,
	\]
	exactly; the Gaussian term therefore contributes $c_L\ell$ and nothing to $a_*$.
\end{proof}

	The entire constant $a_*$ is thus a discreteness/Edgeworth correction.

	In a generic Edgeworth expansion the first correction after the Gaussian is of
	order $\ell^{-1/2}$. Since $\hadj(\ell)=\ell\,\E|W_1||W_2|$, such a term would
	produce a contribution of order $\sqrt\ell$ to the expansion --- larger than the
	constant we are after. Before searching for the constant we must therefore verify
	that this $\sqrt\ell$ coefficient vanishes; the next lemma shows the cancellation is
	a consequence of the parity of the cubic Edgeworth polynomial.

\begin{lemma}[Vanishing square-root correction]\label{lem:nosqrt}
	$\int_{\R^2}|w_1||w_2|\,q_1\phi_C\,dw=0$.
\end{lemma}

\begin{proof}
	By \eqref{eq:q1} it suffices to treat $\int|w_1||w_2|\partial_{w_1}^2\partial_{w_2}\phi_C\,dw$.
	Integrating by parts three times (factor $(-1)^3$, no boundary terms) and using the
	distributional kink identities $\partial_w|w|=\sgn w$, $\partial_w^2|w|=2\delta(w)$,
	\[
	\begin{aligned}
		\int|w_1||w_2|\partial_{w_1}^2\partial_{w_2}\phi_C\,dw
		&=-\int 2\delta(w_1)\sgn(w_2)\phi_C\,dw\\
		&=-2\!\int\sgn(w_2)\phi_C(0,w_2)\,dw_2 .
	\end{aligned}
	\]
	By \eqref{eq:Cinv}, $\phi_C(0,w_2)=\tfrac{\sqrt3}{2\pi}e^{-w_2^2}$ is even and
	$\sgn(w_2)$ is odd, so the integral is $0$; the mirror term vanishes likewise.
\end{proof}

\subsection{The smooth $O(1)$ coefficient}
	\label{ssec:adj-smooth}

	Of the three contributions of the smooth expansion \eqref{eq:expand}, the leading one is
	Lemma~\ref{lem:lead} and the $\sqrt\ell$ term vanishes by parity (Lemma~\ref{lem:nosqrt});
	so the only smooth contribution to the constant is $\asm$, the $q_2$ term of
	\eqref{eq:expand}, evaluated below through Lemmas~\ref{lem:Jab} and~\ref{lem:price}.

	The adjacent pair differs fundamentally from the disjoint pair of
	\S\ref{ssec:Cov-disj-asymp-proof}. There the constant-order residue was $\kappa_{2,2}$ times
	a lattice origin probability, with no third-cumulant term by symmetry. Here the mixed
	third cumulants $\kappa_{21}=-\tfrac13$, $\kappa_{12}=+\tfrac13$ are nonzero, so by
	\eqref{eq:P2} the second Edgeworth coefficient contains both $\tfrac1{24}A_4$ and the
	quadratic term $-\tfrac1{72}A_3^2$. A priori both may contribute; it turns out that the
	whole fourth-cumulant part cancels, and the smooth correction is carried entirely by the
	new degree-$6$ structure of $A_3^2$.

	Define $\asm:=\int_{\R^2}|w_1||w_2|\,q_2\phi_C\,dw$ where $q_2\phi_C$ is the inverse
	transform of $P_2\,e^{-\frac12 t^\top Ct}$. Writing $P_2=\sum c_{ab}(it_1)^a(it_2)^b$
	and integrating by parts (the sign $(-1)^{a+b}=1$, since $a+b\in\{4,6\}$),
	\begin{equation}\label{eq:Iab}
		\asm=\sum_{(a,b)}c_{ab}\,I_{ab},\qquad
		I_{ab}=\int_{\R^2}\big(\partial_{w_1}^a|w_1|\big)\big(\partial_{w_2}^b|w_2|\big)\phi_C\,dw .
	\end{equation}
	Thus the problem has been reduced to a finite table of Gaussian pairings.  The
	derivatives of $|w|$ are distributions concentrated at the kink once the order is
	at least two, so these integrals are not ordinary integration-by-parts identities
	in the classical sense.  They are finite-part/distributional identities, justified
	by the regularization stated earlier; the advantage is that the final numbers are
	values and derivatives of a Gaussian density on the coordinate axes.

	With $w=\sigma u$, $\sigma^2=\tfrac23$, one has $\partial_w^a|w|=\sigma^{1-a}\partial_u^a|u|$
	and $\phi_C(w)=\sigma^{-2}\phi_\rho(u)$, where
	\begin{equation}\label{eq:phirho}
		\phi_\rho(u)=\frac{1}{\pi\sqrt3}\,e^{-\frac23(u_1^2+u_1u_2+u_2^2)}
	\end{equation}
	is the unit-variance normal with $\rho=-\tfrac12$. Hence $I_{ab}=\sigma^{2-a-b}J_{ab}$, where
	\[
		J_{ab}=\int(\partial^a|u_1|)(\partial^b|u_2|)\phi_\rho\,du;
	\]
	for $a+b=4$, $I_{ab}=\tfrac32 J_{ab}$, and for $a+b=6$, $I_{ab}=\tfrac94 J_{ab}$.

	From \eqref{eq:A4}, on degree $4$ (where $(it_1)^a(it_2)^b=t_1^at_2^b$) the
	coefficients of $\tfrac1{24}A_4$ are
	$c_{40}=c_{04}=-\tfrac1{36}$, $c_{31}=c_{13}=\tfrac1{18}$, $c_{22}=-\tfrac1{12}$.
	From \eqref{eq:A3}, $A_3^2=t_1^4t_2^2-2t_1^3t_2^3+t_1^2t_2^4$, and on degree $6$
	(where $(it_1)^a(it_2)^b=-t_1^at_2^b$) the coefficients of $-\tfrac1{72}A_3^2$ are
	$c_{42}=c_{24}=+\tfrac1{72}$, $c_{33}=-\tfrac1{36}$.

	For the unit-variance normal with correlation $r$, Price's theorem
	\cite{price1958} gives
	\[
		\partial_r^k\E_r[g]=\E_r[\partial_{u_1}^k\partial_{u_2}^k g].
	\]
	Applied to
	$g=|u_1||u_2|$ and $G(r)=\E_r|u_1||u_2|=\tfrac2\pi(\sqrt{1-r^2}+r\arcsin r)$,
	$J_{kk}=\partial_r^kG|_{r=-1/2}$.  This distributional form is justified by the
	regularization of Remark~\ref{rem:regularization}; the values are also computed
	directly in Lemma~\ref{lem:Jab}. Since
	$\partial_rG=\tfrac2\pi\arcsin r$,
	\begin{equation}\label{eq:Gder}
		\partial_rG\big|_{r=-1/2}=-\tfrac13,\qquad
		\partial_r^2G\big|_{r=-1/2}=\frac4{\pi\sqrt3},\qquad
		\partial_r^3G\big|_{r=-1/2}=-\frac{8\sqrt3}{9\pi},
	\end{equation}
	which we recover below as $J_{11},J_{22},J_{33}$.

\begin{lemma}[The kink integrals]\label{lem:Jab}
	With $\phi_\rho$ as in \eqref{eq:phirho},
	\[
		J_{22}=\frac4{\pi\sqrt3},\quad J_{31}=\frac2{\pi\sqrt3},\quad
		J_{40}=-\frac2{\pi\sqrt3},
	\]
	\[
		J_{33}=-\frac8{3\pi\sqrt3},\qquad J_{42}=J_{24}=-\frac{16}{3\pi\sqrt3}.
	\]
\end{lemma}

\begin{proof}
	Use the distributional identities $\partial^2|u|=2\delta$, $\partial^3|u|=2\delta'$,
	$\partial^4|u|=2\delta''$ (in the regularized sense of
	Remark~\ref{rem:regularization}) and $\phi_\rho(0,0)=\tfrac1{\pi\sqrt3}$.
	Each derivative beyond the first moves mass onto an axis; if both variables are
	differentiated at least twice, only the value or a derivative of $\phi_\rho$ at the
	origin remains.  The two mixed cases below keep one absolute value uncollapsed and
	are therefore computed by conditioning one normal coordinate on the other.

	\emph{Even/even.}
	$J_{22}=\int2\delta(u_1)2\delta(u_2)\phi_\rho=4\phi_\rho(0,0)=\tfrac4{\pi\sqrt3}$;
	$J_{42}=4\,\partial_{u_1}^2\phi_\rho(u_1,0)|_0$, and with
	$\phi_\rho(u_1,0)=\tfrac1{\pi\sqrt3}e^{-\frac23u_1^2}$ (so the second derivative at
	$0$ is $-\tfrac43\cdot\tfrac1{\pi\sqrt3}$), $J_{42}=-\tfrac{16}{3\pi\sqrt3}$. We spell
	this degree-$6$ case out through the mollifier of
	Remark~\ref{rem:regularization} once, as a template for the others: replacing each
	$|u_i|$ by $m_\varepsilon$ gives
	\[
		J_{42}^{(\varepsilon)}=\iint m_\varepsilon^{(4)}(u_1)\,m_\varepsilon''(u_2)\,
		\phi_\rho(u_1,u_2)\,du,
	\]
	and since $m_\varepsilon''=2\eta_\varepsilon$, integrating by parts twice in $u_1$ gives
	\[
		J_{42}^{(\varepsilon)}=4\iint\eta_\varepsilon(u_1)\eta_\varepsilon(u_2)\,
		\partial_{u_1}^2\phi_\rho(u_1,u_2)\,du
		\;\longrightarrow\;4\,\partial_{u_1}^2\phi_\rho(0,0)=-\tfrac{16}{3\pi\sqrt3}=J_{42}
	\]
	as $\varepsilon\downarrow0$, by continuity of $\partial_{u_1}^2\phi_\rho$ (equivalently, by
	dominated convergence applied to the pairing of $m_\varepsilon(u_1)m_\varepsilon(u_2)$ with
	$\partial_{u_1}^4\partial_{u_2}^2\phi_\rho$, with the majorant
	$(|u_1|+\varepsilon)(|u_2|+\varepsilon)$).

	\emph{Odd/odd.} $J_{33}=4\,\partial_{u_1}\partial_{u_2}\phi_\rho(0,0)$; from
	$\partial_{u_2}\phi_\rho=-\tfrac23(u_1+2u_2)\phi_\rho$ one gets
	$\partial_{u_1}\partial_{u_2}\phi_\rho(0,0)=-\tfrac23\phi_\rho(0,0)$, so
	$J_{33}=-\tfrac8{3\pi\sqrt3}=-\tfrac{8\sqrt3}{9\pi}=\partial_r^3G$, as predicted.

	\emph{Mixed $J_{40}$.}  This is the first place where the calculation is not just
	an evaluation of $\phi_\rho$ at the origin.  Since only the first variable is
	differentiated four times,
	\[
		\partial_{u_1}^4|u_1|=2\delta''(u_1),\qquad \partial_{u_2}^0|u_2|=|u_2|,
	\]
	and hence
	\[
		J_{40}
		=2\int_{\R}|u_2|\,\partial_{u_1}^2\phi_\rho(0,u_2)\,du_2.
	\]
	Equivalently, if
	\[
		g(u_1):=\int_{\R}|u_2|\phi_\rho(u_1,u_2)\,du_2,
	\]
	then $J_{40}=2g''(0)$.  The remaining integral is most cleanly evaluated by
	conditioning one Gaussian coordinate on the other.  Factor
	$\phi_\rho(u_1,u_2)=\phi(u_1)p(u_2\mid u_1)$, where $\phi$ is the standard normal
	density and, for correlation $\rho=-1/2$,
	\[
		u_2\mid u_1\sim N(m,s^2),\qquad m=-\frac{u_1}{2},\qquad s^2=\frac34.
	\]
	Thus
	\[
		g(u_1)=\phi(u_1)\Psi(m(u_1)),\qquad
		\Psi(m):=\E|N(m,s^2)|.
	\]
	For the folded-normal mean~\cite{leone1961,kanrobotti2017},
	\[
		\Psi'(m)=\E\,\sgn(m+sZ)=2\Phi(m/s)-1,\qquad
		\Psi''(0)=\frac1s\sqrt{\frac2\pi},\qquad
		\Psi(0)=s\sqrt{\frac2\pi}.
	\]
	Differentiating $g(u_1)=\phi(u_1)\Psi(-u_1/2)$ twice, the mixed first-derivative
	terms vanish at $0$ because $\phi'(0)=0$ and $\Psi'(0)=0$.  Therefore
	\[
		g''(0)=\phi''(0)\Psi(0)+\phi(0)\frac14\Psi''(0).
	\]
	Using $\phi(0)=1/\sqrt{2\pi}$ and $\phi''(0)=-\phi(0)$,
	\[
		g''(0)=\frac1{\sqrt{2\pi}}\Big(\frac14\Psi''(0)-\Psi(0)\Big)
		=\frac1\pi\Big(\frac1{4s}-s\Big)
		=-\frac1{\pi\sqrt3}\qquad \left(s=\frac{\sqrt3}{2}\right).
	\]
	Hence $J_{40}=2g''(0)=-\tfrac2{\pi\sqrt3}$.

	\emph{Mixed $J_{31}=-2h'(0)$}, $h(u_1)=\int\sgn(u_2)\phi_\rho\,du_2
	=\phi(u_1)(2\Phi(m/s)-1)$. Since the second factor vanishes at $u_1=0$ and
	$\phi'(0)=0$, $h'(0)=\phi(0)\cdot(-\tfrac1{s\sqrt{2\pi}})=-\tfrac1{2\pi s}$, so
	$J_{31}=\tfrac1{\pi s}=\tfrac2{\pi\sqrt3}$.
\end{proof}

\begin{lemma}[Smooth constant]\label{lem:price}
	The fourth-cumulant part of \eqref{eq:Iab} vanishes and the squared-third-cumulant
	part carries the whole:
	\[
		\sum_{a+b=4}c_{ab}I_{ab}=0,\qquad \sum_{a+b=6}c_{ab}I_{ab}=-\frac{\sqrt3}{18\pi},
		\qquad\text{so}\qquad \asm=-\frac{\sqrt3}{18\pi}.
	\]
\end{lemma}

\begin{proof}
	\emph{Degree $4$} ($I_{ab}=\tfrac32J_{ab}$), with $J_{40}=J_{04}$, $J_{31}=J_{13}$:
	\begin{align*}
		\sum_{a+b=4}c_{ab}I_{ab}
		&=\tfrac32\Big[-\tfrac1{18}J_{40}+\tfrac19J_{31}-\tfrac1{12}J_{22}\Big]\\
		&=\tfrac32\cdot\frac1{\pi\sqrt3}
		\Big[-\tfrac1{18}(-2)+\tfrac19(2)-\tfrac1{12}(4)\Big]\\
		&=\tfrac32\cdot\frac1{\pi\sqrt3}\Big[\tfrac19+\tfrac29-\tfrac13\Big]=0 .
	\end{align*}
	\emph{Degree $6$} ($I_{ab}=\tfrac94J_{ab}$; the two equal terms $J_{42}=J_{24}$
	combine, $2c_{42}=\tfrac1{36}$):
	\begin{align*}
		\sum_{a+b=6}c_{ab}I_{ab}&=\tfrac94\Big[\tfrac1{36}J_{42}-\tfrac1{36}J_{33}\Big]
		=\tfrac94\cdot\tfrac1{36}\big(J_{42}-J_{33}\big)\\
		&=\tfrac94\cdot\tfrac1{36}\cdot\Big(-\tfrac8{3\pi\sqrt3}\Big)
		=-\frac1{6\pi\sqrt3}=-\frac{\sqrt3}{18\pi}.
	\end{align*}
	Hence $\asm=-\tfrac{\sqrt3}{18\pi}=-0.0306293830\ldots$.
\end{proof}

\subsection{The lattice corner correction: real-space guide}
	\label{ssec:adj-lattice}

	The smooth calculation of the preceding subsections deliberately ignored one
	piece of information: the exact pair $(D_1,D_2)$ is not distributed over all
	integer points in the plane.  It lives on a thinner lattice.  Indeed, from
	\[
		D_1-D_2=(X_2-X_1)-(X_3-X_2)=3X_2-\ell\equiv-\ell\pmod3,
	\]
	only one residue class mod $3$ is occupied: the support is a coset of the index-$3$
	sublattice
	\[
		\Lambda=\{d\in\Z^2:d_1\equiv d_2\pmod 3\},
		\qquad \operatorname{covol}(\Lambda)=v=3.
	\]
	The factor $v$ is the density correction: one point of the coset represents an
	area of size $3$, so a smooth integral over the plane is approximated by
	$v$ times the sum over the occupied coset.

	For a smooth test function this distinction would not affect the constant term.
	After the Gaussian scaling the mesh size is $O(\ell^{-1/2})$, and the usual
	Riemann-sum or Poisson-summation error is lower order.  The function used here,
	however, is not smooth:
	\[
		(d_1,d_2)\longmapsto |d_1|\,|d_2|.
	\]
	It has a kink whenever $d_1=0$ and another kink whenever $d_2=0$.  These two
	axes are the only places where the continuum approximation loses a constant
	amount.  The lattice constant $\alatt$ is exactly the sum of the two one-axis
	defects.

	This subsection is meant only to show where that constant comes from.  We replace
	the exact lattice law on the occupied coset $\Lambda+x_0$ by its leading local
	Gaussian approximation $v\,\phi_{\ell C}$.  This gives the model sum
	\[
		S:=\sum_{d\in\Lambda+x_0}|d_1||d_2|\phi_{\ell C}(d),
	\]
	and the continuum counterpart of this model is
	\[
		vS\approx \int_{\R^2}|d_1||d_2|\,\phi_{\ell C}(d)\,dd=c_L\ell .
	\]
	Thus the lattice part should be the constant term in $vS-c_L\ell$.  This is not
	an exact representation of $\hadj(\ell)$: it is a real-space guide.  The exact
	identity is the finite-part Fourier integral in \S\ref{ssec:adj-dual}, and
	Lemma~\ref{lem:corner-rig} proves that the constant extracted below is indeed the
	rigorous lattice contribution, with an $O(\ell^{-1})$ remainder.

	There is one bookkeeping nuisance.  The two axes are geometrically symmetric, but
	the coset $\Lambda+x_0$ meets them in different visible ways.  If we sum first
	along lines parallel to the $d_2$-axis, one axis appears with spacing $3$ and an
	average over three residue classes; the other appears with spacing $1$, the
	representative value $B_2(0)$, and the covolume factor $v=3$.  These are only two
	coordinate descriptions of the same index-$3$ support geometry.  The Fourier proof
	in \S\ref{ssec:adj-dual} restores the symmetry and proves that no off-axis or
	residue-class term of constant size has been omitted.

	The only one-dimensional input needed for the model calculation is the Euler--Poisson
	correction caused by a kink.  For a Schwartz function $W$ varying on a scale much
	larger than the lattice spacing $s$, the standard Poisson summation formula
	(see, for instance, Stein--Shakarchi~\cite{stein_shakarchi2003} or
	Katznelson~\cite{katznelson2004}) applied to
	$m\mapsto|m|W(m)$ on the progression $c+s\Z$ gives
	\begin{equation}\label{eq:corner-id}
		\sum_{m\in c+s\Z}|m|\,W(m)=\frac1s\!\int_\R|y|W(y)\,dy-s\,B_2(\{c/s\})\,W(0)+o(1),
	\end{equation}
	with $B_2(x)=x^2-x+\tfrac16$.  The first term is the continuum approximation; the
	second is the constant defect created at the nonsmooth point $m=0$.  Formally, it
	comes from the $j\neq0$ Poisson modes, using
	$\widehat{|y|}(\xi)=-2/\xi^2$ and
	$\sum_{j\neq0}j^{-2}e^{2\pi ijx}=2\pi^2B_2(\{x\})$.  In the exact proof this same
	mechanism reappears as the constant term $s(0)=-1/6$ in the local expansion of the
	Fourier kernel $S(\theta)$.

	Parametrize the coset by $d_1=k$, $d_2\in(k-r)+3\Z$, $r=-\ell\bmod3$.  Applying
	\eqref{eq:corner-id} first in $d_2$ (spacing $3$) and then in $k$ (spacing $1$,
	$B_2(0)=\tfrac16$), and multiplying by the covolume factor $v=3$, one obtains the
	two predicted kink-line defects
	\[
		v\,S-c_L\ell
		=\underbrace{-\tfrac{v}{3}\cdot\tfrac16\,J_{\mathrm{axis}}}_{d_1\text{-kink}}
		\;\underbrace{-\,9\,\overline{B_2}\!\int|x|\phi_{\ell C}(x,0)\,dx}_{d_2\text{-kink}}+o(1),
	\]
	where the periodic factor $B_2(\{\tfrac{k-r}3\})$ has been replaced by its
	residue-average
	\[
		\overline{B_2}
		=\tfrac13\left[B_2(0)+B_2(\tfrac13)+B_2(\tfrac23)\right]
		=\tfrac13\left[\tfrac16-\tfrac1{18}-\tfrac1{18}\right]
		=\tfrac1{54}.
	\]
	Moreover,
	\[
		J_{\mathrm{axis}}=\int|y|\phi_{\ell C}(0,y)\,dy
		=\int|x|\phi_{\ell C}(x,0)\,dx
		=\frac{\sqrt3}{2\pi}.
	\]
	These quantities are $\ell$-independent, as follows from
	$\phi_{\ell C}(0,y)=\tfrac{\sqrt3}{2\pi\ell}e^{-y^2/\ell}$ and
	$\int_0^\infty ye^{-y^2/\ell}dy=\ell/2$. Since $-9\overline{B_2}=-\tfrac16$, each
	kink line contributes $-\tfrac16\cdot\tfrac{\sqrt3}{2\pi}=-\tfrac{\sqrt3}{12\pi}$, so
	\begin{equation}\label{eq:alatt}
		\alatt=-\frac{\sqrt3}{12\pi}-\frac{\sqrt3}{12\pi}=-\frac{\sqrt3}{6\pi}
		=-0.0918881492\ldots .
	\end{equation}
	The formal proof of the statements used in this guide is deliberately postponed.
	The exact replacement for the model sum is the finite-part identity
	\eqref{eq:rep}, introduced in \S\ref{ssec:adj-dual}.  The use of the singular
	absolute-value kernel is justified by Lemma~\ref{lem:periodic-abs-kernel}, its
	local principal part by Lemma~\ref{lem:continuous-abs-kernel}, and the continuity
	needed to pair the Edgeworth remainder with the tensor finite part by
	Lemma~\ref{lem:tensor-fp-continuity}.  The smooth continuum part is evaluated in
	Lemma~\ref{lem:central-singular}, while the two mixed singular-smooth origin
	pieces, which are the rigorous counterparts of the two kink-line defects above,
	are evaluated in Lemma~\ref{lem:corner-rig}.  Finally, the paragraphs on the
	twisted saddles and off-saddle regions in \S\ref{ssec:adj-dual} prove that the
	residue classes and the complement of the origin saddle contribute only
	$O(\ell^{-1})$, so they do not change the constant term \eqref{eq:alatt}.
	This computation is heuristic on one point: for a single residue class the
	$j\neq0$ Poisson tail decays only like $O(\ell^{-1/2})$, so one should not justify
	the replacement by $\overline{B_2}$ term by term in this real-space form.  The next
	two subsections remove this point entirely.  In Fourier variables the residue
	classes become the two twisted saddles, their contribution is only $O(\ell^{-1})$,
	and Lemma~\ref{lem:corner-rig} proves that the two origin-saddle corner pieces give
	\eqref{eq:alatt} with a uniform $O(\ell^{-1})$ remainder.

\subsection{Uniform closure in the dual domain}
	\label{ssec:adj-dual}

	It remains to justify that the formal expansion may be paired with the singular
	kernel and that no hidden constant arises from the residue classes or from the
	complement of the origin saddle; the outcome is that the dual-domain proof accounts
	only for the pure singular origin piece $c_L\ell+\asm$ and the two mixed
	singular-smooth corner pieces summing to $\alatt$, so that
	$\hadj(\ell)=c_L\ell+\asm+\alatt+O(\ell^{-1})$ after the assembly of
	\S\ref{ssec:adj-assembly}.

	We now prove \eqref{eq:alatt} with a uniform $O(\ell^{-1})$ remainder, working in
	Fourier space, where the residue structure becomes the phases of two dual-lattice
	saddles and the slow per-residue tail disappears.
	The price of this cleaner residue bookkeeping is that the absolute values are now
	represented by the singular periodic kernel $S(\theta)$.  The proof is therefore a
	localization argument with finite parts: near the origin the singularity produces
	the continuum term and the lattice corner; near the two twisted saddles the kernel
	is smooth and the phase only contributes at order $\ell^{-1}$; away from all
	saddles the characteristic function is exponentially small.

	\paragraph{Exact representation} The characteristic function of $(D_1,D_2)$ is
	\begin{equation}
	\label{eq:charPhi}
		\Phi(\theta)=\tfrac13(e^{-i\theta_1}+e^{i(\theta_1-\theta_2)}+e^{i\theta_2}),
	\end{equation}
	so $\Phi(\theta)^\ell$ is a trigonometric polynomial.
	Lemma~\ref{lem:periodic-abs-kernel} gives the \emph{exact} identity
	\begin{equation}\label{eq:rep}
		\hadj(\ell)=\frac1{(2\pi)^2}\,\fp\!\iint_{[-\pi,\pi]^2}
		\Phi(\theta)^\ell S(\theta_1)S(\theta_2)\,d\theta .
	\end{equation}
	The Laurent expansion of its kernel at the origin is
	$S(\theta)=-\tfrac2{\theta^2}-\tfrac16+O(\theta^2)$.  Indeed,
	$\sin(\theta/2)=\tfrac{\theta}{2}(1-\theta^2/24+O(\theta^4))$, so
	$\csc^2(\theta/2)=4/\theta^2+1/3+O(\theta^2)$.

	\paragraph{The dual-lattice saddles}
	The large-$\ell$ behavior of \eqref{eq:rep} is controlled by the points where
	$|\Phi|$ is maximal.  Since $\Phi$ is the average of three unit complex numbers,
	one always has $|\Phi(\theta)|\le1$.  Equality can occur only when the three
	phases in
	\[
		e^{-i\theta_1},\qquad e^{i(\theta_1-\theta_2)},\qquad e^{i\theta_2}
	\]
	are equal.  Equivalently,
	\[
		2\theta_1\equiv\theta_2,\qquad
		\theta_1\equiv2\theta_2,\qquad
		\theta_1+\theta_2\equiv0
		\pmod {2\pi}.
	\]
	The first two congruences imply the third and give
	$3\theta_1\equiv0$, $\theta_2\equiv-\theta_1$.  Thus, on the torus
	$[-\pi,\pi]^2$, there are exactly three maxima:
	\[
		Q_0=(0,0),\qquad
		Q_1=\Bigl(\frac{2\pi}3,-\frac{2\pi}3\Bigr),\qquad
		Q_2=\Bigl(-\frac{2\pi}3,\frac{2\pi}3\Bigr).
	\]

	At these points
	\[
		\Phi(Q_0)=1,\qquad
		\Phi(Q_1)=\bar\omega,\qquad
		\Phi(Q_2)=\omega,\qquad
		\omega=e^{2\pi i/3}.
	\]
	Each maximum is non-degenerate.  With the same covariance matrix $C$ as in the
	central limit theorem,
	\[
		\Phi(Q_j+\eta)
		=\Phi(Q_j)\exp\!\left\{-\frac12\eta^\top C\eta+O(|\eta|^3)\right\},
		\qquad \eta\to0,
	\]
	and $\det C=\tfrac13>0$.  In fact the two nonzero saddles are exact translates
	of the origin saddle:
	\[
		\Phi(Q_1+\eta)=\bar\omega\,\Phi(\eta),
		\qquad
		\Phi(Q_2+\eta)=\omega\,\Phi(\eta),
	\]
	with the arguments understood on the torus.

	These two additional saddles are the Fourier-side trace of the index-$3$ support
	lattice.  In real space, the support condition was the congruence
	$d_1-d_2\equiv-\ell\pmod3$.  In Fourier space the same information appears as
	the phases
	\[
		\Phi(Q_j)^\ell\in\{1,\bar\omega^\ell,\omega^\ell\},
	\]
	which remember the residue class of $\ell$ modulo $3$.  This is the dual-domain
	version of the periodic Bernoulli factor
	$B_2(\{(k-r)/3\})$ in the real-space guide \eqref{eq:corner-id}.

	Finally, only small neighborhoods of these three points can contribute at
	algebraic order.  If the three $\delta$-balls are removed, then
	$|\Phi|\le e^{-c}$ on the remaining compact set for some $c>0$.  Pairing the
	tensor finite-part kernel $S(\theta_1)S(\theta_2)$ against a smooth function is
	controlled by a $C^{2,2}$ norm, so this off-saddle region contributes only
	$O(\ell^4e^{-c\ell})$.  All constant terms therefore come from the three local
	saddles described above.

\begin{lemma}[Continuous absolute-value multiplier]\label{lem:continuous-abs-kernel}
	With the Fourier convention used here,
	\[
		\frac1{2\pi}\fp\!\int_{\R}\frac{-2}{\theta^2}e^{ix\theta}\,d\theta=|x|
	\]
	in the sense of tempered distributions.
\end{lemma}

\begin{proof}
	Let $F$ denote the inverse Fourier transform on the left.  Twice
	differentiating in $x$ multiplies the Fourier-side distribution by
	$(i\theta)^2=-\theta^2$, hence
	\[
		F''=\mathcal F^{-1}(2)=2\delta_0.
	\]
	The even distributional solutions of $F''=2\delta_0$ are $|x|+c$.
	The Hadamard finite part $\fp(\theta^{-2})=-\frac{d}{d\theta}\pv(\theta^{-1})$
	is homogeneous of degree $-2$; its inverse transform is therefore homogeneous of
	degree $1$.  The additive constant is not homogeneous of degree $1$, so
	$c=0$ and $F=|x|$.
	This is the local real-line normalization of the same singular kernel used in
	Lemma~\ref{lem:periodic-abs-kernel}; that the periodic Abel finite part matches
	$-2\fp(\theta^{-2})$ near $\theta=0$ with no spurious $\delta_0$ is made precise in
	Lemma~\ref{lem:reg-match} below.
\end{proof}

\begin{lemma}[Matching of the periodic and local finite parts]\label{lem:reg-match}
	Near the origin the periodic Abel finite part $S$ of
	Lemma~\ref{lem:periodic-abs-kernel} splits as
	\[
		S(\theta)=-2\,\fp(\theta^{-2})+s(\theta),\qquad s(0)=-\tfrac16,
	\]
	with $s$ smooth; the principal part is exactly $-2\,\fp(\theta^{-2})$, no multiple of
	$\delta_0$ being created or lost in the split.
\end{lemma}

\begin{proof}
	We derive the local decomposition directly from the Abel-regularized kernel, by
	integrating twice; all identities below are exact distributional facts, so no
	$\delta_0$ can be created or lost unnoticed. Since
	$|m|e^{im\theta}=-\tfrac1{|m|}\,\partial_\theta^2\,e^{im\theta}$ for $m\ne0$,
	the Abel kernel has, for each $r<1$, the \emph{exact} primitive representation
	\[
		S_r(\theta)=-T_r''(\theta),\qquad
		T_r(\theta):=\sum_{m\ne0}\frac{r^{|m|}}{|m|}e^{im\theta}
		=-2\log\bigl|1-re^{i\theta}\bigr| ,
	\]
	the last equality by summing $\sum_{m\ge1}\tfrac{(re^{i\theta})^m}m
	=-\log(1-re^{i\theta})$ and adding its conjugate. As $r\uparrow1$,
	$T_r\to T(\theta):=-2\log\bigl|2\sin(\theta/2)\bigr|$ pointwise off $2\pi\Z$ and,
	being dominated by an integrable logarithmic singularity, in $L^1_{\rm loc}$,
	hence in the sense of distributions; differentiating twice,
	\[
		S=-T''\qquad\text{on }(-\pi,\pi),
	\]
	as distributions. Now split the primitive near the origin:
	\[
		T(\theta)=-2\log|\theta|+t(\theta),\qquad
		t(\theta):=-2\log\Bigl|\frac{\sin(\theta/2)}{\theta/2}\Bigr|,
	\]
	with $t$ smooth and even on $(-\pi,\pi)$ (the removed singularity), and
	$t(\theta)=\tfrac{\theta^2}{12}+O(\theta^4)$ from
	$\log\tfrac{\sin x}x=-\tfrac{x^2}6+O(x^4)$ at $x=\theta/2$. The distributional
	derivatives of the logarithmic part are the classical identities, in exactly the
	normalization fixed by Lemma~\ref{lem:continuous-abs-kernel}:
	\[
		\partial_\theta\log|\theta|=\pv(\theta^{-1}),\qquad
		\partial_\theta\,\pv(\theta^{-1})=-\fp(\theta^{-2}),
	\]
	so $\partial_\theta^2\bigl(-2\log|\theta|\bigr)=+2\,\fp(\theta^{-2})$, and
	therefore
	\[
		S=-T''=-2\,\fp(\theta^{-2})-t''(\theta).
	\]
	This is the asserted split with the smooth part $s:=-t''$, and
	$s(0)=-t''(0)=-\tfrac16$. No $\delta_0$ appears at any step: the two derivative
	identities are exact, and every other object is a locally integrable function.
	The values $-2$ and $s(0)=-\tfrac16$ are consistent with the pointwise Laurent
	expansion $S(\theta)=-\tfrac2{\theta^2}-\tfrac16+O(\theta^2)$ (equivalently
	$\csc^2(\theta/2)=4/\theta^2+\tfrac13+O(\theta^2)$); these are the constants that
	feed $\alatt$ in Lemma~\ref{lem:corner-rig}.
	The same argument in two variables gives
	$S_r(\theta_1)S_r(\theta_2)=\partial_1^2\partial_2^2[T_r(\theta_1)T_r(\theta_2)]$ with
	$T_r\otimes T_r\to T\otimes T$ in $L^1(\mathbb T^2)$; hence the tensor Abel pairing of
	\eqref{eq:rep} is the periodic distribution $\partial_1^2\partial_2^2(T\otimes T)=S\otimes S$,
	which may be paired with arbitrary smooth functions and localized by a smooth partition of
	unity.
\end{proof}

	\paragraph{Origin saddle}
	Choose a smooth partition of unity on the torus, with a reflection-symmetric cutoff
	equal to $1$ near $Q_0$ and supported in a sufficiently small square about $Q_0$,
	and analogous cutoffs around $Q_1,Q_2$.  No constant extracted below depends on
	this choice: the difference between two admissible partitions is supported in a
	compact set on which $|\Phi|\le\varrho'<1$, and pairing the kernel against such a
	difference is controlled by the $C^{2,2}$ seminorm of
	Lemma~\ref{lem:tensor-fp-continuity}, hence is $O(\ell^4\varrho'^{\,\ell})$ ---
	exponentially small, so all cutoff-boundary terms are absorbed into the stated
	remainders and no cutoff-dependent constant survives at any algebraic order.
	On the origin square split each kernel as
	$S(\theta)=-\tfrac2{\theta^2}+s(\theta)$, where
	$s(\theta)=-\tfrac16+O(\theta^2)$ is smooth.  This gives four pieces of
	$S(\theta_1)S(\theta_2)$: $\tfrac4{\theta_1^2\theta_2^2}$ (a),
	$-\tfrac2{\theta_1^2}s(\theta_2)$ (b), its mirror (c), and $s(\theta_1)s(\theta_2)$
	(d).  The cutoffs make every subsequent finite-part pairing a pairing of a
	compactly supported distribution with a smooth function.
	\begin{itemize}
		\item Piece (a) is the pure continuum/Edgeworth account and is evaluated in
		Lemma~\ref{lem:central-singular}, using the multiplier identity
		of Lemma~\ref{lem:continuous-abs-kernel},
		to give $c_L\ell+\asm+O(\ell^{-1})$.
		\item Pieces (b),(c) give $\alatt$ in closed form, below.
		\item Piece (d) has a bounded smooth kernel.  After
		$\theta=u/\sqrt\ell$, its Jacobian contributes $\ell^{-1}$ and the remaining
		Gaussian integral is bounded, so its origin contribution is $O(\ell^{-1})$.
	\end{itemize}

\begin{lemma}[Continuity estimate for the tensor finite part]
	\label{lem:tensor-fp-continuity}
	Let $\chi\in C_c^\infty(\R^2)$.  There is a constant $C_\chi$ such that every
	$F\in C^4$ satisfies
	\[
		\left|
		\left\langle
		\fp(u_1^{-2})\otimes\fp(u_2^{-2}),\,\chi F
		\right\rangle
		\right|
		\le
		C_\chi
		\max_{0\le a_1,a_2\le2}
		\sup_{u\in\operatorname{supp}\chi}
		\left|\partial_{u_1}^{a_1}\partial_{u_2}^{a_2}F(u)\right|.
	\]
	Thus the tensor finite part is continuous in a seminorm involving at most two
	derivatives in each variable. The constant $C_\chi$ depends only on
	$\max_{a_1,a_2\le2}\|\partial^a\chi\|_\infty$ and not on $\operatorname{supp}\chi$; in
	particular it is uniform for the dilated cutoffs $\chi_0(u/\sqrt\ell)$, $\ell\ge1$.
\end{lemma}

\begin{proof}
	In one dimension, for a compactly supported $C^2$ function $f$, one convenient
	definition of the Hadamard finite part (see, for example,
	Estrada--Kanwal~\cite{estrada_kanwal} or
	Gel'fand--Shilov~\cite{gelfand_shilov}) is
	\[
		\left\langle\fp(u^{-2}),f\right\rangle
		=\int_{|u|\le1}\frac{f(u)-f(0)-u f'(0)}{u^2}\,du
		+\int_{|u|>1}\frac{f(u)}{u^2}\,du-2f(0).
	\]
	Here the subtraction removes the non-integrable part of $u^{-2}$ at the origin,
	while the constant term $-2f(0)$ fixes the standard normalization
	$\fp(u^{-2})=-\frac{d}{du}\pv(u^{-1})$.
	Taylor's formula gives
	$|f(u)-f(0)-u f'(0)|\le\tfrac12u^2\|f''\|_\infty$ on $|u|\le1$;
	the second integral is bounded by $2\|f\|_\infty$, since $\int_{|u|>1}u^{-2}\,du=2$; in
	particular the constant does not depend on the support of $f$.
	Consequently
	\[
		\big|\langle\fp(u^{-2}),f\rangle\big|
		\le C\max_{0\le j\le2}\|f^{(j)}\|_\infty.
	\]
	Apply this estimate first in $u_1$, treating $u_2$ as a parameter.  The resulting
	function of $u_2$ and its first two derivatives are bounded by mixed derivatives
	$\partial_{u_1}^{a_1}\partial_{u_2}^{a_2}(\chi F)$ with
	$a_1,a_2\le2$.  Applying the same one-dimensional estimate in $u_2$, and absorbing
	the fixed derivatives of $\chi$ into $C_\chi$, proves the tensor estimate.
	This is the technical reason why all later remainder estimates are formulated in a
	$C^{2,2}$ norm rather than only in a supremum norm: the kernel has a second-order
	finite-part singularity in each variable.
\end{proof}

\begin{lemma}[Uniform $C^{2,2}$ control of the central remainder]\label{lem:c2-remainder}
	Fix $0<\epsilon<\tfrac12$. On the central region $|u|\le\ell^{1/2-\epsilon}$ write
	\[
		\Phi(u/\sqrt\ell)^\ell
		=e^{-\frac12u^\top Cu}\bigl[1+\ell^{-1/2}P_1(u)+\ell^{-1}P_2(u)+\ell^{-3/2}P_3(u)\bigr]
		+\widetilde{\mathcal R}_\ell(u),
	\]
	the $P_k$ being the rescaled-cumulant polynomials (degree $\le3k$). Then for every
	multi-index $a=(a_1,a_2)$ with $0\le a_1,a_2\le2$ there are constants $C',M$
	such that
	\[
		\sup_{|u|\le\ell^{1/2-\epsilon}}
		\bigl|\partial_u^a\widetilde{\mathcal R}_\ell(u)\bigr|
		\le C'\,\ell^{-2}\,(1+|u|^{M})\,e^{-\frac18u^\top Cu}.
	\]
	In particular $\widetilde{\mathcal R}_\ell\to0$ in the $C^{2,2}$ seminorm in
	which the tensor product
	$\fp(u_1^{-2})\otimes\fp(u_2^{-2})$ is continuous.
\end{lemma}

\begin{proof}
	This is the $s=3$, $\nu=2$ instance of Lemma~\ref{thm:edgeworth} for the present
	increment, used only as a \emph{local} expansion on the central window about
	$\theta=0$ (the two twisted saddles are treated separately, so the index-$3$
	periodicity of the adjacent increment is irrelevant here). The proof below is
	the self-contained argument promised there --- analyticity of $\log\Phi$ near
	the origin plus Cauchy estimates on a complex tube --- and it also fixes the
	retained polynomials $P_1,P_2,P_3$ used below; it transposes verbatim to any
	dimension and order, which is why Lemma~\ref{thm:edgeworth} needs no external
	citation.

	The joint per-trial characteristic function $\Phi$ is a trigonometric polynomial
	(the increments are bounded), hence entire.  Since $\Phi(0)=1$, there is a fixed
	polydisc $|\theta_j|<\delta$ on which the branch
	$K(\theta)=\log\Phi(\theta)$, $K(0)=0$, is analytic.  On
	$|u_j|\le\delta\sqrt\ell$ we therefore have
	\[
		\ell K(u/\sqrt\ell)
		=
		-\tfrac12u^\top Cu
		+\ell^{-1/2}B_3(u)+\ell^{-1}B_4(u)
		+\ell^{-3/2}B_5(u)+\ell^{-2}B_6(u)+\cdots ,
	\]
	where $B_r$ is homogeneous of degree $r$.  Exponentiating the non-quadratic part
	produces
	\[
		P_1=B_3,\qquad
		P_2=B_4+\tfrac12B_3^2,\qquad
		P_3=B_5+B_3B_4+\tfrac16B_3^3,
	\]
	which are the Edgeworth polynomials appearing in the statement.  
	Therefore the formal expansion is
	\[
		\Phi(u/\sqrt\ell)^\ell
		=
		e^{-\frac12u^\top Cu}
		\left[
			1+\ell^{-1/2}P_1(u)+\ell^{-1}P_2(u)+\ell^{-3/2}P_3(u)
		\right]
		+\widetilde{\mathcal R}_\ell(u).
	\]	
	It remains to show that
	the omitted terms are uniformly small on the whole central window, not only for
	fixed $u$.

	Write $q(u)=u^\top Cu$ and let $\lambda_0>0$ be the smallest eigenvalue of $C$, so $u^\top Cu\ge\lambda_0|u|^2$.
	Choose a large fixed constant $R$ and split the central window into
	$|u|\le R\sqrt{\log\ell}$ and
	$R\sqrt{\log\ell}\le |u|\le\ell^{1/2-\epsilon}$.
	The logarithmic boundary is chosen because
	$e^{-c q(u)}\le
	e^{-c\lambda_0R^2\log\ell}\le \ell^{-c\lambda_0R^2}$ there; by increasing the fixed number $R$
	we can make this any prescribed negative power of $\ell$.

	On the inner region $|u|\le R\sqrt{\log\ell}$ the cumulant correction
	\[
		Z_\ell(u):=\ell^{-1/2}B_3(u)+\ell^{-1}B_4(u)+\ell^{-3/2}B_5(u)+\cdots
	\]
	is uniformly small, since its leading size is
	$O(\ell^{-1/2}|u|^3)=O(\ell^{-1/2}(\log\ell)^{3/2})=o(1)$.  The analytic tail in
	$\ell K(u/\sqrt\ell)$ beginning with order $\ell^{-2}$ is therefore bounded by
	$C\ell^{-2}(1+|u|^{M_1})$.  Applying
	$|e^z-(1+z+z^2/2+z^3/6)|\le C|z|^4e^{|z|}$ to $Z_\ell$ and collecting the powers
	$\ell^{-1/2},\ell^{-1},\ell^{-3/2}$ gives
	\[
		|\widetilde{\mathcal R}_\ell(u)|
		\le C\ell^{-2}(1+|u|^{M_2})e^{-\frac13q(u)}
		\qquad (|u|\le R\sqrt{\log\ell}).
	\]
	The exponent $1/3$ is only a convenient weakening of the true Gaussian exponent
	$1/2$.

	On the outer central region, Taylor's formula for the real part of $K$ gives
	$\operatorname{Re}K(\theta)=-\frac12\theta^\top C\theta+O(|\theta|^3)$ near
	zero.  For $\theta=u/\sqrt\ell$, we get
	\[
		\ell\operatorname{Re}K(u/\sqrt\ell)
		\le
		-\frac12q(u)+C'\frac{|u|^3}{\sqrt\ell}.
	\]
	On the full central window $|u|\le\ell^{1/2-\epsilon}$,
	\[
		\frac{|u|^3}{\sqrt\ell}
		\le
		\frac1{\lambda_0}\frac{|u|}{\sqrt\ell}q(u)
		\le C_0\ell^{-\epsilon}q(u).
	\]
	Hence, for all sufficiently large $\ell$,
	\[
		\ell\operatorname{Re}K(u/\sqrt\ell)
		\le
		-\frac25q(u),\quad\text{i.e.,}\quad
		|\Phi(u/\sqrt\ell)|^\ell
		\le
		e^{-\frac25q(u)}.
	\]
	The constant $2/5$ is not special; it can be any fixed number strictly
	between $1/4$ and $1/2$.
	
	The retained
	Gaussian-polynomial terms are bounded by the same expression times a fixed
	polynomial factor.  Now
	\[
		e^{-\frac25q(u)}
		=
		e^{-\frac14q(u)}e^{-\frac{3}{20}q(u)}
		\le
		\ell^{-(3/20)\lambda_0R^2}e^{-\frac14q(u)}
	\]
	on $|u|\ge R\sqrt{\log\ell}$.  Choosing $R$ so that
	$(3/20)\lambda_0R^2\ge2$ absorbs the desired $\ell^{-2}$ factor, and the fixed
	polynomial losses are absorbed by increasing $M$.  Thus the exact term and all
	retained terms are $O(\ell^{-2}(1+|u|^M)e^{-q(u)/4})$ on the outer region, and
	so is their difference.

	Combining the inner and outer estimates gives the stated bound for
	$\widetilde{\mathcal R}_\ell$ itself.

	The derivative bounds follow from analyticity by a Cauchy estimate, but the
	estimate must be fed a bound on a \emph{complex} neighborhood, so we first
	extend the sup bound to the tube
	\[
		\mathcal T_\ell:=\bigl\{u\in\C^2:\ |{\rm Re}\,u|\le2\ell^{1/2-\epsilon},\
		|{\rm Im}\,u_j|\le1\ (j=1,2)\bigr\},
	\]
	which is contained in the analyticity polydisc $|u_j|<\delta\sqrt\ell$ for large
	$\ell$. Write $u=x+iy$, $x={\rm Re}\,u$. Every ingredient of the inner and outer
	estimates extends to $\mathcal T_\ell$ with $x$ in place of $u$, at the cost of
	bounded factors coming from the bounded imaginary part: for the Gaussian,
	\[
		\bigl|e^{-\frac12u^\top Cu}\bigr|
		=e^{-\frac12x^\top Cx+\frac12y^\top Cy}\le C\,e^{-\frac12q(x)}
		\qquad(|y|\le\sqrt2);
	\]
	the polynomials satisfy $|P_k(u)|\le C(1+|x|)^{3k}$; the cumulant correction
	obeys $|Z_\ell(u)|\le C\ell^{-1/2}(1+|x|)^3$ on $\mathcal T_\ell$, small on the
	inner region; and for the exact term, Taylor's formula for the analytic $K$ on
	the polydisc gives
	\[
		\ell\,{\rm Re}\,K(u/\sqrt\ell)
		\le-\tfrac12q(x)+\tfrac12q(y)+C'\ell^{-1/2}|u|^3
		\le-\tfrac25q(x)+C''
		\qquad(u\in\mathcal T_\ell),
	\]
	by the same absorption $\ell^{-1/2}|u|^3\le C_0\ell^{-\epsilon}q(x)+C$ as before.
	Repeating the inner/outer split verbatim with these complexified bounds yields
	\[
		\bigl|\widetilde{\mathcal R}_\ell(u)\bigr|
		\le C\,\ell^{-2}\,(1+|x|^{M})\,e^{-\frac14q(x)}
		\qquad(u\in\mathcal T_\ell).
	\]
	Now fix a real $u_0$ with $|u_0|\le\ell^{1/2-\epsilon}$ and apply Cauchy's
	inequality on the unit polydisc $\{\,|u_j-u_{0,j}|=1\,\}\subset\mathcal T_\ell$:
	\[
		\bigl|\partial_u^{a}\widetilde{\mathcal R}_\ell(u_0)\bigr|
		\le a_1!\,a_2!\sup_{|u_j-u_{0,j}|=1}
		\bigl|\widetilde{\mathcal R}_\ell(u)\bigr|
		\le C''\,\ell^{-2}\,(1+|u_0|^{M'})\,e^{-\frac18u_0^\top Cu_0},
	\]
	where the last step uses $(1+|x|^M)\le C(1+|u_0|^M)$ and
	$e^{-\frac14q(x)}\le Ce^{-\frac18q(u_0)}$ for $|x-u_0|\le\sqrt2$ (moving the real
	part by a bounded amount changes the Gaussian weight by a bounded factor). This
	is the stated bound for all $0\le a_1,a_2\le2$ simultaneously.
	Lemma~\ref{lem:tensor-fp-continuity} then gives the
	final continuity assertion for the tensor finite-part pairing.
\end{proof}

\begin{lemma}[Dual/real duality of the finite-part pairings]\label{lem:dual-real}
	Let $g_k\phi_C$ be the inverse Fourier transform of $e^{-\frac12u^\top Cu}P_k(u)$, so
	that $g_1=q_1$ and $g_2=q_2$ in the notation of \eqref{eq:q1}. Then, for the tensor
	Hadamard finite part,
	\[
		\frac4{(2\pi)^2}\Bigl\langle\fp(u_1^{-2})\otimes\fp(u_2^{-2}),\,
		e^{-\frac12u^\top Cu}P_k(u)\Bigr\rangle
		=\int_{\R^2}|w_1|\,|w_2|\,g_k(w)\,\phi_C(w)\,dw .
	\]
	In particular $k=0$ gives $\int|w_1||w_2|\phi_C=c_L$, and $k=2$ gives
	$\int|w_1||w_2|q_2\phi_C=\asm$ (Lemma~\ref{lem:price}).
\end{lemma}

\begin{proof}
	No limit interchange occurs here: the identity is an instance of the
	\emph{definition} of the Fourier transform of a tempered distribution, applied
	to a Schwartz test function. Indeed, $g_k\phi_C$ is a Schwartz function (a
	polynomial times a nondegenerate Gaussian), and, by the definition of $g_k$,
	its Fourier transform is exactly $e^{-\frac12u^\top Cu}P_k(u)$ (also Schwartz).
	The tensor kernel
	$W:=\tfrac4{(2\pi)^2}\fp(u_1^{-2})\otimes\fp(u_2^{-2})$
	is a tempered distribution, and by Lemma~\ref{lem:continuous-abs-kernel}
	(tensored in the two variables, in the Fourier convention fixed there) its
	inverse Fourier transform is the function $|w_1|\,|w_2|$. For any tempered
	distribution and any Schwartz function $\varphi$ one has
	$\langle\widehat W,\varphi\rangle=\langle W,\widehat\varphi\rangle$ \emph{by
	definition}; applying this with $\varphi=g_k\phi_C$ gives
	\[
		\Bigl\langle W,\,e^{-\frac12u^\top Cu}P_k\Bigr\rangle
		=\Bigl\langle W,\widehat{g_k\phi_C}\Bigr\rangle
		=\Bigl\langle\widehat W,\,g_k\phi_C\Bigr\rangle
		=\int_{\R^2}|w_1|\,|w_2|\,g_k(w)\,\phi_C(w)\,dw,
	\]
	the last equality because $\widehat W=|w_1||w_2|$ is a locally integrable
	function of polynomial growth paired with a Schwartz function. (If one prefers
	to see the finite parts regularized explicitly: replace $\fp(u_j^{-2})$ by its
	Abel regularization, for which the identity is a classical absolutely
	convergent computation, and pass to the limit on both sides --- on the left by
	the $C^{2,2}$ continuity of Lemma~\ref{lem:tensor-fp-continuity} applied to the
	fixed Schwartz function, on the right by dominated convergence. Both routes
	give the same normalization.) The $C^{2,2}$ continuity of
	Lemma~\ref{lem:tensor-fp-continuity} is \emph{not} needed for this identity; it
	is used only later, when the exact $\Phi^\ell$ is replaced by its Edgeworth
	expansion inside the pairing (Lemma~\ref{lem:c2-remainder}).
\end{proof}

\begin{lemma}[Central singular piece]\label{lem:central-singular}
	Piece {\rm(a)} of the origin saddle equals
	\[
		c_L\ell+\asm+O(\ell^{-1}).
	\]
	In particular, neither an $O(\sqrt\ell)$ nor an $O(\ell^{-1/2})$ term occurs.
\end{lemma}

\begin{proof}
	Let $\chi_0$ be the even cutoff used for the origin saddle, equal to $1$ in a
	neighborhood of the origin and supported in a sufficiently small square where
	the origin is the only maximum of $|\Phi|$.  Piece (a) is
	\[
		I_a(\ell)=
		\frac1{(2\pi)^2}
		\left\langle
		\frac4{\theta_1^2\theta_2^2},
		\chi_0(\theta)\Phi(\theta)^\ell
		\right\rangle_{\!\mathrm{fp}}.
	\]
	Here and below the brackets denote the tensor Hadamard finite part.  With
	$\theta=u/\sqrt\ell$ and $\chi_\ell(u)=\chi_0(u/\sqrt\ell)$, the homogeneity of
	$\fp(\theta^{-2})$ gives
	\[
		I_a(\ell)=
		\frac{4\ell}{(2\pi)^2}
		\left\langle
		\fp(u_1^{-2})\otimes\fp(u_2^{-2}),
		\chi_\ell(u)\Phi(u/\sqrt\ell)^\ell
		\right\rangle .
	\]
	The factor is $\ell$: the two singular kernels contribute $\ell^2$, and the
	two-dimensional Jacobian contributes $\ell^{-1}$.

	We first localize to the growing central window $|u|\le\ell^{1/2-\epsilon}$ by a smooth
	window $\psi_\ell(u):=\psi(u/\ell^{1/2-\epsilon})$, where $\psi\in C_c^\infty(\R^2)$ is even,
	$\psi=1$ on $|u|\le\tfrac12$ and $\psi=0$ on $|u|\ge1$; each derivative of $\psi_\ell$ is
	$O(\ell^{-(1/2-\epsilon)})$ per order, so the factor $\psi_\ell$ is absorbed in the constant of
	Lemma~\ref{lem:tensor-fp-continuity} (a sharp indicator would not do, since the kernel is
	singular on the axes).  On the support of $1-\psi_\ell$ inside the origin square,
	$|\Phi(u/\sqrt\ell)|^\ell\le e^{-c|u|^2}\le e^{-c\ell^{1-2\epsilon}/4}$; applying
	Lemma~\ref{lem:tensor-fp-continuity} to that piece gives at most a fixed power of $\ell$
	times $e^{-c\ell^{1-2\epsilon}/4}$.  This is smaller than any inverse power of $\ell$, even
	after the outside factor $\ell$.  Thus all constant and $O(\ell^{-1})$ information comes
	from the central window.

	On that window Lemma~\ref{lem:c2-remainder} gives, in the $C^{2,2}$ seminorm
	needed for finite-part pairing,
	\[
		\Phi(u/\sqrt\ell)^\ell=e^{-\frac12u^\top Cu}
		\big[1+\ell^{-1/2}P_1(u)+\ell^{-1}P_2(u)
		+\ell^{-3/2}P_3(u)\big]
		+\widetilde{\mathcal R}_\ell(u),
	\]
	with
	\[
		\partial^a\widetilde{\mathcal R}_\ell(u)
		=O\!\left(\ell^{-2}(1+|u|^M)e^{-\frac18u^\top Cu}\right),
		\qquad 0\le a_1,a_2\le2 .
	\]
	Since the tensor finite part is continuous in this seminorm
	(Lemma~\ref{lem:tensor-fp-continuity}), the remainder contributes
	$O(\ell^{-2})$ inside the brackets, and therefore $O(\ell^{-1})$ to
	$I_a(\ell)$ after multiplication by the outside factor $\ell$.

	The cutoffs $\chi_\ell\psi_\ell$ can be suppressed in the retained Gaussian--polynomial
	terms at the cost of $O(\ell^{-1})$.  Indeed, $1-\chi_\ell\psi_\ell$ is supported where
	$|u|\ge\tfrac12\ell^{1/2-\epsilon}$ inside the local Gaussian model, and
	Lemma~\ref{lem:tensor-fp-continuity} again converts the Gaussian tail into an
	exponentially small error.  Thus the retained pairings are the whole-space
	distributional Gaussian pairings, which by Lemma~\ref{lem:dual-real} equal the
	real-space kink integrals of the smooth Edgeworth section: the leading term gives
	$c_L\ell$ (Lemma~\ref{lem:lead}), and the $\ell^{-1}P_2$ term gives $\asm$
	(Lemma~\ref{lem:price}).

	The polynomials $P_1$ and $P_3$ are odd under the central reflection
	$u\mapsto-u$: $P_1$ has total degree $3$, while $P_3$ is a sum of the
	total-odd-degree structures $A_5$, $A_3A_4$, and $A_3^3$.  The Gaussian, the
	cutoff $\chi_\ell$, and the tensor finite-part distribution
	$\fp(u_1^{-2})\otimes\fp(u_2^{-2})$ are even under this reflection.  Therefore the
	$\ell^{-1/2}P_1$ pairing, which would have produced an $O(\sqrt\ell)$ term, and
	the $\ell^{-3/2}P_3$ pairing, which would have produced an $O(\ell^{-1/2})$ term,
	vanish exactly.

	Combining the retained even terms with the remainder estimate gives
	\[
		I_a(\ell)=c_L\ell+\asm+O(\ell^{-1}),
	\]
	and the preceding parity argument proves the asserted absence of the
	$O(\sqrt\ell)$ and $O(\ell^{-1/2})$ terms.
\end{proof}

\begin{lemma}[Finite-part Gaussian]\label{lem:fpgauss}
	For $a>0$, $\ \fp\!\int_\R\dfrac{e^{-au^2}}{u^2}\,du=-2\sqrt{\pi a}.$
\end{lemma}

\begin{proof}
	For even $f$, $\fp\!\int_\R\tfrac{f}{u^2}\,du=\int_\R\tfrac{f(u)-f(0)}{u^2}\,du$. With
	$f=e^{-au^2}$,
	\[
		\frac{d}{da}\int_0^\infty\frac{e^{-au^2}-1}{u^2}\,du=-\int_0^\infty e^{-au^2}\,du=-\tfrac12\sqrt{\pi/a},
	\]
	and integrating from $a=0$ gives $\int_0^\infty\tfrac{e^{-au^2}-1}{u^2}\,du=-\sqrt{\pi a}$; double it.
\end{proof}

\begin{lemma}[Corner constant, rigorous]\label{lem:corner-rig}
	Pieces (b),(c) of the origin saddle each equal
	$-\tfrac{\sqrt3}{12\pi}+O(\ell^{-1})$.  Hence their combined contribution is
	\[
		\alatt+O(\ell^{-1}),\qquad \alatt=-\frac{\sqrt3}{6\pi},
	\]
	confirming the heuristic real-space value obtained above.
\end{lemma}

\begin{proof}
	In (b), rescale $\theta=u/\sqrt\ell$: $-\tfrac2{\theta_1^2}=-\tfrac{2\ell}{u_1^2}$,
	$d\theta=du/\ell$ (the $\ell$'s cancel), $s(\theta_2)\to s(0)=-\tfrac16$, and
	$\Phi^\ell\to e^{-\frac12u^\top Cu}$. Thus
	\[
		\text{(b)}_{\mathrm{lead}}=\frac1{4\pi^2}\cdot\frac13\,
		\fp\!\int_{\R^2}\frac{e^{-\frac12u^\top Cu}}{u_1^2}\,du .
	\]
	Completing the square, $\tfrac12u^\top Cu=\tfrac13(u_2-\tfrac{u_1}2)^2+\tfrac14u_1^2$, so
	integrating out $u_2$ gives $\int e^{-\frac12u^\top Cu}\,du_2=\sqrt{3\pi}\,e^{-u_1^2/4}$. By
	Lemma~\ref{lem:fpgauss} ($a=\tfrac14$), $\fp\int_\R e^{-u_1^2/4}u_1^{-2}\,du_1=-\sqrt\pi$,
	so
	\[
		\fp\!\int_{\R^2}\frac{e^{-\frac12u^\top Cu}}{u_1^2}\,du=-\pi\sqrt3,\qquad
		\text{(b)}_{\mathrm{lead}}=\tfrac1{4\pi^2}\cdot\tfrac13\cdot(-\pi\sqrt3)=-\tfrac{\sqrt3}{12\pi}.
	\] The $\ell^{-1/2}$ term couples $u_1^{-2}$ to $A_3(u)$;
	since $A_3(-u)=-A_3(u)$ and the rest is even, it integrates to $0$. The remaining
	corrections are each $O(\ell^{-1})$, as follows.

	The seminorm used for all the one-singular-variable pairings below is the
	one-dimensional estimate inside the proof of
	Lemma~\ref{lem:tensor-fp-continuity}: for $f\in C^2_c(\R)$,
	\begin{equation}\label{eq:fp-1d-seminorm}
		\bigl|\langle\fp(u_1^{-2}),f\rangle\bigr|
		\le C\max_{0\le j\le2}\|f^{(j)}\|_\infty,
	\end{equation}
	applied at each fixed $u_2$ and then integrated in $u_2$; a pairing
	$\langle\fp(u_1^{-2}),F(u_1,u_2)\rangle$ integrated in $u_2$ is thus bounded by
	$C\int\max_{j\le2}\sup_{u_1}|\partial_{u_1}^jF(u_1,u_2)|\,du_2$.

	First write
	\[
		s(u_2/\sqrt\ell)=s(0)+\ell^{-1}u_2^2r_\ell(u_2),
		\qquad
		r_\ell(u_2)=\int_0^1(1-t)\,s''(tu_2/\sqrt\ell)\,dt,
	\]
	so that $r_\ell$ and its first two derivatives are bounded uniformly in $\ell$
	on the rescaled origin support ($s$ is smooth on the fixed origin square and
	$|u_2/\sqrt\ell|\le\delta$ there; each $u_2$-derivative of $r_\ell$ brings a
	harmless factor $\ell^{-1/2}$).  For the curvature term,
	$F=e^{-\frac12u^\top Cu}u_2^2r_\ell(u_2)$ satisfies
	$\max_{j\le2}\sup_{u_1}|\partial_{u_1}^jF|\le C(1+u_2^2)e^{-cu_2^2}$, which is
	integrable in $u_2$; by \eqref{eq:fp-1d-seminorm} the pairing is $O(1)$,
	and the prefactor $\ell^{-1}$ makes the curvature
	contribution $O(\ell^{-1})$.

	Second, after the vanishing odd term, the next central-saddle term has the form
	\[
		\ell^{-1}e^{-\frac12u^\top Cu}P_2(u)
	\]
	plus the remainder $\widetilde{\mathcal R}_\ell$, whose $u_1$-derivatives
	through order two are bounded by
	$C\ell^{-2}(1+|u|^M)e^{-\frac18u^\top Cu}$ (Lemma~\ref{lem:c2-remainder});
	both satisfy the hypotheses of \eqref{eq:fp-1d-seminorm} with $u_2$-integrable
	majorants, so their pairings are $O(\ell^{-1})$ and $O(\ell^{-2})$
	respectively.  Products of the curvature term with the
	$\ell^{-1/2}P_1$ term are smaller still.

	Finally, replacing the rescaled cutoff region by $\R^2$ leaves
	$|u|\ge c\sqrt\ell$ in the local Gaussian model.  The finite-part estimate
	introduces at most a polynomial factor, so this tail is
	$O(\ell^M e^{-c'\ell})=O(\ell^{-1})$.  Thus
	\[
		\text{(b)}
		=-\frac{\sqrt3}{12\pi}+O(\ell^{-1}).
	\]
	Piece (c) is identical by symmetry.
\end{proof}

	\paragraph{Twisted saddles}
	At $Q_1,Q_2$ the arguments $\theta_1,\theta_2$ are bounded away from $0$, so the
	\emph{whole} kernel $S(\theta_1)S(\theta_2)$ is smooth and bounded there.  We spell
	out the local computation at $Q_1$; the saddle $Q_2$ is its complex conjugate.  In
	the local coordinate $\theta=Q_1+\eta$ the phase is \emph{exact},
	$\Phi(Q_1+\eta)^\ell=\bar\omega^\ell\,\Phi(\eta)^\ell$, so no contour deformation or
	complex saddle normalization is needed: the constant phase $\bar\omega^\ell$ factors
	out, and what remains is the same central-window analysis as at the origin, now with
	a smooth amplitude.  Write
	$A(\theta):=S(\theta_1)S(\theta_2)$; since $S(\pm\tfrac{2\pi}3)
	=-\tfrac12\csc^2(\tfrac\pi3)=-\tfrac23$, the amplitude value is
	$A(Q_1)=\tfrac49$.  After the scaling $\eta=u/\sqrt\ell$ (Jacobian $\ell^{-1}$),
	the central expansion of Lemma~\ref{lem:c2-remainder} (sup-norm version; no finite
	part is involved here) and the Taylor expansion of the amplitude give
	\[
		\frac1{(2\pi)^2}\,\frac{\bar\omega^\ell}{\ell}\!\int
		e^{-\frac12u^\top Cu}
		\Bigl[1+\tfrac{P_1(u)}{\sqrt\ell}+\dots\Bigr]
		\Bigl[A(Q_1)+\tfrac{\nabla A(Q_1)\cdot u}{\sqrt\ell}+O\bigl(\tfrac{|u|^2}\ell\bigr)\Bigr]
		du .
	\]
	The two $\ell^{-1/2}$ terms are odd and integrate to zero against the even
	Gaussian, so the relative correction is $O(\ell^{-1})$; with
	$\int_{\R^2}e^{-\frac12u^\top Cu}\,du=2\pi/\sqrt{\det C}=2\pi\sqrt3$, the two
	saddles contribute, respectively,
	\[
		\bar\omega^\ell\left[\frac{2\sqrt3}{9\pi\ell}+O(\ell^{-2})\right],
		\qquad
		\omega^\ell\left[\frac{2\sqrt3}{9\pi\ell}+O(\ell^{-2})\right],
	\]
	the leading coefficient being
	\[
		\frac1{(2\pi)^2}\left(\frac49\right)
		\frac{2\pi}{\ell\sqrt{\det C}}=\frac{2\sqrt3}{9\pi\ell}
		\qquad(\text{incidentally }=-\,\frac{a_*}{\ell}).
	\]
	Their sum is
	$\frac{4\sqrt3}{9\pi\ell}\cos(2\pi\ell/3)+O(\ell^{-2})$.  Thus the
	residue-carrying saddles never reach constant order; they first appear in the
	$O(\ell^{-1})$ term.  This explicitly replaces the heuristic residue-averaging of
	\eqref{eq:corner-id}.
	Consequently the constant term is independent of $\ell\bmod3$.  The residue class
	still exists, but it is pushed one order lower; this is exactly what the numerical
	Richardson extrapolation was exploiting when it worked inside a fixed residue
	class.

	\paragraph{Off-saddle and axis estimate}
	Fix $\delta\in(0,\pi/3)$ and use the smooth partition of unity introduced above,
	subordinate to the origin square
	$U_0=\{|\theta_1|,|\theta_2|<\delta\}$, neighborhoods $U_1,U_2$ about the twisted saddles,
	the axis tubes
	$T_1=\{|\theta_1|<\delta,\ |\theta_2|\ge\delta\}$ and
	$T_2=\{|\theta_2|<\delta,\ |\theta_1|\ge\delta\}$, and the bulk
	$Q=\{|\theta_1|,|\theta_2|\ge\delta\}\setminus(U_1\cup U_2)$. Since $|\Phi|=1$ only at
	the three saddles (all contained in $U_0\cup U_1\cup U_2$),
	$\varrho:=\max_{T_1\cup T_2\cup Q}|\Phi|<1$.
	\begin{itemize}
		\item On $Q$ the kernel $S(\theta_1)S(\theta_2)$ is smooth and bounded, so the
		contribution is $O(\varrho^\ell)$.
		\item On $T_1$ the variable $\theta_2$ is bounded away from $0$, so $S(\theta_2)$ is
		smooth and bounded; the inner $\theta_1$-integral is a one-dimensional Hadamard
		finite part of the order-two kernel $S(\theta_1)=-2/\theta_1^2+O(1)$ against
		$\theta_1\mapsto\chi(\theta)\Phi(\theta_1,\theta_2)^\ell$, where $\chi$ is the
		corresponding smooth cutoff.  The order-two distribution estimate gives
	\[
		\Bigl|\,\fp\!\int S(\theta_1)\chi(\theta)\Phi(\theta_1,\theta_2)^\ell\,d\theta_1\Bigr|
		\le C\max_{j\le2}\sup_{\theta_1}
		\bigl|\partial_{\theta_1}^j(\chi\Phi^\ell)\bigr|
		\le C\,\ell^2\varrho^{\ell-2},
	\]
		since derivatives of $\Phi^\ell$ of order at most two are
		$O(\ell^2|\Phi|^{\ell-2})$ and $|\Phi|\le\varrho$ on the cutoff support.
		Integrating the bounded factor $S(\theta_2)$ over the finite
		$\theta_2$-range gives $O(\ell^2\varrho^\ell)$; the tube $T_2$ is identical by symmetry.
	\end{itemize}
	With $c:=-\log\varrho>0$ both bounds are $O(\ell^2 e^{-c\ell})$ and are absorbed into the
	$O(\ell^{-1})$ remainder. Hence only the three saddles contribute at constant order,
	which is the localization used in the assembly below.

\subsection{Assembly and verification}
	\label{ssec:adj-assembly}

\begin{proof}[Proof of Theorem~\ref{thm:astar}]
	By \eqref{eq:rep} and the localization, $\hadj(\ell)$ is the sum of the three
	saddle contributions up to $O(\ell^{-1})$. The origin saddle gives
	$c_L\ell+\asm$ (piece (a)) $+\,\alatt$ (pieces (b)+(c)) $+\,O(\ell^{-1})$ (piece
	(d)); the twisted saddles give $O(\ell^{-1})$. With
	$\asm=-\tfrac{\sqrt3}{18\pi}$ (Lemma~\ref{lem:price}) and
	$\alatt=-\tfrac{\sqrt3}{6\pi}$ (Lemma~\ref{lem:corner-rig}),
	\[
		a_*=\asm+\alatt=-\frac{\sqrt3}{18\pi}-\frac{\sqrt3}{6\pi}
		=-\frac{\sqrt3}{18\pi}-\frac{3\sqrt3}{18\pi}=-\frac{4\sqrt3}{18\pi}
		=-\frac{2\sqrt3}{9\pi}. \qedhere
	\]
\end{proof}

	The decomposition is remarkably clean. The smooth part
	$\asm=-\tfrac{\sqrt3}{18\pi}$ comes entirely from the squared-third-cumulant
	Edgeworth term; the lattice part $\alatt=-\tfrac{\sqrt3}{6\pi}$ from the index-$3$
	sublattice and the two kink lines. Numerically the lattice part is exactly
	three quarters of $a_*$ and the smooth Edgeworth correction the remaining quarter:
	$\alatt/a_*=\tfrac{1/6}{2/9}=\tfrac34$, $\asm/a_*=\tfrac14$. This is the
	$\rho=-\tfrac12$ counterpart of the disjoint theorem ($\rho=0$): there the third
	cumulants vanished and the constant was $\kappa_{2,2}$ times an origin probability; here
	the fourth-cumulant integral is $0$, the smooth part is carried by the new
	$(\text{third cumulant})^2$ term, and the dominant three quarters is the lattice
	corner.

\section{Synthesis of the variance expansion}
	\label{sec:synthesis}

	With all three pieces in closed form, the correction $\beta(n)$ assembles. Recall
	$\Var(d)=N\,\alpha(n)-\beta(n)+O(1/N)$, and the three contributions: the per-edge
	term \eqref{eq:beta-per-edge}, the disjoint term $+\tfrac{(n-2)(n-3)}{\pi n}$ (from
	the covariance $-1/(\pi n)$ of Theorem~\ref{thm:Cov-disj}), and the adjacent term
	$-2(n-2)a_{\mathrm{adj}}(n)$ with $a_{\mathrm{adj}}(n)=a_*+\tfrac{n-2}{2\pi n}$ from \eqref{eq:cov-adj-struct}.
	The only bookkeeping issue is the sign.  The per-edge, adjacent, and disjoint
	calculations naturally give constant-order contributions to
	$\Var(d)-N\alpha(n)$; the coefficient $\beta(n)$ is the negative of their sum.  This
	is why negative covariance corrections enter $\beta(n)$ with a positive sign, and
	why the adjacent constant appears below through $-2(n-2)a_{\mathrm{adj}}(n)$.

	The synthesis has two different logical levels.  The coefficient $\beta(n)$ is fully
	closed: every constant entering it has been evaluated in the preceding sections.
	The next coefficient $\gamma(n)$ is also obtained by the same exact recombination
	of blocks, but one block contains a second-order triangular-lattice constant.  We
	therefore state the formula for $\gamma(n)$ as a theorem with that constant defined
	intrinsically, and separate the experimentally identified algebraic value of the
	constant into a conjecture.  This separation is important: the expansion and the
	linear dependence of $\gamma(n)$ on the three block coefficients do not rely on the
	unproved closed form of $a_\triangle$.

\subsection{The correction $\beta(n)$}
	\label{ssec:beta-correction}

\begin{theorem}[Closed form for $\beta(n)$]\label{thm:beta-closed}
	For fixed $n\ge3$, with $a_*=-2\sqrt3/(9\pi)$,
	\begin{equation}\label{eq:beta-closed}
		\beta(n)=-\frac{(n-2)(n+1)}{2\pi n}+\frac{4\sqrt3\,(n-2)}{9\pi},
	\end{equation}
	and, expanding $\beta(n)=\beta_1n+\beta_0+\beta_{-1}/n+\cdots$,
	\begin{equation}\label{eq:beta-coeffs}
		\beta_1=\frac{8\sqrt3-9}{18\pi}\approx0.08588,\quad
		\beta_0=\frac{9-16\sqrt3}{18\pi}\approx-0.33092,\quad
		\beta_{-1}=\frac1\pi\approx0.31831,
	\end{equation}
	with $\beta_{-k}=0$ for $k\ge2$ exactly: by \eqref{eq:beta-closed} $\beta(n)$ is a
	linear function of $n$ plus a constant multiple of $(n-2)(n+1)/n=n-1-2/n$, whose
	Laurent expansion in $1/n$ terminates at the $1/n$ term.
	(Only $n=2$ is genuinely exceptional, Theorem~\ref{thm:Var-n2}; the degenerate cases $n=3,4$
	are made explicit in \S\ref{ssec:beta-small-n}.)
\end{theorem}

\begin{proof}
	Summing the three contributions (Theorem~\ref{thm:Var-D1} and
	Proposition~\ref{prop:adj-transfer} for $n\ge3$; the disjoint term has multiplicity $0$ at
	$n=3$ and is Theorem~\ref{thm:Cov-disj} for $n\ge4$),
	\[
		\beta(n)=-\frac{(n-1)(n-2)}{2\pi n}+\frac{(n-2)(n-3)}{\pi n}-2(n-2)\Big(a_*+\frac{n-2}{2\pi n}\Big),
	\]
	the disjoint sign following from the convention $\beta=-(\Var(d)-N\alpha)$. Factor
	$n-2$. The three rational terms inside the remaining bracket combine as
	\[
		-\frac{n-1}{2\pi n}+\frac{n-3}{\pi n}-\frac{n-2}{\pi n}
		=-\frac{n+1}{2\pi n}.
	\]
	Since $-2a_*=4\sqrt3/(9\pi)$, this proves \eqref{eq:beta-closed}. Finally,
	\[
		\frac{(n-2)(n+1)}n=n-1-\frac2n,
	\]
	so collecting the coefficients of $n,1,n^{-1}$ gives
	\eqref{eq:beta-coeffs}; no lower powers can occur because the displayed identity
	is exact.
\end{proof}

	\medskip\noindent\emph{Cyclic consequence.} Lemma~\ref{lem:linear-cyclic} with
	Theorems~\ref{thm:Var-D1} and~\ref{thm:beta-closed} gives, for $n\ge3$,
	$\Var(d^{c})=\tfrac{n}{n-2}\bigl[\Var(d)-\Var(|D_1|)\bigr]=N\alpha_c-\beta_c(n)+O(1/N)$ with
	\[
		\beta_c(n)=\frac{n}{n-2}\Bigl[\beta(n)+\frac{n-2}{2\pi n}\Bigr]=\frac{8\sqrt3-9}{18\pi}\,n=\beta_1 n
	\]
	exactly: the cyclic correction is linear in $n$, with no constant or $1/n$ term. The cyclic
	statistic is treated in full in Part~5~\cite{banic_elezovic_p5}.

	\medskip
	The closed form \eqref{eq:beta-closed} matches the $\beta(n)$ extracted from the
	exact $\sigma^2(n,N)$ at $N=2000$ (via $\beta(n)\approx N[\alpha(n)-\sigma^2(n,N)]$);
	here and in the tables below the exact $\sigma^2(n,N)$ is evaluated from the exact
	decomposition \eqref{eq:Var-d-3comp}, the disjoint block being the single $O(N)$
	binomial sum of Corollary~\ref{cor:exact-disjoint-moment}:
	\begin{center}
	\begin{tabular}{r|r|r}
		$n$ & $\beta(n)$ (Theorem~\ref{thm:beta-closed}) & $\beta(n)$ empirical ($N{=}2000$)\\\hline
		$10$  & $0.55972$  & $0.56097$\\
		$50$  & $3.96946$  & $4.00424$\\
		$100$ & $8.26028$  & $8.40226$\\
		$200$ & $16.84670$ & $17.42987$\\
	\end{tabular}
	\end{center}
	The empirical values exceed the closed $\beta(n)$ by approximately $-\gamma(n)/N>0$ --- the
	second-order term of \eqref{eq:intro-expansion}, with $\gamma(n)<0$
	(Theorem~\ref{thm:gamma-linear}); the gap grows with $n$ (from $\approx0.001$ at $n=10$ to
	$\approx0.58$ at $n=200$, where higher-order terms add a further $4\%$) and vanishes as
	$1/N$. Richardson extrapolation of the exact $\sigma^2(n,N)$ over $400\le N\le1500$
	reproduces $\beta(n)$ to at least $18$ digits and $\gamma(n)$ (with the conjectural value
	of $a_\triangle$) to at least $16$ digits for $5\le n\le20$.

	Because $\beta(n)$ and $\gamma(n)$ grow like $\beta_1n$ and $\gamma_2n^2$, the natural small
	parameter is $n/N$, the reciprocal of the mean count per bin. The relative error
	of the truncations of \eqref{eq:intro-expansion} against the exact $\sigma^2(n,N)$, for
	$20\le n\le200$ (the dependence on $n$ at fixed $N/n$ is weak, and the ranges below cover it), is
	\begin{center}
	\begin{tabular}{c|c|c|c}
		$N/n$ & $\alpha(n)$ only & with $\beta(n)/N$ & with $\gamma(n)/N^2$\\\hline
		$2$  & $4.6$--$5.2\%$    & $0.91$--$0.94\%$  & $0.20$--$0.23\%$\\
		$5$  & $1.6$--$1.8\%$    & $0.12\%$          & $0.009$--$0.010\%$\\
		$10$ & $0.75$--$0.84\%$  & $0.028\%$         & $0.0010$--$0.0012\%$
	\end{tabular}
	\end{center}
	Thus the constant correction $\beta(n)$ is what makes the expansion usable for histograms
	with few counts per bin, and $\gamma(n)$ brings the relative error below $10^{-3}$ already at
	$N/n=5$.

\subsection{The second-order coefficient $\gamma(n)$}

\begin{lemma}[Disjoint $1/N$ coefficient]\label{lem:r1-disjoint}
	For fixed $n\ge5$,
	\[
		\Cov(|D_1|,|D_3|)
		=
		-\frac1{\pi n}
		-\frac{n+1}{8\pi n}\frac1N
		+O(N^{-2}).
	\]
	Thus the disjoint block coefficient used in the second-order recombination is
	\[
		r_1(n)=-\frac{n+1}{8\pi n}.
	\]
\end{lemma}

\begin{proof}
	Let $M:=X_1+X_2+X_3+X_4\sim\Bin(N,4/n)$.  Conditionally on $M=m$, the four
	involved cells have the uniform four-bin multinomial law, hence
	\[
		\E(|D_1||D_3|)=\E h(M),
	\]
	where $h$ is the disjoint skeleton of Theorem~\ref{thm:h-closed}.  From
	\[
		h(2j)=\frac{j^2}{2^{4j-1}}\binom{2j}{j}^{2},\qquad
		h(2j+1)=\frac{j(2j+1)}{2^{4j}}\binom{2j}{j}^{2},
	\]
	and Stirling's expansion
	\[
		\frac{\binom{2j}{j}^{2}}{16^j}
		=
		\frac1{\pi j}\left(1-\frac1{4j}+\frac1{32j^2}+O(j^{-3})\right),
	\]
	one obtains
	\[
		h(m)=\frac{m}{\pi}-\frac1{2\pi}
		+\frac1m\left(-\frac1{8\pi}+\frac{(-1)^m}{4\pi}\right)
		+O(m^{-2}).
	\]
	For $n\ge5$, the parity term is exponentially small after averaging over
	$M$, because $\E[(-1)^M M^{-1};M\ge1]=O(e^{-cN})$, and the lower tail
	$M<2N/n$ is exponentially small.  Also
	$\E[M^{-1};M\ge1]=(N\cdot4/n)^{-1}+O(N^{-2})$.  Therefore
	\[
		\E h(M)
		=
		\frac{4N}{\pi n}-\frac1{2\pi}-\frac{n}{32\pi N}
		+O(N^{-2}).
	\]
	By Lemma~\ref{lem:imported-mean}, with
	$e_1=-(n-2)/16$ and $e_2=(-3n^2+12n+4)/512$,
	\[
		(\E|D_1|)^2
		=
		\frac{4N}{\pi n}
		-\frac{n-2}{2\pi n}
		+\frac{4(e_1^2+2e_2)}{\pi n}\frac1N
		+O(N^{-2}).
	\]
	The coefficient of $N^{-1}$ in the last display is
	\[
		\frac{4(e_1^2+2e_2)}{\pi n}
		=
		-\frac{n^2-4n-4}{32\pi n}.
	\]
	Subtracting $(\E|D_1|)^2$ from $\E h(M)$ gives the constant
	$-1/(\pi n)$ and the $N^{-1}$ coefficient
	\[
		-\frac{n}{32\pi}
		+
		\frac{n^2-4n-4}{32\pi n}
		=
		-\frac{n+1}{8\pi n}.
	\]
\end{proof}

\begin{lemma}[Residue-class expansion of the adjacent skeleton]\label{lem:atriangle-existence}
	For each $r\in\{0,1,2\}$ there is a constant $b_r$ such that, as
	$\ell\to\infty$ through $\ell\equiv r\pmod3$,
	\[
		\hadj(\ell)=c_L\ell+a_*+\frac{b_r}{\ell}+O(\ell^{-2}).
	\]
	In particular
	\[
		b_r=\lim_{\substack{\ell\to\infty\\ \ell\equiv r\!\!\pmod3}}
		\ell\bigl(\hadj(\ell)-c_L\ell-a_*\bigr)
	\]
	exists, and the $b_r$ are the residue-class part of the $1/\ell$ coefficient
	in the adjacent triangular-lattice expansion.
\end{lemma}

\begin{proof}
	The proof of $a_*$ needs only an $O(\ell^{-1})$ remainder; here we retain one
	more order on each residue class. Start from the exact finite-part representation
	\eqref{eq:rep}, with the partition of unity used for $a_*$, localized at the
	origin saddle $Q_0=(0,0)$, at the twisted saddles
	$Q_1=(2\pi/3,-2\pi/3)$, $Q_2=(-2\pi/3,2\pi/3)$, and away from all saddles.

	At $Q_0$, expand $S(\theta)=-2/\theta^2-\tfrac16+O(\theta^2)$ and put
	$\theta=u/\sqrt\ell$.  Because the pure singular piece carries the outside
	factor $\ell$ after rescaling, an $O(\ell^{-2})$ remainder in the origin contribution
	$\mathcal Q_0(\ell)$
	requires the central expansion to be controlled to order $\ell^{-3}$ \emph{inside}
	the pairing; retaining only through $P_4$, with an $O(\ell^{-5/2})$ remainder,
	would leave $O(\ell^{-3/2})$ after multiplication by $\ell$, which is not
	enough.  We therefore carry the central characteristic function \emph{two}
	orders beyond the $a_*$ expansion:
	\[
		\Phi(u/\sqrt\ell)^\ell=e^{-\frac12u^\top Cu}
		\bigl[1+\ell^{-1/2}P_1+\ell^{-1}P_2+\ell^{-3/2}P_3+\ell^{-2}P_4
		+\ell^{-5/2}P_5\bigr]+R_\ell(u).
	\]
	The analyticity argument of Lemma~\ref{lem:c2-remainder}, with two more retained
	terms, gives the required uniform remainder: equivalently, the same proof is
	used with truncation parameter $s=5$, so no new estimate is involved.  Thus
	\[
		\partial^a R_\ell(u)=O\bigl(\ell^{-3}(1+|u|^M)e^{-\frac18u^\top Cu}\bigr),
		\qquad 0\le a_1,a_2\le2,
	\]
	on the growing central window (localized, as in Lemma~\ref{lem:central-singular}, by the
	smooth window $\psi_\ell$).  The tensor finite-part estimate needs only two
	derivatives in each variable, hence this remainder contributes $O(\ell^{-3})$
	inside the pairing, i.e., $O(\ell^{-2})$ to $\mathcal Q_0(\ell)$ after the outside factor
	$\ell$.  The only possible half-integer powers come from the odd polynomials
	$P_1$, $P_3$, and now $P_5$ (a sum of the odd-total-degree structures $A_7$,
	$A_3A_6$, $A_4A_5$, $A_3^2A_5$, $A_3A_4^2$, $A_3^3A_4$, $A_3^5$); each of their
	finite-part pairings vanishes because the Gaussian, the
	symmetric cutoff, and the tensor kernel are all even.  The same parity argument
	applies to the mixed origin pieces (b), (c) and to the smooth piece (d) of the
	origin split: their kernels ($\fp(u_1^{-2})\,s$, $s\,\fp(u_2^{-2})$, $s\otimes s$,
	with $s$ even) are even in each variable, their outside factors are $1$ and
	$\ell^{-1}$ respectively, so retaining through $P_4$ suffices there, the even
	Taylor expansion $s(u/\sqrt\ell)=s(0)+\tfrac12s''(0)u^2/\ell+O(\ell^{-2})$
	producing only integer powers.  Thus the origin contribution
	has an expansion in integer powers,
	\[
		\mathcal Q_0(\ell)=c_L\ell+a_*+\frac{d_0}{\ell}+O(\ell^{-2}),
	\]
	for a real constant $d_0$.  This $d_0$ is common to all residue classes, since
	the origin saddle has no mod-$3$ phase.

	At the twisted saddles the situation is different but simpler.  The points
	$Q_1,Q_2$ are bounded away from the coordinate axes, so
	$S(\theta_1)S(\theta_2)$ is a smooth amplitude there and no finite-part
	singularity is present.  The explicit local computation was carried out in the
	twisted-saddle paragraph of \S\ref{ssec:adj-dual}: the exact phase identity
	$\Phi(Q_1+\eta)^\ell=\bar\omega^\ell\Phi(\eta)^\ell$ factors out the constant
	phase, the odd terms integrate to zero against the even Gaussian, and one order
	further the even corrections supply a $\tau_2/\ell^2$ term.  Therefore the contribution of
	$Q_1$ is
	\[
		\mathcal Q_1(\ell)=\bar\omega^\ell
		\left(\frac{\tau_1}{\ell}+\frac{\tau_2}{\ell^2}+O(\ell^{-3})\right),
		\qquad
		\tau_1=\frac1{(2\pi)^2}\Bigl(\frac49\Bigr)\frac{2\pi}{\sqrt{\det C}}
		=\frac{2\sqrt3}{9\pi},
	\]
	with $\tau_1$ \emph{real}.  The saddle $Q_2$ is the conjugate one, using
	$\Phi(Q_2+\eta)=\omega\,\Phi(\eta)$, and gives
	\[
		\mathcal Q_2(\ell)=\omega^\ell
		\left(\frac{\tau_1}{\ell}
		+\frac{\overline{\tau_2}}{\ell^2}+O(\ell^{-3})\right).
	\]
	These are the only terms that carry the residue $\ell\bmod3$, and they first
	appear at order $\ell^{-1}$.

	On the off-saddle region we have $|\Phi|\le\varrho<1$. Derivatives of $\Phi^\ell$
	of order at most two in each variable are $O(\ell^4\varrho^{\ell-4})$, and the finite-part
	continuity estimate then makes the off-saddle contribution $O(\ell^4\varrho^\ell)$, hence
	$O(\ell^{-2})$.  Adding the three saddle contributions gives
	\[
		\hadj(\ell)=c_L\ell+a_*+
		\frac{d_0+\tau_1\bar\omega^\ell+\overline{\tau_1}\,\omega^\ell}{\ell}
		+O(\ell^{-2}).
	\]
	On the class $\ell\equiv r\pmod3$ the numerator is constant; the coefficient
	from the statement of the lemma is therefore, with the real
	$\tau_1=2\sqrt3/(9\pi)$,
	\[
		b_r=d_0+\tau_1\bigl(\bar\omega^r+\omega^r\bigr)
		=d_0+\frac{4\sqrt3}{9\pi}\cos\frac{2\pi r}{3}.
	\]
	The common term $d_0$ appears in every $b_r$ because it comes from the origin
	saddle; only the twisted part varies with $r$.  This proves the asserted
	residue-class expansion.  In particular, since $\sum_{r=0}^2\omega^r=0$, the twisted
	amplitudes cancel on averaging,
	\[
		\sum_{r=0}^2 b_r=3d_0,
	\]
	so the residue average is carried by the origin saddle alone.
\end{proof}

	Recall that $\hadj(\ell)$ was defined in \eqref{eq:hadj-def} as the conditional
	expectation of $|D_1||D_2|$ under the uniform three-bin multinomial law.  Equivalently,
	it is the exact adjacent skeleton
	\[
		\hadj(\ell)=3^{-\ell}\!\!\sum_{\substack{x_1+x_2+x_3=\ell\\ x_i\ge0}}\!\!
		\frac{\ell!}{x_1!x_2!x_3!}\,|x_2-x_1|\,|x_3-x_2|.
	\]
	Lemma~\ref{lem:atriangle-existence} extracts from this sequence the three
	residue-class coefficients $b_r$ in the $1/\ell$ term.  The constant below is the
	normalization of their residue average that enters the adjacent covariance after
	the binomial averaging over $L=X_1+X_2+X_3$.

\begin{definition}[Averaged second-order adjacent constant]\label{def:atriangle}
	With the residue limits $b_r$ of Lemma~\ref{lem:atriangle-existence}, set
	\begin{equation}\label{eq:atriangle-def}
		a_\triangle:=\frac1{32}+\frac{\pi}{9}\sum_{r=0}^2 b_r,
	\end{equation}
	the averaged second-order adjacent triangular-lattice constant.
	Numerically, extrapolating the residue-class limits $b_r$ by Richardson's method within
	each class $\ell\bmod3$, from exact values of $\hadj(\ell)$ up to $\ell=2100$, one finds
	\[
		a_\triangle=0.0027388755955740363205\ldots,
	\]
	in agreement with the closed form of Conjecture~\ref{conj:atriangle} to more than $40$
	significant digits; equivalently $\sum_rb_r=3d_0$ with
	$d_0=-0.0272261182924341752845\ldots$.
	The constant is a bookkeeping repackaging of the residue average $\sum_r b_r=3d_0$, chosen
	so that the adjacent coefficient $q_1(n)$ of Lemma~\ref{lem:q1-adjacent} comes out clean:
	the $1/32$ is the leading-$n$ smooth per-edge coefficient imported from the Part~1 mean
	expansion (through $e_1,e_2$ of Lemma~\ref{lem:imported-mean}) and the $\pi/9$ is the
	$L^{-1}$-mixing normalization $\E[L^{-1};L\ge1]\sim n/(3N)$, both read off in the proof of
	Lemma~\ref{lem:q1-adjacent}. The lattice content of $a_\triangle$ is therefore $d_0$ alone.
\end{definition}

\begin{lemma}[Adjacent $1/N$ coefficient]\label{lem:q1-adjacent}
	For fixed $n\ge4$ (the proof uses only $3/n<1$), with $b_r$ and $a_\triangle$ as in
	Lemma~\ref{lem:atriangle-existence} and \eqref{eq:atriangle-def},
	\[
		\Cov(|D_1|,|D_2|)=\frac{2N}{n}B+\Bigl(a_*+\frac{n-2}{2\pi n}\Bigr)
		+\frac{q_1(n)}{N}+O(N^{-2}),
	\]
	with $q_1(n)=\dfrac{a_\triangle n}{\pi}-\dfrac{n+1}{8\pi n}$.
\end{lemma}

\begin{proof}
	Let $L:=X_1+X_2+X_3\sim\Bin(N,3/n)$. Conditionally on $L=\ell$ the three
	involved cells carry the uniform three-bin law, so $\E(|D_1||D_2|)=\E\hadj(L)$.
	The lower tail $L<N/(2n)$ is exponentially small and on its complement
	$L^{-2}=O(N^{-2})$, so the residue expansion of
	Lemma~\ref{lem:atriangle-existence} may be averaged term by term:
	\[
		\E\hadj(L)=c_L\,\E L+a_*+\E\!\left[\frac{b_{L\bmod3}}{L};\,L\ge1\right]+O(N^{-2}).
	\]
	Here $\E L=3N/n$. The residue-weighted average requires an argument, since
	equidistribution of $L\bmod3$ alone does not control the weight $1/L$; we use a
	root-of-unity filter. Write $\theta:=3/n$ and, with $\omega=e^{2\pi i/3}$,
	\[
		b_{L\bmod3}=\sum_{r=0}^2 b_r\,\mathbf 1[L\equiv r]
		=\frac13\sum_{r=0}^2 b_r\sum_{t=0}^{2}\omega^{-rt}\,\omega^{tL},
	\]
	so that
	\[
		\E\!\left[\frac{b_{L\bmod3}}{L};\,L\ge1\right]
		=\frac13\sum_{r=0}^2 b_r\sum_{t=0}^{2}\omega^{-rt}\,
		\E\bigl[\omega^{tL}L^{-1};\,L\ge1\bigr].
	\]
	The twisted terms $t=1,2$ are exponentially small: by the Laplace
	representation $L^{-1}=\int_0^\infty e^{-Lu}\,du$ ($L\ge1$),
	\[
		\E\bigl[\omega^{tL}L^{-1};\,L\ge1\bigr]
		=\int_0^\infty\Bigl[\bigl(1-\theta+\theta\omega^te^{-u}\bigr)^N
		-(1-\theta)^N\Bigr]du,
	\]
	and with $\varrho_t:=\max_{0\le y\le1}\bigl|1-\theta+\theta\omega^ty\bigr|<1$
	(strict, since $\omega^t\ne1$: the modulus is $<1$ for $0<y\le1$ by the strict
	triangle inequality and equals $1-\theta<1$ at $y=0$), the mean-value theorem in
	$y=e^{-u}$ bounds the integrand by $N\theta e^{-u}\varrho_t^{\,N-1}$, whence the
	integral is $O(N\varrho_t^N)=O(e^{-cN})$. The untwisted term $t=0$ carries the
	whole contribution: by the same Taylor-with-remainder argument as in the proof of
	\cite[Lemma~5.1]{banic_elezovic_p1} (expansion of $x\mapsto x^{-1}$ about $\mu_L=3N/n$
	on $\{L\ge\mu_L/2\}$, exponentially small tails),
	\[
		\E[L^{-1};L\ge1]=\frac{n}{3N}+O(N^{-2}),
	\]
	which is sufficient at this order. Combining,
	\[
		\E\!\left[\frac{b_{L\bmod3}}{L};\,L\ge1\right]
		=\frac13\Bigl(\sum_{r=0}^2 b_r\Bigr)\E[L^{-1};L\ge1]+O(e^{-cN})
		=\frac{n}{9N}\sum_{r=0}^2 b_r+O(N^{-2}),
	\]
	hence $\E|D_1||D_2|=\frac{3c_LN}{n}+a_*+\frac{n}{9N}\sum_{r}b_r+O(N^{-2})$. On the
	other hand, the imported mean expansion (Lemma~\ref{lem:imported-mean}) gives
	\[
		(\E|D_1|)^2=\frac{4N}{\pi n}-\frac{n-2}{2\pi n}
		-\frac{n^2-4n-4}{32\pi n}\frac1N+O(N^{-2}).
	\]
	Subtracting, and using $3c_L-\tfrac4\pi=2B$, gives the leading $2NB/n$, the
	constant $a_*+\tfrac{n-2}{2\pi n}$, and the $N^{-1}$ coefficient
	$\frac{n}{9}\sum_r b_r+\frac{n^2-4n-4}{32\pi n}$. By \eqref{eq:atriangle-def},
	$\frac{n}{9}\sum_r b_r=\frac{a_\triangle n}{\pi}-\frac{n}{32\pi}$, so this
	coefficient equals
	\[
		\frac{a_\triangle n}{\pi}-\frac{n}{32\pi}+\frac{n^2-4n-4}{32\pi n}
		=\frac{a_\triangle n}{\pi}-\frac{n+1}{8\pi n}=q_1(n).
	\]
\end{proof}

\begin{theorem}[Second-order coefficient $\gamma(n)$]\label{thm:gamma-linear}
	Expanding each of the three quantities in \eqref{eq:Var-d-3comp}
	one order beyond $\beta$, write their $1/N$ coefficients as $p_1,q_1,r_1$.
	Then for $n\ge5$ the
	$N^{-2}$ coefficient of
	\[
		\sigma^2(n,N)=\alpha(n)-\frac{\beta(n)}{N}+\frac{\gamma(n)}{N^2}+o(N^{-2})
	\]
	is the same recombination as the variance itself,
	\[
		\gamma(n)=(n-1)\,p_1(n)+2(n-2)\,q_1(n)+(n-2)(n-3)\,r_1(n),
	\]
	with
	\[
		p_1(n)=\frac{n^2-4n-4}{32\pi n},\qquad
		q_1(n)=\frac{a_\triangle n}{\pi}-\frac{n+1}{8\pi n},
	\]
	\[
		r_1(n)=-\frac{n+1}{8\pi n}.
	\]
	The per-edge coefficient $p_1$ is imported from Part~1 (Lemma~\ref{lem:imported-mean}),
	the disjoint coefficient $r_1$ is Lemma~\ref{lem:r1-disjoint}, and the adjacent
	coefficient $q_1$ is derived from the binomial mixing of the triangular-lattice
	expansion in Lemma~\ref{lem:q1-adjacent} (the restriction to $n\ge5$ comes from the
	disjoint block, Lemma~\ref{lem:r1-disjoint}; the adjacent coefficient is valid for $n\ge4$).
	In particular $\gamma(n)=\gamma_2 n^2+O(n)$ with
	\[
		\gamma_2=\frac{64a_\triangle-3}{32\pi}
		=-0.0280979294715\ldots .
	\]
\end{theorem}

\begin{proof}
	Insert the three block expansions into the exact decomposition
	\eqref{eq:Var-d-3comp}. The coefficient of $1/N$ in $\Var(d)$ is therefore
	\[
		(n-1)p_1+2(n-2)q_1+(n-2)(n-3)r_1.
	\]
	Here the multiplicities are the same as in the original variance decomposition:
	there are $n-1$ edge variances, $2(n-2)$ ordered adjacent-pair contributions after
	the covariance factor is included, and $(n-2)(n-3)$ ordered disjoint-pair
	contributions.  No new combinatorics enters at this order; only the coefficients
	inside the three already isolated blocks have changed.
	After division by $N$, this becomes the coefficient of $N^{-2}$ in
	$\sigma^2(n,N)=\Var(d)/N$, proving the recombination formula. For its leading
	large-$n$ term, use
	\[
		p_1(n)=\frac{n}{32\pi}+O(1),\qquad
		q_1(n)=\frac{a_\triangle n}{\pi}+O(1),\qquad
		r_1(n)=-\frac1{8\pi}+O(n^{-1}).
	\]
	The coefficient of $n^2$ is consequently
	\[
		\gamma_2=\frac1{32\pi}+\frac{2a_\triangle}{\pi}-\frac1{8\pi}
		=\frac{64a_\triangle-3}{32\pi}.
	\]
	Using the numerical evaluation of $a_\triangle$ quoted in
	Definition~\ref{def:atriangle} gives the displayed numerical value of
	$\gamma_2$.
\end{proof}

	\medskip\noindent The constant $\gamma_2$ will reappear in Part~5~\cite{banic_elezovic_p5}
	as the leading coefficient of the second-order correction of the cyclic statistic.

\begin{conjecture}[Closed form of the adjacent second-order constant]
	\label{conj:atriangle}
	The adjacent triangular-lattice constant in \eqref{eq:atriangle-def} is
	\[
		a_\triangle=\frac1{32}-\frac{4\sqrt3}{243}.
	\]
	Equivalently,
	\[
		\gamma_2
		=
		-\frac1{32\pi}-\frac{8\sqrt3}{243\pi},
	\]
	and, in terms of the origin-saddle constant of Lemma~\ref{lem:atriangle-existence},
	$d_0=\tfrac13\sum_rb_r=-\tfrac{4\sqrt3}{81\pi}$, i.e., $b_0=\tfrac{32\sqrt3}{81\pi}$ and
	$b_1=b_2=-\tfrac{22\sqrt3}{81\pi}$; the residue-class coefficients are then
	$b_r=d_0+\tfrac{4\sqrt3}{9\pi}\cos\tfrac{2\pi r}{3}$.
\end{conjecture}

	\medskip\noindent\emph{Consequences, conditional on Conjecture~\ref{conj:atriangle}.} For $n\ge5$,
	\[
		\gamma(n)=-\frac{n(n+1)}{32\pi}-\frac{8\sqrt3\,n(n-2)}{243\pi}+\frac{n-1}{8\pi n},
	\]
	and in the degenerate cases of \S\ref{sec:small-n}, where the residue (resp.\ parity)
	oscillation of the adjacent (resp.\ disjoint) block enters at this order,
	\[
		\gamma(3;N)=-\frac{7}{24\pi}-\frac{8\sqrt3}{81\pi}+\frac{8\sqrt3}{9\pi}\cos\frac{2\pi N}{3},
		\qquad
		\gamma(4;N)=-\frac{17}{32\pi}-\frac{64\sqrt3}{243\pi}+\frac{(-1)^N}{2\pi}.
	\]
	Here $\gamma(3;N)=2p_1(3)+2\bigl[b_{N\bmod3}-\tfrac{7}{96\pi}\bigr]$, the adjacent block being
	unaveraged at $n=3$ ($L=N$; Lemma~\ref{lem:atriangle-existence} with
	Lemma~\ref{lem:imported-mean}), and $\gamma(4;N)=3p_1(4)+4q_1(4)+2r_1(4;N)$ with
	$r_1(4;N)=-\tfrac{5}{32\pi}+\tfrac{(-1)^N}{4\pi}$, the disjoint block being unaveraged at
	$n=4$ ($M=N$; \S\ref{ssec:three-four-bins}); the conjecture enters through $d_0$ and
	$a_\triangle$ only. Each agrees with Richardson extrapolation of the exact variance to at
	least $16$ digits ($5\le n\le20$), resp.\ $24$ digits on every residue class ($n=3,4$);
	$\gamma(3;N)$ is positive on the class $N\equiv0\pmod3$, where $b_0=\tfrac{32\sqrt3}{81\pi}$
	dominates.

\section{Variance for small numbers of bins ($n\le4$)}
	\label{sec:small-n}

	Having assembled the expansion for $n\ge3$ (and $\gamma(n)$ for $n\ge5$), we record what the decomposition
	\eqref{eq:Var-d-3comp} yields for a small number of bins, where the variance is
	exact or elementary and the degenerate structure is transparent. The three
	contributions switch on one at a time: at $n=2$ only the per-edge variance
	survives, at $n=3$ the adjacent covariance appears while the disjoint term is
	still absent, and $n=4$ is the smallest case in which all three are present.
	\begin{center}
	\begin{tabular}{c|c|c|c}
		$n$ & per-edge $(n{-}1)$ & adjacent $2(n{-}2)$ & disjoint $(n{-}2)(n{-}3)$\\\hline
		$2$ & $1$ & $0$ & $0$\\
		$3$ & $2$ & $2$ & $0$\\
		$4$ & $3$ & $4$ & $2$
	\end{tabular}
	\end{center}
	These boundary cases share one feature that organizes the section. Each block's
	\emph{minimal} case --- per-edge at $n=2$, adjacent at $n=3$, disjoint at $n=4$ ---
	is the one in which the relevant conditioning total equals $N$ deterministically,
	with no remainder bin to mix over. The block's underlying lattice term is then
	\emph{unaveraged}, and a parity or residue oscillation survives: $(-1)^N$ on $\Z$ for the
	per-edge block ($n=2$), already at constant order, and a residue mod $3$ on the triangular
	lattice for the adjacent block ($n=3$) and parity mod $2$ on $\Z^2$ for the disjoint block
	($n=4$), both one order lower, in the $N^{-2}$ coefficient of $\Var(d)/N$. Above the
	minimal value the conditioning total is genuinely random and
	these oscillations are smoothed away.

\subsection{The two-bin case: a folded binomial}

	For $n=2$ there is a single edge and $K=X_1+X_2=N$ is \emph{deterministic}; this
	is the most degenerate case, in which the mixing variable of
	Theorem~\ref{thm:Var-D1} collapses to a point. By \eqref{eq:Var-d-3comp} only the
	per-edge term survives, and with $|D_1|=|N-2X_1|$, $X_1\sim\Bin(N,\tfrac12)$, the
	variance is exact.

\begin{theorem}[Exact variance at $n=2$]\label{thm:Var-n2}
	For the uniform two-bin law,
	\begin{equation}\label{eq:Var-n2}
		\Var(d)\big|_{n=2}=N-g(N)^2,\qquad
		g(k):=\frac{k}{2^{\,k-1}}\binom{k-1}{\lfloor(k-1)/2\rfloor},
	\end{equation}
	where $g(N)=\E|D_1|$. Asymptotically
	\[
		\Var(d)\big|_{n=2}=\Bigl(1-\frac2\pi\Bigr)N+\frac{(-1)^N}{\pi}+O(N^{-1}),
	\]
	so $\alpha(2)=1-2/\pi=A$, and the constant-order term oscillates with the parity
	of $N$: writing $\beta(2;N)$ for this $N$-dependent constant, $\beta(2;N)=-(-1)^N/\pi$.
\end{theorem}

\begin{proof}
	Here $\E D_1^2=\Var(D_1)=2N/n=N$ (Lemma~\ref{lem:cov-Dj} at $n=2$), and since
	$|D_1|^2=D_1^2$,
	\[
		\Var(d)\big|_{n=2}=\Var(|D_1|)=N-(\E|D_1|)^2 .
	\]
	With $X_1\sim\Bin(N,\tfrac12)$ the variable $D_1=N-2X_1$ is symmetric about $0$,
	so $\E|D_1|=2\,\E D_1^{+}$. Writing
	$(N-2j)\binom Nj=N\big[\binom{N-1}{j}-\binom{N-1}{j-1}\big]$, the sum
	\[
		\E|D_1|=2^{1-N}\!\!\sum_{0\le j<N/2}(N-2j)\binom Nj
	\]
	telescopes to $g(N)$ in \eqref{eq:Var-n2}. The central-binomial expansion of
	$g(N)$ gives
	\[
		g(N)^2=\frac{2N}\pi-\frac{(-1)^N}\pi+O(N^{-1}),
	\]
	whence the stated expansion, and
	\[
		N\alpha(2)-\Var(d)\big|_{n=2}=-\frac{(-1)^N}{\pi}+O(N^{-1}),\qquad\text{i.e.,}\quad
		\beta(2;N)=-\frac{(-1)^N}{\pi}.
	\]
\end{proof}

	The closed form $g(N)=\E|D_1|$ is twice the mean absolute deviation of the binomial
	law $\Bin(N,\tfrac12)$; it is classical since de~Moivre (1730; see Diaconis and
	Zabell~\cite{diaconis_zabell1991}) and was recorded with its parity effect in
	\cite[Theorem~1 and Equation~(20)]{banic_elezovic_tvor}. Its complete asymptotic expansion is developed
	in~\cite{elezovic_mad}; the absolute moments of the binomial folded at its mean, of
	which $|D_1|$ at $n=2$ is the symmetric case, are the subject
	of~\cite{elezovic_folded_mean}.

	The exact closed form and its asymptotics are illustrated below; the parity
	oscillation of $\beta(2;N)$ is visible in the last column.
	\begin{center}
	\begin{tabular}{r|c|c|c}
		$N$ & $\Var(d)|_{n=2}$ (exact) & numeric & $(1-\tfrac2\pi)N+\tfrac{(-1)^N}\pi$\\\hline
		$2$  & $1$        & $1.000000$  & $1.045$\\
		$3$  & $3/4$      & $0.750000$  & $0.772$\\
		$4$  & $7/4$      & $1.750000$  & $1.772$\\
		$5$  & $95/64$    & $1.484375$  & $1.499$\\
		$6$  & $159/64$   & $2.484375$  & $2.499$\\
		$40$ & ---        & $14.851505$ & $14.854$\\
		$80$ & ---        & $29.387727$ & $29.389$
	\end{tabular}
	\end{center}

\subsection{Three and four bins}
	\label{ssec:three-four-bins}

	For $n=3$ the disjoint multiplicity $(n-2)(n-3)$ vanishes, so
	\eqref{eq:Var-d-3comp} reduces to
	\[
		\Var(d)\big|_{n=3}=2\,\Var(|D_1|)+2\,\Cov(|D_1|,|D_2|),
	\]
	with no disjoint term: the three bins exhaust the sample, the two edges share only
	the middle bin, and the conditional adjacent problem is the pure triangular-lattice
	trinomial $\Mult(N;\tfrac13,\tfrac13,\tfrac13)$ --- the very geometry that defines
	the adjacent constants $a_*$ and $a_\triangle$ of \eqref{eq:atriangle-def}.

	The four-bin case is the first to carry a disjoint pair, $\{1,2\}$ and $\{3,4\}$;
	since these exhaust all bins, $D_1+D_3\equiv N\pmod 2$. At leading order the
	disjoint increments are uncorrelated (Lemma~\ref{lem:cov-Dj}) and contribute
	nothing, but at constant order the disjoint covariance survives,
	$\Cov(|D_1|,|D_3|)\to-1/(4\pi)$.  This is the endpoint case $n=4$ of
	Theorem~\ref{thm:Cov-disj}, covered there by the exact skeleton; the
	characteristic-function derivation of \S\ref{ssec:Cov-disj-asymp-proof} does not apply,
	because for $n=4$ there is no stay-put increment $(0,0)$ and the Fourier bound used there
	degenerates at the second lattice point $(\pi,\pi)$.  We record the skeleton computation
	explicitly, since its parity term feeds $\gamma(4;N)$: the four-bin total is
	deterministic, so
	\[
		\E(|D_1||D_3|)=h(N).
	\]
	Stirling's expansion in the two parity formulas for $h$ gives
	\[
		h(N)=\frac N\pi-\frac1{2\pi}+O(N^{-1}),
	\]
	whereas Lemma~\ref{lem:imported-mean} at $n=4$ gives
	\[
		(\E|D_1|)^2=\frac N\pi-\frac1{4\pi}+O(N^{-1}).
	\]
	Hence
	\[
		\Cov(|D_1|,|D_3|)
		=
		h(N)-(\E|D_1|)^2
		=
		-\frac1{4\pi}+O(N^{-1}).
	\]
	One order further, the expansion
	$h(N)=\tfrac N\pi-\tfrac1{2\pi}+\tfrac1N\bigl(-\tfrac1{8\pi}+\tfrac{(-1)^N}{4\pi}\bigr)+O(N^{-2})$
	from the proof of Lemma~\ref{lem:r1-disjoint} and
	$(\E|D_1|)^2=\tfrac N\pi-\tfrac1{4\pi}+\tfrac1{32\pi N}+O(N^{-2})$ (Lemma~\ref{lem:imported-mean}
	at $n=4$) give the $1/N$ coefficient
	\[
		r_1(4;N)=-\frac{5}{32\pi}+\frac{(-1)^N}{4\pi},
	\]
	the parity of $h(N)$ surfacing one order deeper than the constant, in $\gamma(4;N)$
	(\S\ref{ssec:beta-small-n}).

	In all three cases the leading variance is the specialization of \eqref{eq:z2}
	with $A,B$ from \eqref{eq:A-B-def},
	\begin{align*}
		\alpha(2)&=A=1-\tfrac2\pi,\qquad
		\alpha(3)=\tfrac43(A+B)=\tfrac{14}9+\tfrac{4(\sqrt3-4)}{3\pi},\\
		\alpha(4)&=\tfrac12(3A+4B)=\tfrac{11}6+\tfrac{2\sqrt3-7}\pi ,
	\end{align*}
	each a finite, exactly computable variance for every $N$; for instance
	$\Var(d)|_{n=3}=\tfrac29$ and $\Var(d)|_{n=4}=\tfrac14$ at $N=1$, and
	$\alpha(n)$ increases to $\sigma_\infty^2=2A+4B$ as $n\to\infty$
	(Corollary~\ref{cor:sigma-infty}).

\subsection{The constant order across the degenerate cases}
	\label{ssec:beta-small-n}

	The expression \eqref{eq:beta-closed} holds for every $n\ge3$ by
	Theorem~\ref{thm:beta-closed}; here we make the degenerate cases explicit, where no
	new constant is needed.  At $n=3$ the disjoint coefficient $(n-2)(n-3)$ is zero, and the
	adjacent conditional model \emph{is} the trinomial $\Mult(N;\tfrac13,\tfrac13,\tfrac13)$
	that defines $a_*$, so the adjacent block requires no separate evaluation.  At $n=4$ the
	endpoint computation above gives the constant disjoint covariance $-1/(4\pi)$ as the
	exact skeleton value $h(N)-(\E|D_1|)^2$, which coincides with the $n=4$ substitution
	into the limit $-1/(\pi n)$. Thus the same formula gives $\beta(3)$ and $\beta(4)$,
	while at $n=2$ Theorem~\ref{thm:Var-n2} gives an oscillating constant term:
	\[
		\beta(2;N)=-\frac{(-1)^N}{\pi},\qquad
		\beta(3)=\frac{4\sqrt3}{9\pi}-\frac{2}{3\pi},\qquad
		\beta(4)=\frac{8\sqrt3}{9\pi}-\frac{5}{4\pi},
	\]
	with $\beta(3)\approx0.03283$ and $\beta(4)\approx0.09218$. Only $n=2$ lacks a
	single constant $\beta$: its per-edge degeneracy puts the unaveraged $(-1)^N$
	already at the constant order. For $n=3,4$ the analogous oscillation is pushed one
	order deeper, into $\gamma$ --- the triangular residues $b_r$ of
	Lemma~\ref{lem:atriangle-existence} are mod $3$ at $n=3$, and the disjoint parity is mod $2$
	at $n=4$, so there is no single $N$-independent $\gamma(3)$ or $\gamma(4)$; their class-wise
	values are given after Conjecture~\ref{conj:atriangle}. The convergence of $\Var(d)/N$ to
	$\alpha(n)$ is shown below: alternating about $\alpha(2)$ with the parity of $N$ at $n=2$
	(from above for even $N$, from below for odd $N$, as the row $N=41$ shows), and from below
	at $n=3,4$ for $N\ge10$ (at $n=3$ the class $N\equiv0\pmod3$ lies above $\alpha(3)$ for
	$N=3,6,9$, as the positive $\gamma(3;N)$ of \S\ref{sec:synthesis} predicts).
	\begin{center}
	\begin{tabular}{r|c|c|c}
		$N$ & $\Var(d)/N,\ n{=}2$ & $n{=}3$ & $n{=}4$\\\hline
		$20$        & $0.379092$ & $0.590350$ & $0.702803$\\
		$40$        & $0.371288$ & $0.591937$ & $0.705418$\\
		$41$        & $0.355570$ & $0.591970$ & $0.705287$\\
		$80$        & $0.367347$ & $0.592536$ & $0.706645$\\\hline
		$\alpha(n)$ & $0.363380$ & $0.593008$ & $0.707822$
	\end{tabular}
	\end{center}

\section*{Conclusion}

	This paper isolates the variance of the linear discrete total variation statistic
	into three local mechanisms: a single-edge variance, a disjoint-pair covariance,
	and an adjacent-pair covariance.  The first is elementary given the mean expansion of
	Part~1; the second reduces to the closed skeleton $h(m)$; the third is the only genuinely
	new local analysis, because adjacent edges live on a triangular sublattice and the
	absolute-value test function has two kink lines. The disjoint mechanism is the probabilistic
	point of the result: two edges sharing no bin are asymptotically independent, yet their
	absolute values keep a covariance $\Cov(|D_1|,|D_3|)\to-1/(\pi n)$ of relative order $1/N$
	--- the conditioning-on-the-total effect of the fixed sample size, proportional to the exact
	joint fourth cumulant $\kappa_{2,2}=-4N/n^2$ (Proposition~\ref{prop:kappa22}) and proved here
	on the lattice together with its $1/N$ correction.

	The main closed result is the two-term fixed-$n$ expansion
	\[
		\sigma^2(n,N)
		=
		\alpha(n)-\frac{\beta(n)}{N}+O(N^{-2})\qquad(n\ge3),
	\]
	where $\alpha(n)$ is the Gaussian leading coefficient and
	\[
		\beta(n)
		=
		-\frac{(n-2)(n+1)}{2\pi n}
		+
		\frac{4\sqrt3\,(n-2)}{9\pi}.
	\]
	The sign and size of this correction come from a nontrivial cancellation: the
	per-edge and disjoint contributions are rational multiples of $1/\pi$, while the
	adjacent term contributes the triangular-lattice constant
	$a_*=-2\sqrt3/(9\pi)$.  In the proof of $a_*$, one quarter of the constant is a
	smooth Edgeworth effect and three quarters is the lattice corner correction.

	One order further, the same decomposition yields
	\[
		\sigma^2(n,N)
		=
		\alpha(n)-\frac{\beta(n)}{N}
		+\frac{\gamma(n)}{N^2}+o(N^{-2})
		\qquad (n\ge5),
	\]
	with $\gamma(n)$ expressed through the explicitly defined constant
	$a_\triangle$.  The numerical evidence strongly identifies
	$a_\triangle=1/32-4\sqrt3/243$, but the algebraic evaluation of this
	second-order triangular-lattice constant remains the single open closed-form input.
	Thus the remaining conjectural point is not the variance decomposition itself, but
	only the final simplification of one explicitly defined local constant.

	This is the second part of a series on exact and asymptotic moments of discrete total
	variation under multinomial sampling. Part~1~\cite{banic_elezovic_p1} treats the expectation
	under the uniform law; the present paper treats the variance under the uniform law;
	Parts~3 and~4~\cite{banic_elezovic_p3,banic_elezovic_p4} will carry the expectation and
	variance to a general reference law, and Part~5~\cite{banic_elezovic_p5} will close the
	histogram into a ring and treat the cyclic statistic $d^{c}$. The exact identity of
	Lemma~\ref{lem:linear-cyclic},
	\[
		\Var(d)=\Bigl(1-\frac2n\Bigr)\Var(d^{c})+\Var(|D_1|),
	\]
	transfers the constants computed here to the cyclic setting (see the paragraph \emph{Cyclic
	consequence} after Theorem~\ref{thm:beta-closed}) and is the natural framework for the joint large-$n$,
	large-$N$ regime, which is not analyzed in the present fixed-$n$ paper.

\section*{Acknowledgments}
	The authors used an AI assistant (Anthropic's Claude) for literature search,
	symbolic and numerical verification, and editorial support; all mathematical
	results were checked independently by the authors.

	\medskip\noindent\textbf{Code and data.} All exact values quoted in this paper were
	computed from the skeleton form of \eqref{eq:Var-d-3comp} in exact rational or
	multiprecision arithmetic. The scripts and the numerical data behind the tables and the
	extrapolated constants will be made publicly available together with a monograph in
	preparation; until then they are available from the authors on request.


\end{document}